\documentclass{article}
\usepackage[UKenglish]{babel}
\usepackage{amssymb,amsthm}
\usepackage{geometry}
\usepackage{graphicx}
\usepackage[latin1]{inputenc}
\usepackage{microtype}

\newtheorem{theorem}{Theorem}[section]
\newtheorem{lemma}{Lemma}[section] 
\newtheorem{corollary}{Corollary}[section]

\newtheorem{remark}{Remark}[section]
\newtheorem{definition}{Definition}[section]

\newcommand{\orig}[1]{\mathsf{#1}}

\def\noteps{\mathrel{\!\not\mathrel{\,\overline{\varepsilon}\!}\,}}

\newcommand{\nat}{\mbox{\cal N}}

\newcommand{\mtt}{\mbox{{\bf mTT}}}
\newcommand{\mtts}{\mbox{{\bf mTT}$^s$}}

\newcommand{\emtt}{\mbox{{\bf emTT}}}
\newcommand{\mf}{\mbox{{\bf MF}}}
\newcommand{\tar}{\mbox{$\widehat{ID_1}$}}

\newcommand{\Set}[1]{\mathbf{Set}(#1)}
\newcommand{\cx}{\left[\underline{x}\right]}
\newcommand{\cxy}{\left[\underline{x},y\right]}
\newcommand{\cxyz}{\left[\underline{x},y,z\right]}
\newcommand{\cxyzu}{\left[\underline{x},y,z,u\right]}
\newcommand{\pair}[2]{\langle #1,#2 \rangle}

\newcommand{\rec}{\mathbf{rec}}
\newcommand{\lrec}{\mathbf{listrec}}
\newcommand{\ite}[3]{\{\mathbf{ite}\}(#1,#2,#3)}

\begin{document}
\title{An extensional Kleene realizability semantics for the Minimalist Foundation} 
\author{Maria Emilia Maietti and Samuele Maschio}


\maketitle
\date{ }
\begin{abstract}
We build a Kleene realizability semantics for the two-level Minimalist Foundation {\bf MF}, ideated by Maietti and Sambin in 2005 and completed by Maietti in 2009.   Thanks to this semantics we prove that both levels of {\bf MF}
are consistent  with the (Extended) formal Church Thesis {\bf CT}.

Since {\bf MF} consists of two levels, an intensional one, called \mtt,
 and an extensional one, called \emtt, linked by an interpretation, it is enough to build a realizability  semantics for the intensional level
\mtt\ to get one for the extensional one \emtt, too.
Moreover, both levels consists of type theories based on versions of Martin-L{\"o}f's type theory.

Our realizability semantics for \mtt\ is a modification of the realizability semantics by Beeson in 1985  for extensional
first order Martin-L{\"o}f's type theory with one universe. So
it is formalized in Feferman's classical
arithmetic
theory of inductive definitions, called \tar.
It is called {\it extensional} Kleene realizability semantics   since it validates extensional
equality of type-theoretic functions {\bf extFun}, as in Beeson's one.

The main modification we perform on Beeson's semantics is to interpret
propositions, which are defined primitively in {\bf MF}, in a proof-irrelevant way.
As a consequence, we gain the validity of {\bf CT}. Recalling that
{\bf extFun}+ {\bf CT}+ {\bf AC} are inconsistent over arithmetics with finite types,
we conclude that our semantics does not validate
the Axiom of Choice {\bf AC} on generic types. 
On the contrary,  Beeson's semantics
does validate {\bf AC},
being this a theorem of Martin-L{\"o}f's theory, but it does not validate
{\bf CT}.  The semantics we present here appears to be the best Kleene realizability semantics for the extensional
level \emtt. Indeed Beeson's semantics is not an option for
\emtt\  since {\bf AC} on generic sets added to it entails the excluded middle.
\end{abstract}
\section{Introduction}
A foundation for mathematics should be called constructive only if the mathematics arising from it could be considered genuinely computable.
One way to show this is to produce a realizability model of the foundation where
arbitrary sets are interpreted as data types and 
functions between them are interpreted as programs.
 A key example is Kleene's realizability model for first-order
Intuitionist Arithmetics validating the formal Church Thesis.

Here we will show how to build a realizability model for the  \emph{Minimalist Foundation}, for short {\bf MF}, ideated by Maietti and Sambin in \cite{mtt} and then completed by Maietti in  \cite{m09}, where it is explicit how to extract programs
from its proofs. In particular we show that {\bf MF} is consistent with 
the (Extended) Church Thesis, for short {\bf CT}. This result is part of a project to know to what extent {\bf MF} enjoys the same properties as Heyting arithmetics.

The Minimalist Foundation is intended to constitute a common core among the most relevant constructive and classical foundations. One of its novelties is that it consists of two levels: an intensional level, called \mtt, which should make evident the constructive contents of mathematical proofs in terms of programs,  and an extensional level, called \emtt, formulated in a language close as much as possible to that of ordinary mathematics.  Both  intensional and  extensional levels of  {\bf MF} consist of type systems based on versions of Martin-Lof's type theory with the addition of a primitive notion of propositions: the intensional one is based on \cite{PMTT}
and the extensional one on \cite{ML84}. Actually \mtt\ 
can be considered a {\it predicative} version of Coquand's Calculus of Constructions \cite{TC90}.

To build a realizability model for the two-level Minimalist Foundation, it is enough to build it  for its intensional level \mtt.
Indeed an interpretation for the extensional level \emtt\ 
can be then obtained from an interpretation of \mtt\ by composing this with the interpretation of \emtt\ in a suitable setoid model of \mtt\ as in \cite{m09} and analyzed in \cite{qu12}.
Moreover, since the interpretation of  {\bf CT} from the extensional level to the intensional one is equivalent to {\bf CT} itself according to \cite{m09},
a model showing consistency of \mtt\ with {\bf CT} can be turned into a model
showing consistency of \emtt\ with {\bf CT}.

Here, we build a  realizability model for \mtt+ {\bf CT}\  
 by
suitably modifying
Beeson's realizability semantics \cite{beeson} for the extensional version of first order Martin-L{\"o}f's type theory with one universe~\cite{ML84}.
So, as Beeson's semantics  our model is based on Kleene realizability semantics
of intuitionistic arithmetics and it is formalized in Feferman's classical
arithmetic
theory of inductive definitions, called \tar\ (\cite{Fef}). The theory \tar\ is formulated in the language of second-order arithmetics and it consists of PA (Peano Arithmetic) plus the existence of some (not necessary the least) fix point for positive parameter-free arithmetical operators.

We call our Kleene realizability semantics  {\it extensional} since it validates extensional
equality of type-theoretic functions {\bf extFun}, as Beeson's one.

The main modification we perform to Beeson's semantics is to interpret
propositions, which are defined primitively in {\bf MF}, in a proof-irrelevant way.
More in detail we interpret \mtt-sets  as Beeson interpreted Martin-L{\"o}f's sets, propositions are interpreted as
 trivial quotients of Kleene realizability interpretation of intuitionistic connectives, and the universe of \mtt-small propositions is interpreted as a suitable quotient of some fix point including all the codes
of small propositions by using the technique  Beeson adopted to interpret  Martin-L{\"o}f's universe.

As a consequence in our model we gain the validity of {\bf CT} but
we loose the validity of the full Axiom of Choice {\bf AC}.  Instead
in Beeson's semantics,  {\bf AC} is valid,  being this a theorem of  Martin-L{\"o}f's theory, but {\bf CT} is not.
All these results follow from  the well known fact
that   {\bf extFun}+ {\bf CT}+ {\bf AC} over arithmetics with finite types are inconsistent. Therefore in the presence of  {\bf extFun} as in our \emtt,
either one validates {\bf CT} as we do here, or {\bf AC}
as in Beeson's semantics. Recalling that the addition of  {\bf AC} on generic sets in \emtt\ entails
the excluded middle, Beeson's semantics is not an option for \emtt.
Therefore the semantics we present here appears to be the best
Kleene realizability semantics for the extensional level \emtt.

Actually a consistency proof for \emtt\ with {\bf CT} could also be  obtained by interpreting this theory in the internal
theory of Hyland's effective topos~\cite{eff}.
But here  we have obtained a proof in a {\it predicative theory}, whilst classical,
 as \tar. As a future work we intend to generalize the notion of effective
topos to that of a predicative effective topos in order to extract the categorical structure behind  our realizability interpretation.

\section{The Minimalist Foundation} 

In \cite{m09} a two-level formal system, called {\bf  Minimalist Foundation},
for short {\bf MF},
is completed following the design advocated in \cite{mtt}.
The  two levels of {\bf MF} are both given by a type theory \`a la  Martin-L{\"o}f: the intensional level, called \mtt, is an intensional
type theory including aspects of  Martin-L{\"o}f's one
in \cite{PMTT} (and extending the set-theoretic version in \cite{mtt} with collections), and its extensional level, called \emtt,  is an extensional type theory including
aspects of  extensional Martin-L{\"o}f's one in \cite{ML84}.
Then a quotient model of setoids \`a la Bishop~\cite{Bishop,disttheshof,ven,notepal} over the intensional level is  used in \cite{m09} to interpret
the extensional level in the intensional one. A categorical study of this quotient model has been carried on 
 in~\cite{qu12,elqu,uxc}
and related to the construction of Hyland's effective topos~\cite{eff,tripos}.

{\bf MF} was ideated in \cite{mtt} to be constructive and  minimalist, that is compatible with (or interpretable in)  most relevant
constructive and classical foundations for mathematics in the literature.
According to these desiderata, {\bf MF} has the following peculiar features (for a more extensive description see also  \cite{whyp}):
\begin{itemize}

\item {\bf MF has two types of entities: sets and collections.}
This is a consequence of the fact that
a minimalist foundation compatible with most of constructive theories in the literature, among which, for example,
 Martin-L{\"o}f's one in \cite{PMTT}, should be certainly predicative and
based on intuitionistic predicate logic, including at least
the axioms of Heyting arithmetic. For instance it could be a many-sorted logic, such as Heyting arithmetic of finite types~\cite{DT88},  where  sorts, that we call {\it types}, include 
the basic sets we need to represent our mathematical entities.
But in order to represent topology in an intuitionistic and predicative way, then {\bf MF}  needs to be equipped with two kinds of entities: sets and collections.
Indeed, the {\it power of a non-empty set}, namely the discrete topology over a non-empty set, fails to be a set in a predicative foundation, and it is only a {\it collection}.

\item
{\bf  MF has two types of propositions.}  This is a consequence of the previous
characteristic. Indeed the presence of sets and collections, where the latter include the representation of power-collections of subsets,
yields to distinguish
two types of propositions to remain predicative: those closed under quantifications on sets, called
{\it small propositions}  in \cite{m09}, from those
closed under any kind of quantification, called  {\it propositions} in \cite{m09}.
This distinction is crucial in the definition of ``subset of a set'' we adopt
in \mf: a subset of
a set $A$ is indeed an equivalence class of small propositional functions from $A$.


\item
{\bf MF has two types of functions.}
As in  Coquand's Calculus of Constructions~\cite{TC90}, or Feferman's predicative theories \cite{Fef}, in {\bf MF} we distinguish the notion of functional relation
from that of type-theoretic function. In particular {\it in {\bf MF} only
type-theoretic functions between two sets form a set, while  functional relations between two sets form generally a collection}.

This restriction is crucial to make {\bf MF} compatible with classical predicative
theories as Feferman's predicative theories \cite{Fef}.
Indeed it is well-known  that the  {\it addition of the principle of excluded middle} can
turn a predicative theory
where functional relations between sets form a set, as Aczel's CZF or Martin-L{\"o}f's type theory,
 into an impredicative one where
 {\it power-collections become sets}. 
\end{itemize}

\subsection{The intensional level  of the Minimalist Foundation}
Here we describe the intensional level of the Minimalist Foundation
in \cite{m09}, which is represented by a dependent type theory called
\mtt. This type theory is  written  in the style of Martin-L{\"o}f's type theory \cite{PMTT} 
 by means of the following four kinds of judgements:
$$
A \ type\ [\Gamma] \hspace{.5cm} A=B\ type\ [\Gamma] 
\hspace{.5cm} a \in A\ 
 [\Gamma] \hspace{.5cm} a=b \in A\ [\Gamma] 
 $$
that is the type judgement (expressing that something is a specific type), 
the type equality judgement (expressing when
two types are equal), the
term judgement (expressing that something is a term of a certain type)
 and the term equality judgement
(expressing the {\it definitional equality} between terms of the same type), respectively, all under a context  $\Gamma$.

The word {\it type} is used as a meta-variable to indicate
four kinds of entities: collections, sets, propositions and small propositions, namely
 $$
 type \in \{ col, set,prop,prop_s\, \}
 $$
Therefore, in \mtt\ types are actually formed by using the following
judgements:
  $$
  A \ set\ [\Gamma] \qquad D\ col\ [\Gamma] 
\qquad \phi\ prop\
 [\Gamma]\qquad \psi\ prop_s\
 [\Gamma]
 $$
saying that $A$ is a set, that $D$ is a collection, that $\phi$ is a proposition
and that $\psi$ is a small proposition.

Here, contrary to \cite{m09} where capital latin letters are used as meta-variables for all types,  we use greek letters $\psi, \phi$ as meta-variables for propositions, we mostly use capital latin letters $A,B$ as meta-variables for sets and capital latin letters $C,D$ as meta-variables for collections.

As in the intensional version of Martin-L{\"o}f's type theory, in \mtt\
 there are two kinds of equality concerning terms:
one is the definitional equality of terms of the same type given
 by the judgement
$$
a=b\in A\ [\Gamma]
$$
which is decidable,
and the other is the  propositional equality written 
$$
\mathsf{\mathsf{Id}}(A,a,b)\ prop\ [\Gamma]
$$
which is not necessarily decidable.

 We now proceed by briefly describing the various kinds of types in \mtt, starting
from small propositions and propositions and then passing to sets and finally collections.

{\it Small propositions} in \mtt\  include all the logical constructors of intuitionistic predicate logic with equality and quantifications
restricted to sets: 
\\

\begin{tabular}{lc}
$\phi\ prop_s\  \equiv$& $ \perp\enspace\mid$  $ \enspace\phi \wedge \psi \enspace\mid$ 
 $\phi\vee \psi \enspace \mid$ $\ \phi\rightarrow \psi \enspace\mid $ $(\forall x\in A)\ \phi(x) \enspace\mid$  $(\exists x\in A)\ \phi(x) \enspace\mid$  $\ \mathsf{Id}(A,a,b)$
\end{tabular}
\\

\noindent
{\it provided that $A$ is a set}.

Then,   {\it propositions} in \mtt\  include all the logical constructors of intuitionistic predicate logic with equality and quantifications
on all kinds of types, i.e. sets and collections. Of course, small propositions are also propositions.
\\

\begin{tabular}{lc}
$\phi\ prop\  \equiv$& $ \phi\ prop_s\, \mid$  $ \phi\wedge \psi\ \mid$ 
 $\phi \vee \psi\ \mid$ $\ \phi\rightarrow \psi\ \mid $ $(\forall x\in D)\ \phi(x)\ \mid$  $(\exists x\in D)\ \phi(x)\ \mid$  $\ \mathsf{Id}(D,d,b)$
\end{tabular}
\\

In order to close sets under comprehension, for example to include
the set of positive natural numbers $\{ x\in \nat\ \mid\ x \geq 1\}$, and to define operations on such sets, 
 we need to think of propositions
as types of their proofs:
small propositions are seen as sets of their proofs
while generic propositions are seen  as collections of their proofs. That is, we  add to \mtt\  the following rules
 $$
 \begin{array}{l}
 \mbox{\bf prop$_s$-into-set) }\ \ \displaystyle{ \frac
        {\displaystyle\  \phi
         \hspace{.1cm} prop_s\ }
 { \displaystyle\ \phi
         \hspace{.1cm} set\ }}
 \end{array}
 \qquad\qquad
 \begin{array}{l}
     \mbox{\bf prop-into-col) }\ \ 
 \displaystyle{ \frac
        {\displaystyle\  \phi\ prop\  }
       { \displaystyle\  \phi\ col  \ }}
 \end{array}
 $$
Before explaining the difference between the notion of set and collection we
describe their constructors in \mtt.

{\it Sets} in \mtt\  are characterized as  inductively generated types and they include the following:\\

\begin{tabular}{lc}
$ A\ set\  \equiv$& $ \phi\ prop_s \mid$ $  N_0 \mid$  $ N_1 \mid$ $ N \mid$ $ List(A)  \mid $
 $(\Sigma x\in A)\, B(x) \mid$ $A+B \mid $ $ (\Pi x\in A)\, B(x)\, $
\end{tabular}
\\

\noindent
where the notation $N_0$ stands for the empty set,
$N_1$ for the singleton set, $N$ for the set of natural numbers, $List(A)$ for the set of Lists on the set $A$,
$(\Sigma x\in A) B(x)$ for the indexed sum of the family of sets $B(x)\ set \ [x\in A]$ indexed on the set $A$, $A+B$ for
the disjoint sum of the set $A$ with the set $B$, $(\Pi x\in A) B(x)$ for the product type of the family of sets $B(x)\ set \ [x\in A]$ indexed on the set $A$. 

It is worth noting that the set $N$ of the natural numbers is not present in a primitive way in \mtt\ since its rules can be derived by putting $N\,\equiv\,List(N_1)$. Here we add it to the syntax of \mtt\, because it plays a prominent role in realizability and we want to interpret it directly in \tar\ to avoid complications due to list encodings.


Finally, {\it collections} in \mtt\ include the following types:\\

\begin{tabular}{lc}
$ D\ col\  \, \equiv\,$& $ A \ set\enspace\mid$ $ \enspace \phi\ prop\ \enspace\mid$  $ \enspace \mathsf{prop_s}\enspace\mid$ $\ A\rightarrow \mathsf{prop_s}\enspace \mid $ $\ (\Sigma x\in D)\, E(x)$
\end{tabular}

\noindent
and all sets are collections thanks to the following rule:
$$
\begin{array}{l}
     \mbox{\bf set-into-col) }\ \ 
 \displaystyle{ \frac
        {\displaystyle\  A\ set\  }
       { \displaystyle\  A\ col  \ }}
 \end{array}
$$

\noindent
where $\mathsf{prop_s}$ stands for the 
collection of (codes for) small propositions and $A\rightarrow  \mathsf{prop_s}$  for the 
collection of propositional functions of the set $A$, while $ (\Sigma x\in D)\, E(x) $ stands
for the indexed sum of the family of collections $E(x)\ col\ [x\in D]$ indexed on the collection $D$.

Note that the collection of small propositions $  \mathsf{prop_s}$
is defined here with codes \`a la Tarski as in \cite{PMTT},
contrary to the version in \cite{m09}, to make the interpretation easier to understand. Its rules are the following.

\noindent
Elements of the collection of small propositions are generated as follows:
\\

{\small
$\begin{array}{ll} 
\mbox{Pr$_1$) }\ \widehat{\bot} \in \mathsf{prop_s} &
\mbox{Pr$_2$) }\displaystyle{ \frac
         { \displaystyle  p\in  \mathsf{prop_s} \hspace{.3cm} q \in
         \ \mathsf{prop_s}}
         {\displaystyle p \widehat{\vee} q\in \mathsf{prop_s} }}\\[15pt]
\mbox{Pr$_3$) }\displaystyle{ \frac
       {\displaystyle p\in
        \ \mathsf{prop_s} \qquad q\in
        \ \mathsf{prop_s} }
{p\widehat{\rightarrow} q \in \mathsf{prop_s} }}&
\mbox{Pr$_4$) }\displaystyle{ \frac
         { \displaystyle  p   \in \mathsf{prop_s}\qquad  q\in
           \mathsf{prop_s} }
         {\displaystyle  p\widehat{\wedge} q\in  \mathsf{prop_s} }}
\\[15pt]
\mbox{Pr$_5$) }\displaystyle{ \frac
{\displaystyle  A \hspace{.1cm} set \hspace{.3cm}  a\in A \hspace{.3cm} b\in A}
         {\displaystyle \widehat{\orig{Id}}(A, a, b) \in \mathsf{prop_s} }}
&
\mbox{Pr$_6$) }\displaystyle{ \frac
         { \displaystyle   p(x)
         \ \mathsf{prop_s}\ \ [x\in B]\qquad B\  set}
         {\displaystyle  \widehat{(\exists x\in B)} p(x)\in
\mathsf{prop_s}}} \\[15pt]
&\mbox{Pr$_7$) }\displaystyle{ \frac
{\displaystyle p(x) \in \mathsf{prop_s}\ [x\in B]\qquad B\ set  }
{\displaystyle \widehat{(\forall x\in B)} p(x) \in \mathsf{prop_s} }}
\end{array}
$}
\\

Elements of the collection of small propositions can be decoded 
as  small propositions via an operator as follows
\\

{\small
\noindent
$
\mbox{$\tau$-Pr) }\displaystyle{ \frac
       {\displaystyle p\in
        \ \mathsf{prop_s} }
{ \tau(p) \ prop_s }}
$
\\
}

\noindent
and this operator satisfies the following definitional equalities:

{\small
$\begin{array}{ll} 
\mbox{eq-Pr$_1$) }\ \tau(\widehat{\bot})= \bot\, prop_s &
\mbox{eq-Pr$_2$) }\displaystyle{ \frac
         { \displaystyle  p\in  \mathsf{prop_s} \hspace{.3cm} q \in
         \ \mathsf{prop_s}}
         {\displaystyle \tau( p \widehat{\vee} q)= \tau(p) \vee \tau(q)\,prop_s}}\\[15pt]
 \mbox{eq-Pr$_3$) }\displaystyle{ \frac {\displaystyle p\in
     \ \mathsf{prop_s} \qquad q\in
         \ \mathsf{prop_s} }
 {\tau(p\widehat{\rightarrow} q)= \tau(p)\rightarrow \tau(q)\,prop_s }}
&
\mbox{eq-Pr$_4$) }\displaystyle{ \frac
         { \displaystyle  p   \in \mathsf{prop_s}\qquad  q\in
           \mathsf{prop_s} }
         {\displaystyle  \tau(p\widehat{\wedge} q)=\tau(p) \wedge \tau(q)\, prop_s}}
\\[15pt]
\mbox{eq-Pr$_5$) }\displaystyle{ \frac
{\displaystyle  A \hspace{.1cm} set \hspace{.3cm}  a\in A \hspace{.3cm} b\in A}
         {\displaystyle \tau(\, \widehat{\orig{Id}}(A, a, b)\,)=
\orig{Id}(A,a,b)\,prop_s    }} &
\mbox{eq-Pr$_6$) }\displaystyle{ \frac
         { \displaystyle   p(x)
         \ \mathsf{prop_s}\ \ [x\in B]\qquad B\  set}
         {\displaystyle \tau( \widehat{(\exists x\in B)} p(x))=
(\exists x\in B) \tau(p(x))\, prop_s }} \\[15pt]
&\mbox{eq-Pr$_7$) }\displaystyle{ \frac
{\displaystyle p(x) \in \mathsf{prop_s}\ [x\in B]\qquad B\ set  }
{\displaystyle \tau(\widehat{(\forall x\in B)} p(x))= (\forall x\in B) \tau(p(x)) \, prop_s}}
\end{array}
$}
\\

In the realizability
interpretation of \mtt\ 
we need to define a subset of natural numbers including codes of \mtt -sets in order
to define the subset of codes of small propositions closed under quantification
on sets.
The existence of such  a subset of set codes says that the realizability interpretation is actually interpreting an extension of 
 \mtt\ with a collection of sets.
In order to simplify the definition of the realizability interpretation,
we interpret  an extension of \mtt, which  
 we call \mtts,  with the addition of the
 collection $  \mathsf{Set}$ of set codes %
whose related rules are the following. We don't give any elimination and conversion rule as those of universes \`a la Tarski in \cite{PMTT} since it would not be validated in the model (because we do not have least fix-points in \tar). 

\noindent
{\small
$ 
\begin{array}{l}
 \mbox{\bf Collection of sets}\\[5pt]
\mbox{F-Se) } \ \displaystyle{\mathsf{Set}\  col 
}\end{array}
$}
\\
\\

\noindent
Elements of the collection of sets are generated as follows:
\\

{\small
$\begin{array}{ll} 
\mbox{Se$_e$) }\ \widehat{N_0} \in \mathsf{Set} &
\mbox{Se$_s$) }\ \widehat{N_1} \in \mathsf{Set}\\[15pt]
\mbox{Se$_l$) }\displaystyle{ \frac
         { \displaystyle  a\in  \mathsf{Set} \hspace{.3cm} }
         {\displaystyle\widehat{List} (a)\in \mathsf{Set} }}
&\mbox{Se$_u$) }\displaystyle{ \frac
         { \displaystyle  a\in  \mathsf{Set} \hspace{.3cm} b\in  \mathsf{Set}  }
         {\displaystyle a \widehat{+} b\in \mathsf{Set} }}
\\[15pt]
\mbox{Se$_\Sigma$) }\displaystyle{ \frac
         { \displaystyle   a(x)
         \ \mathsf{Set}\ \ [x\in B]\qquad B\  set}
         {\displaystyle  \widehat{(\Sigma x\in B)} a(x)\in
\mathsf{Set}}} 
&
\mbox{Se$_\Pi$) }\displaystyle{ \frac
         { \displaystyle   a(x)
         \ \mathsf{Set}\ \ [x\in B]\qquad B\  set}
         {\displaystyle  \widehat{(\Pi x\in B)} a(x)\in
\mathsf{Set}}}\\[15pt]
 \mbox{sp-i-p) }\displaystyle{ \frac {\displaystyle p\in
     \ \mathsf{prop_s}  } {  p\in \mathsf{Set}  }}& 
\end{array}
$}
\\

%
%
%
%

\mtt\ can be viewed as a {\it predicative version} of the Calculus
of Constructions~\cite{TC90}, for short CoC. The main difference with respect
to CoC is that 
\mtt\ distinguishes between sets and collections in a way
 similar to the distinction between sets and classes in axiomatic set theory.
However, all types  of \mtt, i.e. small propositions, propositions, sets and collections, are predicative entities
in the sense that their elements can be generated in an inductive way
by a finite number of rules. According to the notion of set
in Bishop~\cite{Bishop}
and  Martin-L{\"o}f~\cite{pML70}, all \mtt-types are actually sets,
and in fact
\mtt-types can be interpreted as sets in the intensional
version of  Martin-L{\"o}f's type theory in \cite{PMTT}.
The \mtt-distinction  between sets and collections, and the corresponding distinction between small propositions and propositions,
is motivated by the need of distinguishing between predicative entities whose notion of element is a closed concept, and these are called sets,
and those entities whose notion of element is an open concept, and these are called collections.
The motivating idea  is that a set is inductively generated by a finite number of rules whose associated inductive principle does not vary when the theory \mtt\
is extended with new entities (sets, collections or propositions). On the contrary a collection is inductively generated by a finite number of rules
which may vary when the theory is extended with new entities.
 Typical examples of collections
are universes (of sets or propositions): if we extend the theory \mtt\ with a new small proposition, then we need to add a new rule inserting this new
small proposition in the collection of small propositions. 


We recall from \cite{mtt} that the distinction between propositions and sets is crucial to avoid the validity of choice principles.

Finally, it is worth noting that in \mtt\ we restrict substitution term equality rules
to explicit substitution term equality rules 
of the form
\\

\noindent
$$
\begin{array}{l}
      \mbox{sub)} \ \
\displaystyle{ \frac
         { \displaystyle 
\begin{array}{l}
 c(x_1,\dots, x_n)\in C(x_1,\dots,x_n)\ \
 [\, x_1\in A_1,\,  \dots,\,  x_n\in A_n(x_1,\dots,x_{n-1})\, ]   \\[5pt]
a_1=b_1\in A_1\ \dots \ a_n=b_n\in A_n(a_1,\dots,a_{n-1})
\end{array}}
         {\displaystyle c(a_1,\dots,a_n)=c(b_1,\dots, b_n)\in
 C(a_1,\dots,a_{n})  }}
\end{array}     
$$

\noindent
in place of usual term equality rules preserving term constructions typical
of   Martin-L{\"o}f's type theory in \cite{PMTT}. 
This restriction, and in particular the absence of 
the so called $\xi$-rule of lambda-terms
$$
\mbox{  $\xi$} \
\displaystyle{\frac{ \displaystyle c=c'\in  C\ [x\in B]  }
{ \displaystyle \lambda x^{B}.c=\lambda x^{B}.c' \in (\Pi x\in B) C}}
$$
 seems to be crucial to prove consistency of \mtt\ with {\bf AC}+{\bf CT}, as advocated in \cite{mtt}, by means
of  a realizability semantics \`a la Kleene, but this is still an open problem (the realizability semantics given here does not help to solve this since it can not validate {\bf AC} on all types). 
It is worth to recall from \cite{m09} that our restriction of term equality  does not affect the possibility of adopting \mtt\ as
the intensional
level of a  two-level constructive foundation
as intended in  \cite{mtt}. Indeed the 
   term equality rules of \mtt\  suffice to interpret
an extensional level including extensional equality of functions, as that represented
by \emtt, by means of the quotient model  described in \cite{m09} and studied abstractly
in ~\cite{qu12,elqu,uxc}.

\subsection{The extensional level of the Minimalist Foundation}
Here we briefly describe the extensional level \emtt\ of the Minimalist Foundation. This is an extensional dependent type theory extending extensional
Martin-L{\"o}f's type theory in~\cite{ML84} with primitive (proof-irrelevant)
propositions, power-collections and quotients.

The rules of \emtt\ are formulated by using the same  kinds of judgements used for \mtt.
The main peculiar characteristics of \emtt\ in comparison to \mtt\ are the following.

\begin{enumerate}
\item
A primary difference between \emtt\  and \mtt\ is the usual difference between
the so called intensional version of  Martin-L{\"o}f's type theory \cite{PMTT} 
and its extensional one in \cite{ML84} and this is the fact that the definitional equality of terms 
$$
a=b\in A\ [\Gamma]
$$
is no longer {\it decidable} in \emtt\
as it is in the intensional \mtt.
This is in turn due to the fact  that the propositional equality of \emtt\,
as that of \cite{ML84},  called $\mbox{$\orig{Eq}$}(A, a, b)$, is extensional
in the sense that the provability of  $ \mbox{$\orig{Eq}$}(A, a, b)\ [\Gamma]$ in \emtt\ is equivalent to the derivation of the judgement $
a=b\in A\ [\Gamma]
$. Instead, in \mtt\ only the derivation of the definitional equality judgement
$
a=b\in A\ [\Gamma]
$
implies internally the provability of  the intensional propositional equality $ \mbox{$\orig{Id}$}(A, a, b)\ [\Gamma]$
under a generic context.

\item
Another peculiar feature of  \emtt\  employs
the distinction between propositions and sets: this is the addition of
  proof-irrelevance for propositions  captured by the following rules
$$
\mbox{{\bf prop-mono}) }  \,\displaystyle{ \frac{ \displaystyle \phi\ prop\ [\Gamma]\qquad p\in \phi\  [\Gamma] \quad q\in \phi\  [\Gamma]}{\displaystyle p=q\in \phi\ \ [\Gamma]}}  
\qquad
\begin{array}{l}
\mbox{{\bf prop-true}) }  \,\displaystyle{ \frac{ \displaystyle \phi\ prop\ \qquad p\in \phi\  }
{\displaystyle \mathsf{true}\in \phi}} 
\end{array}
$$
 saying that  a proof of a proposition
is {\it unique} and equal to a canonical proof term called
$\mathsf{true}$.
Of course, these rules can not be added to an extensional theory identifying propositions  with sets as Martin-L{\"o}f's one in \cite{ML84}, because they would trivialize all constructors. Moreover, these rules
are not present in the intensional level \mtt\ because proof-irrelevance is a typical
extensional condition. Indeed,  \emtt-propositions  can be thought of as quotients
of intensional propositions under the trivial equivalence relation between proofs. 
\item
Other key differences between the type theories \mtt\ and \emtt\ 
  are the addition in \emtt\ 
of quotient sets 
$$A/\rho\ set \ [\Gamma]$$
 provided that $\rho $ is a small equivalence relation
$\rho \
prop_s\ [x\in A,y\in A]$ on the set $A$, and the addition of the power-collection of the singleton and of
the power-collection of a generic set  $A$ 
$$ \enspace {\cal P}(1)\enspace \qquad \qquad \ A\rightarrow  {\cal P}(1)\enspace $$

\item
A further difference  between the type theories \mtt\ and \emtt\  concerns
 the equality rules between terms. Indeed in \emtt\ equality rules between terms
are the usual ones typical of an extensional type theory in \cite{ML84}
preserving all term constructors. In particular, equality of lambda-functions
is extensional, namely it is possible to prove

$$(\forall x\in A) \mbox{$\orig{Eq}$}(\, B(x) , f(x)\, ,\,  g(x))
\ \rightarrow \ \mbox{$\orig{Eq}$}(\,(\Pi x\in A)B(x)\, ,\
 \lambda x. f(x)\ ,\ \lambda x. g(x)\, )$$
This proposition is not necessarily provable at the intensional level \mtt\ when substituting
the extensional propositional equality $\mbox{$\orig{Eq}$}(A, a,b)$
with the intensional one $\mbox{$\orig{Id}$}(A, a,b)$. 
\end{enumerate}

We end by recalling from \cite{m09} that {\it  a model for \mtt\ can be turned into a model for \emtt\ by using the interpretation of \emtt\ into \mtt\ described in} \cite{m09}. Therefore in the following we are going to define a realizability interpretation just for \mtt, to get one also for \emtt.

\subsection{Untyped syntax of  \mtts}
Usually in type theory the syntax is introduced \emph{in fieri}; for example terms are introduced typically after deriving some conditions or constraints which are required to define them. However for semantical purposes it looks more convenient to present the syntax \emph{a priori in a partial way} by eliminating parts of usual restrictions.

Therefore, since we want to define a realizability interpretation for \mtts, we introduce here the syntax of all \mtts-type and term constructors in a partial way and we refer the reader  to look at \cite{m09} for all the \mtt-rules. Then we will define a partial interpretation for terms of our \emph{extended} syntax and check that this interpretation is well defined in case the constraints for introducing them are validated by the model.

\begin{definition}
Let $\cx$ be a context, i.\,e.\, $\cx=[x_{1},...,x_{\mathsf{n}}]$ is a possibly empty list of distinct variables. \emph{Terms, small propositions, sets, propositions and collections} in context are defined according to the following conditions. If
\begin{enumerate}
\item $t\cx, t'\cx, t''\cx, s\cxy, s'\cxy,r\cxyz,q\cxyzu$ are terms in context;
\item $\phi\cx, \phi'\cx, \psi\cxy$ are small propositions in context;
\item $A\cx, A'\cx, B\cxy$ are sets in context;
\item $\eta\cx, \eta'\cx, \rho\cxy$ are propositions in context;
\item $D\cx, E\cxy$ are collections in context, 
\end{enumerate}
then
\begin{enumerate}
\item $x_{i}\cx$ is a term in context;
\item[]{\it the empty set eliminator} $\mathsf{emp}_{0}(t)\cx$ is a term in context;
\item[]  {\it the singleton constant} $\star\cx$ and {\it the singleton eliminator} $El_{N_{1}}(t,t')\cx$ are terms in context;
\item[] {\it the zero constant} $0\cx$, {\it the successor constructor} $\mathsf{succ}(t)\cx$ and  {\it the  eliminator of natural numbers }$El_{N}(t,t',(y,z)r)\cx$ are terms in context\footnote{The rules for these constructors derive from those of $List(N_{1})$ in \mtt\ by identifying $0$ with $\epsilon$, $\mathsf{succ}(t)$ with $\mathsf{cons}(t,\star)$ and $El_{N}(t,t',(y,z)r)$ with $El_{List(N_{1})}(t,t',(y,y',z)r)$.};
\item[] {\it the lambda abstraction of dependent product} $\lambda y.s\cx$ and {\it its application} $\mathsf{Ap}(t,t')\cx$ are terms in context;
\item[]{\it the pairing of strong indexed sum} $\pair{t}{t'}\cx$ and
{its eliminator} $El_{\Sigma}(t,(y,z)r)\cx$ are terms in context;
\item[] {\it the first injection of  binary disjoint sum} 
$\mathsf{\mathsf{inl}}(t)\cx$ and {\it its second injection} $\mathsf{inr}(t)\cx$ and {\it its eliminator} $El_{+}(t,(y)s,(y)s')\cx$ are terms in context;
\item[] {\it the empty list} $\epsilon\cx$, {\it the list constructor} $\mathsf{cons}(t,t')\cx$ and {\it its eliminator} $El_{List}(t,t',(y,z,u)q)\cx$ are terms in context;
\item[]  {\it the false eliminator} $\mathsf{r}_{0}(t)\cx$ is a term in context;
\item[] {\it the pairing of conjunction} $\langle t,_{\wedge} t'\rangle\cx$, and {\it its first and second projections} $\pi^{\wedge}_{1}(t)\cx$ and  $\pi^{\wedge}_{2}(t)\cx$  are terms in context;
\item[] {\it the first injection of disjunction} $\mathsf{\mathsf{inl}}_{\vee}(t)\cx$, {\it the second injection of disjunction} $ \mathsf{inr}_{\vee}(t)\cx$ and {\it its eliminator} $El_{\vee}(t,(y)s,(y)s')\cx$ are terms in context;
\item[]  {\it the lambda abstraction of implication} $\lambda_{\rightarrow} y.s\cx$ and {\it its application} $\mathsf{Ap}_{\rightarrow}(t,t')\cx$ are terms in context;
\item[] {\it the pairing of existential quantification} $\langle t,_{\exists} t'\rangle\cx$ and {\it its eliminator}  $El_{\exists}(t,(y,z)r)\cx$  are terms in context;
\item[] {\it the lambda abstraction of universal quantification} $\lambda_{\forall} y.s\cx$ and  {\it its application} $\mathsf{Ap}_{\forall}(t,t')\cx$ are terms in context;
\item[] {\it the Propositional Identity term constructor} $\mathsf{id}(t)\cx$ and {\it its eliminator}  $El_{\mathsf{Id}}(t,t',t'',(y)s)\cx\footnote{In the rules for $\mathsf{Id}(A,a,b)$ of $\mtt$ the eliminator $El_{\mathsf{Id}}(p,(x)c)$ is substituted by an eliminator $El_{\mathsf{Id}}(a,b,p,(x)c)$ with explicit reference to $a\in A$ and $b\in A$. The rules remain the same.}$ are terms in context;
\item[] {\it the empty set code} $\widehat{N_{0}}[\underline{x}]$, {\it the singleton code} $\widehat{N_{1}}[\underline{x}]$, {\it the natural numbers set code} $\widehat{N}[\underline{x}]$, {\it the dependent product code} $(\widehat{\Pi y\in A})s[\underline{x}]$, {\it the dependent sum code} $(\widehat{\Sigma y\in A})s[\underline{x}]$, {\it the disjoint sum code} $t\widehat{+}t'[\underline{x}]$, {\it the list code} $\widehat{List}(t)[\underline{x}]$,  {\it the falsum code} $\widehat{\bot}$, {\it the conjunction code} $t\widehat{\wedge}t'$,   {\it the disjunction code} $t\widehat{\vee}t'$,   {\it the implication  code} $ t\widehat{\rightarrow}t'$,  {\it the existential quantification  code} $(\widehat{\exists y \in A})s\cx$, {\it the universal quantification  code} $(\widehat{\forall}y \in A)s\cx$ and {\it the propositional identity code} $\widehat{\mathsf{Id}}(A,t,t')\cx$ are terms in context;\\
\item $\bot\cx$ is a small proposition in context;
\item[] $\tau(t)\cx$ is a small proposition in context;
\item[] $\phi\wedge \phi'\cx$, $\phi\vee\phi'\cx$ and $\phi\rightarrow \phi'\cx$ are small propositions in context;
\item[] $(\exists y\in A)\,\psi\cx$ and $(\forall y\in A)\,\psi\cx$ are small propositions in context;
\item[] $\mathsf{Id}(A,t,t')\cx$ is a small proposition in context;\\
\item $\phi\cx$ is a set in context;
\item[] $N_{0}\cx, N_{1}\cx$ and $N\cx$ are sets in context;
\item[] $(\Pi y\in A)\,B\cx$, $(\Sigma y\in A)\,B\cx$, $A+A'\cx$ and $List(A)\cx$ are sets in context;\\
\item $\phi\cx$ is a proposition in context;
\item[] $\eta\wedge \eta'\cx$, $\eta\vee\eta'\cx$ and $\eta\rightarrow \eta'\cx$ are propositions in context;
\item[] $(\exists y\in D)\,\rho\cx$ and $(\forall y\in D)\,\rho\cx$ are propositions in context;
\item[] $\mathsf{Id}(D,t,t')\cx$ is a proposition in context;\\
\item $\eta\cx$ is a collection in context;
\item[] $A\cx$ is a collection in context;
\item[] $\mathsf{Set}\cx$ is a collection in context;
\item[] $\mathsf{prop_s}\cx$ is a collection in context;
\item[] $A\rightarrow \mathsf{prop_s}\cx$ is a collection in context;
\item[] $(\Sigma y\in D)E\cx$ is a collection in context.\\
\end{enumerate}
For sets in context $A\,[\underline{x}]$ we define an abbreviation $\widehat{A}\,[\underline{x}]$  as follows:\\
\begin{enumerate}
\item $\widehat{\bot}$, $\widehat{N_{0}}$, $\widehat{N_{1}}$ and $\widehat{N}$ were already defined;
\item $\widehat{((\Pi y\in A)\,B)}=(\widehat{\Pi y\in A})\,\widehat{B}$, $\widehat{((\Sigma y\in A)\,B)}=(\widehat{\Sigma y\in A})\,\widehat{B}$, 
\item $\widehat{A+A'}=\widehat{A}\widehat{+}\widehat{A}'$, $\widehat{List(A)}=\widehat{List}(\widehat{A})$,
\item $\widehat{\phi\wedge \phi'}=\widehat{\phi}\,\widehat{\wedge}\,\widehat{\phi'}$, $\widehat{\phi\vee \phi'}=\widehat{\phi}\,\widehat{\vee}\,\widehat{\phi'}$, $\widehat{\phi\rightarrow \phi'}=\widehat{\phi}\,\widehat{\rightarrow}\,\widehat{\phi'}$,
\item $\widehat{((\exists y\in A)\,\psi)}=(\widehat{\exists y\in A})\,\widehat{\psi}$, $\widehat{((\forall y\in A)\,\psi)}=(\widehat{\forall y\in A})\,\widehat{\psi}$, $\widehat{\mathsf{Id}(A,t,s)}=\widehat{\mathsf{Id}}\,(A,t,s)$,
\item $\widehat{\tau(t)}=t$.
\end{enumerate}

\end{definition}

It is clear that the previous definition is overabundant with respect to the common use in type theory. We introduced some terms which we will never find in any standard type theory, as for example the term $0\widehat{\wedge}El_{N_{1}}(\lambda x.x,\lambda_{\rightarrow}y.y)$
which is obtained by gluing together terms which usually have types which are not compatible. For example $0$ is usually typed as a natural number, while $\widehat{\wedge}$ connects codes for small propositions.

\section{The realizability interpretation for \mtts}

The preliminary step in the presentation of the Kleene realizability interpretation consists in presenting the theory of \emph{Inductive Definitions} $\tar$ in which we will interpret $\mtts$.

\subsection{The system $\tar$}
The system $\tar$ is a predicative fragment of second-order arithmetic, more precisely it is the predicative fragment of second-order arithmetic extending Peano arithmetics with some (not necessarily least) fix points for each positive arithmetical operator. Its number terms are number variables (we assume that these variables are equal to those of $\mtts$ ), 
the constant $0$ and the terms built by applying the unary successor functional symbol $succ$ and the binary sum and product functional symbols $+$ and $*$ to number terms. Set terms are only set variables $X,Y,Z...$. The \emph{arithmetical} formulas  are obtained starting from $t=s$ and $t\varepsilon X$ with $t,s$ number terms and $X$ a set variable, by applying the connectives $\wedge, \vee, \neg, \rightarrow$ and the number quantifiers $\forall x$, $\exists x$. 
Moreover let us give the following two definitions.
\begin{definition} An occurrence of a set variable $X$ is \emph{positive} in an arithmetical formula $\varphi$ if and only if $\varphi$ is $t\varepsilon X$ for some number term $t$ or $\varphi$ is $\psi\wedge \psi'$, $\psi'\wedge \psi$, $\psi\vee \psi'$, $\psi'\vee \psi$ , $\psi'\rightarrow \psi$, $\exists x\, \psi$ or $\forall x\, \psi$ and the occurrence of $X$ is a positive occurrence of $X$ in $\psi$.
\end{definition}
\begin{definition} An arithmetical formula $\varphi$ with exactly one free number variable $n$ and one free set variable $X$ which occurs only positively is called an \emph{admissible} formula.
\end{definition}
In order to define the system $\tar$ we add to the language of arithmetic a unary predicate symbol $P_{\varphi}$ for every admissible formula $\varphi$ . The atomic formulas of $\tar$ are 
\begin{enumerate}
\item $t=s$ with $t$ and $s$ number terms,
\item $t\varepsilon X$ with $t$ a number term and $X$ a set variable,
\item $P_{\varphi}(t)$ with $t$ a number term and $\varphi$ an admissible formula.
\end{enumerate}
All formulas of $\tar$ are obtained by atomic formulas by applying connectives, number quantifiers and set quantifiers. 
\\
\\
The axioms of $\tar$ are the axioms of Peano Arithmetic plus the following three axiom schemata: 
\begin{enumerate}
\item \emph{Comprehension schema}: for all formulas $\varphi(x)$ of $\tar$ without set quantifiers
\[\exists X\forall x (x\varepsilon X \leftrightarrow \varphi(x))\]  
\item  \emph{Induction schema}: for all formulas $\varphi(x)$ of $\tar$ 
\[(\varphi(0)\wedge \forall x(\varphi(x)\rightarrow \varphi(succ(x))))\rightarrow \forall x \,\varphi(x)\]
\item \emph{Fix point schema}: for all admissible formulas $\varphi$ 
\[\varphi[P_{\varphi}/X]\leftrightarrow P_{\varphi}(x)\] 
where $\varphi[P_{\varphi}/X]$ is the result of substituting in $\varphi$ all instances of $x\varepsilon X$ with $P_{\varphi}(x)$.
\end{enumerate} 

The system $\tar$ allows us to define predicates as fix points, by using axiom schema $3$, if they are presented in a appropriate way (i.\,e.\,using admissible formulas).

A \emph{definable class} ${\cal C}$ of $\tar$ is a formal writing $\{x|\varphi(x)\}$ where $\varphi(x)$ is a formula of $\tar$. In this case we write $x\varepsilon {\cal C}$ as a shorthand for $\varphi(x)$.\\

\emph{Notation of computable operators in $\tar$}.

As it is well known, it is certainly possible to express a G\"odelian coding of recursive functions in $\tar$ using Kleene's predicate since it is already possible to do this in $\mathsf{PA}$. In particular we can consider a definitional extension of $\tar$ (which we still call $\tar$) in which there are terms with Kleene's brackets $\{t\}(s)$ and there is a predicate $\{t\}(s)\downarrow$ stating that the term with Kleene's brackets is well defined ($s$ is in the domain of the recursive function coded by $t$). We will write $\{t\}(s_{1},...,s_{\mathsf{n}})$ as a shorthand defined by induction: it is $\{t\}(s_{1})$ if $n=1$ while if $n>1$ and if we have already defined $\{t\}(s_{1},...,s_{\mathsf{n}})$, then $\{t\}(s_{1},...,s_{\mathsf{n}+1})=\{\{t\}(s_{1},...,s_{\mathsf{n}})\}(s_{\mathsf{n}+1})$. We denote by $\mathbf{succ}$ a numeral for which in $\{\mathbf{succ}\}(x)=succ(x)$ in $\tar$.

 As we well know, the s-m-n lemma (see e.\,g.\,\cite{odi}) gives  the structure of a partial combinatorial algebra to natural numbers endowed with Kleene application and this structure can be expressed in $\tar$. In particular we can find numerals $\mathbf{p},\mathbf{p}_{1},\mathbf{p}_{2}$ representing a fixed primitive recursive bijective pairing function with primitive recursive first and second projections. We will write $p_{1}(x)$, $p_{2}(x)$ and $\langle x,y\rangle$ as abbreviations for $\{\mathbf{p}_{1}\}(x)$, $\{\mathbf{p}_{2}\}(x)$ and $\{\mathbf{p}\}(x,y)$ respectively. It is also possible to define a numeral $\mathbf{ite}\footnote{if then else}$ representing the definition by cases ($\{\mathbf{ite}\}(n,m,l)\simeq\footnote{$a\simeq b$ means that $a\downarrow$ if and only if $b\downarrow$ and in this case $a=b$ in $\tar$.} m$ if $n=0$, $\{\mathbf{ite}\}(n,m,l)\simeq l$ if $n\neq 0$). We can also encode recursively finite list of natural numbers with natural numbers in such a way that the empty list is coded by $0$ and the concatenation is a recursive function which can be coded by a numeral $\mathbf{cnc}$. We have moreover numerals $\mathbf{rec}$ and $\mathbf{listrec}$ representing natural numbers recursion and lists recursion. These numbers in particular satisfy the following requirements:
\begin{enumerate}
\item $\{\rec\}(n,m,0)\simeq n$; 
\item $\{\rec\}(n,m,k+1)\simeq \{m\}(k,\{\rec\}(n,m,k))$;
\item $\{\lrec\}(n,m,0)\simeq n$;
\item $\{\lrec\}(n,m,\mathbf{cnc}(k,l))\simeq \{m\}(k,l,\{\lrec\}(n,m,k))$.
\end{enumerate}
For this representation of lists, the component functions $(-)_{j}$, turn out to be recursive. 

Moreover we can always define $\lambda$-terms $\Lambda n.t$ in $\tar$ for terms $t$ built with numerals, variables and Kleene application, in such a way that $\{\Lambda x.t\}(n)\simeq t[n/x]$ and $\{\Lambda x_{1}...\Lambda x_{\mathsf{n}}. t\}(n)\simeq \Lambda x_{2}...\Lambda x_{\mathsf{n}}.t[n/x_{1}].$

\subsection{The definition of interpretation}

The realizability interpretation for   \mtts\ 
we are going to describe
is a modification of Beeson's realizability semantics \cite{beeson} for the extensional version of first order Martin-L{\"o}f's type theory with one universe~\cite{ML84}. So it
 will be given in $\tar$ as Beeson's one.
Here we describe the key points of such an interpretation on which we follow
Beeson's semantics:
\begin{list}{-}{ }
\item all types of \mtts\  are interpreted as quotients of definable classes of $\tar$, intended as classes of ``their realizers''. In particular we use Beeson's technique of interpreting Martin-L\"of's universe to interpret the collection of (codes for) small propositions of \mtts. In order to do this it is crucial to have fix points and hence this is why we work in the theory $\tar$;

\item terms are interpreted as (codes) of recursive functions;

\item equality between terms in context is interpreted as extensional equality of recursive functions; 

\item the interpretation of substitution will be proven to be equivalent to the substitution in interpretation; 

\item we interpret $\lambda$-abstraction by using {\it s-m-n} lemma of computability, but then, in order to validate the condition of the previous point, we impose equality of type-theoretic functions to be extensional. Therefore the principle of Extensional Equality of Functions will turn out to be valid in our model.
\end{list}

\noindent
Instead we do not follow Beeson's semantics in the interpretation of propositions:
\begin{list}{-}{ }
\item  in order to validate {\bf formal Church Thesis} we interpret
propositions as trivial\footnote{A quotient is trivial if it is determined by a trivial relation i.\,e.\, a relation for which all pairs of elements are equivalent.} quotients of original Kleene realizability.
As a consequence Martin-L{\"o}f's isomorphism of propositions-as-sets
 together with the validity of the {\bf Axiom of Choice} is not validated
in our realizability semantics contrary to Beeson's one.
\end{list}





We can summarize the interpretation of terms and types with the following table:
\[
\begin{tabular}{|c|c|}
\hline
Terms &(codes) of recursive functions\\
\hline
Collections		&Quotients of definable classes $(C,\simeq)$\\
\hline
Propositions &quotients of definable classes on trivial $\simeq$\\
\hline
\end{tabular}
\]

\subsubsection*{The interpretation of terms}
Before giving the interpretation of \mtts-terms, we need to present explicitly a convention about how to encode \mtts-sets with numerals. We will code sets 
as $\{\mathbf{p}\}(a, \langle b_{1},...,b_{\mathsf{n}}\rangle)$, where $a$ is a number coding a particular constructor and $\langle b_{1},...,b_{\mathsf{n}}\rangle$ is a lists of codes for ingredients needed by the constructor itself. The following table makes evident the choices for $a$:
\\
\\
{\tiny
\[
\begin{tabular}{|c|c|c|c|c|c|c|c|c|c|c|c|}
\hline 
& & & & & & & & & & &  \\
$N_{0},N_{1},N$ &$\Pi$ & $\Sigma$ & $+$ &$List$ &$\bot$ & $\wedge$ &$\vee$ &$\rightarrow$ & $\exists$  & $\forall$    & $\mathsf{Id}$\\
& & & & & & & & & & &  \\
\hline 
& & & & & & & & & & &  \\
1 &2 &3 &4 &5 &6 &7 &8 &9 &10 &11 &12 \\
& & & & & & & & & & &  \\
\hline
\end{tabular}
\]
}
\\
\\
Notice that codes for small propositions must have $a>5$.

We can now proceed to the definition of the interpretation of \mtts-terms.
\begin{definition}
Terms in context $t[x_{1},...,x_{\mathsf{n}}]$ are interpreted as 
$${\cal I}(t[x_{1},...,x_{\mathsf{n}}])=\Lambda x_{1}...\Lambda x_{\mathsf{n}}.{\cal I}(t)$$
where ${\cal I}(t)$ are terms of the extended language of $\tar$ defined as follows
\begin{enumerate}
\item If $x$ is a variable, then ${\cal I}(x)=x$;
\item ${\cal I}(\mathsf{emp}_{0}(t))={\cal I}(\mathsf{r}_{0})=0$;
\item ${\cal I}(\star)=0$ and ${\cal I}(El_{N_{1}}(t,t'))={\cal I}(t')$;
\item ${\cal I}(0)=0$ and ${\cal I}(\mathsf{succ}(t))=\{\mathbf{succ}\}({\cal I}(t))$,
\item[]${\cal I}(El_{N}(t,t',(y,z)r))=\{\rec\}({\cal I}(t'),\Lambda y.\Lambda z. {\cal I}(r),{\cal I}(t))$;
\item ${\cal I}(\lambda y.s)={\cal I}(\lambda_{\rightarrow} y.s)={\cal I}(\lambda_{\forall} y.s)=\Lambda y. {\cal I}(s)$,
\item[]${\cal I}(\mathsf{Ap}(t,t'))={\cal I}(\mathsf{Ap}_{\rightarrow}(t,t'))={\cal I}(\mathsf{Ap}_{\forall}(t,t'))=\{{\cal I}(t)\}({\cal I}(t'))$;
\item ${\cal I}(\pair{t}{t'})={\cal I}(\langle t,_{\wedge} t'\rangle)={\cal I}(\langle t,_{\exists} t'\rangle)=\{\mathbf{p}\}({\cal I}(t),{\cal I}(t'))$, 
\item[]${\cal I}(El_{\Sigma}(t,(y,z)r))={\cal I}(El_{\exists}(t,(y,z)r))=\{\Lambda y.\Lambda z. {\cal I}(r)\}(\{\mathbf{p}_{1}\}({\cal I}(t)),\{\mathbf{p}_{2}\}({\cal I}(t)))$, 
\item[]${\cal I}(\pi^{\wedge}_{1}(t))=\{\mathbf{p}_{1}\}({\cal I}(t))$,
\item[]${\cal I}(\pi^{\wedge}_{2}(t))=\{\mathbf{p}_{2}\}({\cal I}(t))$;
\item ${\cal I}(\mathsf{\mathsf{inl}}(t))={\cal I}(\mathsf{\mathsf{inl}}_{\vee}(t))=\{\mathbf{p}\}(0,{\cal I}(t))$, 
\item[]${\cal I}(\mathsf{inr}(t))={\cal I}(\mathsf{inr}_{\vee}(t))=\{\mathbf{p}\}(1,{\cal I}(t))$,
\item[]${\cal I}(El_{+}(t,(y)s,(y)s'))={\cal I}(El_{\vee}(t,(y)s,(y)s'))=$\\
$\qquad\qquad\qquad\ite{\mathbf{p}_{1}({\cal I}(t))}{\{\Lambda y. {\cal I}(s)\}(\{\mathbf{p}_{2}\}({\cal I}(t)))}{\{\Lambda y. {\cal I}(s')\}(\{\mathbf{p}_{2}\}({\cal I}(t)))}$;
\item ${\cal I}(\epsilon)=0$ and ${\cal I}(\mathsf{cons}(t,t'))=\{\mathbf{cnc}\}({\cal I}(t),{\cal I}(t'))$,
\item[]$El_{List}(t,t',(y,z,u)q)=\{\lrec\}({\cal I}(t'),\Lambda y.\Lambda z. \Lambda u. {\cal I}(q),{\cal I}(t))$;
\item ${\cal I}(\mathsf{id}(t))=0$,
\item[]${\cal I}(El_{\mathsf{Id}}(t,t',t'',(y)s))=\{\Lambda y.{\cal I}(s)\}({\cal I}(t))$;
\item ${\cal I}(\widehat{N_{0}})=\{\mathbf{p}\}(1,0)$, ${\cal I}(\widehat{N_{1}})=\{\mathbf{p}\}(1,1)$ and ${\cal I}(\widehat{N})=\{\mathbf{p}\}(1,2)$, 
\item[] ${\cal I}((\widehat{\Pi y \in A})s)=\{\mathbf{p}\}(2,(\{\mathbf{p}\}({\cal I}(\widehat{A}),(\Lambda y.{\cal I}(s)))))$, 
\item[] ${\cal I}((\widehat{\Sigma y \in A})s)=\{\mathbf{p}\}(3,(\{\mathbf{p}\}({\cal I}(\widehat{A}),(\Lambda y.{\cal I}(s)))))$,
\item[] ${\cal I}(t\widehat{+} t')=\{\mathbf{p}\}(4,(\{\mathbf{p}\}({\cal I}(t),{\cal I}(t')))$, 
\item[] ${\cal I}(\widehat{List}(t))=\{\mathbf{p}\}(5,{\cal I}(t))$, 
\item[] ${\cal I}(\widehat{\bot})=\{\mathbf{p}\}(6,0)$, 
\item[] ${\cal I}(t\widehat{\wedge} t')=\{\mathbf{p}\}(7,(\{\mathbf{p}\}({\cal I}(t),{\cal I}(t')))$, 
\item[] ${\cal I}(t\widehat{\vee} t')=\{\mathbf{p}\}(8,(\{\mathbf{p}\}({\cal I}(t),{\cal I}(t')))$, 
\item[] ${\cal I}(t\widehat{\rightarrow} t')=\{\mathbf{p}\}(9,(\{\mathbf{p}\}({\cal I}(t),{\cal I}(t')))$, 
\item[] ${\cal I}((\widehat{\exists y \in A})s)=\{\mathbf{p}\}(10,(\{\mathbf{p}\}({\cal I}(\widehat{A}),(\Lambda y.{\cal I}(s)))))$, 
\item[] ${\cal I}((\widehat{\forall y \in A})s)=\{\mathbf{p}\}(11,(\{\mathbf{p}\}({\cal I}(\widehat{A}),(\Lambda y.{\cal I}(s)))))$,
\item[] ${\cal I}(\widehat{\mathsf{Id}}(A,t,t'))=\{\mathbf{p}\}(12,(\{\mathbf{p}\}({\cal I}(\widehat{A}),(\{\mathbf{p}\}({\cal I}(t),{\cal I}(t'))))))$,
\end{enumerate}

%
\end{definition}

For the sake of example let us consider the interpretation of the term in context $t[x,y,z]$ defined as $\widehat{\mathsf{Id}}(\mathsf{Id}(N,x,x),y,z)[x,y,z]$:
\[\begin{array}{rl}
{\cal I}(t)[x,y,z])&=\Lambda x.\Lambda y. \Lambda z. {\cal I}(\widehat{\mathsf{Id}}(\mathsf{Id}(N,x,x),y,z))\\
										&=\Lambda x.\Lambda y.\Lambda z. \{\mathbf{p}\}(12,\,\{\mathbf{p}\}({\cal I}(\widehat{\mathsf{Id}(N,x,x)}),\{\mathbf{p}\}(y,z))\\\
										&=\Lambda x.\Lambda y.\Lambda z. \{\mathbf{p}\}(12,\,\{\mathbf{p}\}({\cal I}(\widehat{\mathsf{Id}}(N,x,x)),\{\mathbf{p}\}(y,z)))\\
										&=\Lambda x.\Lambda y.\Lambda z. \{\mathbf{p}\}(12,\,\{\mathbf{p}\}( \{\mathbf{p}\}(12,\,\{\mathbf{p}\}({\cal I}(\widehat{N}),\{\mathbf{p}\}(x,x))),\,\{\mathbf{p}\}(y,z)))\\ 	
										&=\Lambda x.\Lambda y.\Lambda z. \{\mathbf{p}\}(12,\,\{\mathbf{p}\}( \{\mathbf{p}\}(12,\,\{\mathbf{p}\}(\{\mathbf{p}\}(1,2),\{\mathbf{p}\}(x,x))),\,\{\mathbf{p}\}(y,z))).\\
\end{array}\]

We say that an interpretation of a term in context $t[\underline{x}]$ is well defined if ${\cal I}(t[\underline{x}])\downarrow$ is provable in $\tar$. Notice that the interpretations of terms in non-empty contexts are always well defined.

Notice moreover that in $\tar$
\begin{enumerate}
\item ${\cal I}(El_{N_{1}}(\star, t'))\simeq {\cal I}(t')$;
\item ${\cal I}(El_{N}(0,t,(y,z)s))\simeq {\cal I}(t)$;
\item ${\cal I}(El_{N}(\mathsf{succ}(t'),t,(y,z)s))\simeq {\cal I}(s)[{\cal I}(t')/y,{\cal I}(El_{N}(t',t,(y,z)s))/z]$
\item ${\cal I}(\mathsf{Ap}(\lambda y.s,t))\simeq {\cal I}(s)[{\cal I}(t)/y]$;
\item ${\cal I}(\mathsf{Ap}_{\rightarrow}(\lambda_{\rightarrow} y.s,t))\simeq {\cal I}(s)[{\cal I}(t)/y]$;
\item ${\cal I}(\mathsf{Ap}_{\forall}(\lambda_{\forall} y.s,t))\simeq {\cal I}(s)[{\cal I}(t)/y]$;
\item ${\cal I}(El_{\Sigma}(\langle t,t'\rangle,(y,z)r))\simeq {\cal I}(r)[{\cal I}(t)/y,{\cal I}(t')/z]$;
\item ${\cal I}(El_{\exists}(\langle t,_{\exists}t'\rangle,(y,z)r))\simeq {\cal I}(r)[{\cal I}(t)/y,{\cal I}(t')/z]$;
\item ${\cal I}(\pi^{\wedge}_{1}(\langle t,_{\wedge}t'\rangle))\simeq {\cal I}(t)$;
\item ${\cal I}(\pi^{\wedge}_{2}(\langle t,_{\wedge}t'\rangle))\simeq {\cal I}(t')$;
\item ${\cal I}(El_{+}(\mathsf{\mathsf{inl}}(t),(y)s,(y)s'))\simeq {\cal I}(s)[{\cal I}(t)/y]$;
\item ${\cal I}(El_{+}(\mathsf{inr}(t),(y)s,(y)s'))\simeq {\cal I}(s')[{\cal I}(t)/y]$;
\item ${\cal I}(El_{\vee}(\mathsf{\mathsf{inl}_{\vee}}(t),(y)s,(y)s'))\simeq {\cal I}(s)[{\cal I}(t)/y]$;
\item ${\cal I}(El_{\vee}(\mathsf{inr}_{\vee  }(t),(y)s,(y)s'))\simeq {\cal I}(s')[{\cal I}(t)/y]$;
\item ${\cal I}(El_{\mathsf{Id}}(t,\mathsf{id}(t),(y)s))\simeq {\cal I}(s)[{\cal I}(t)/y]$;
\item ${\cal I}(El_{List}(\epsilon,t',(y,z,u)q))\simeq {\cal I}(t')$;
\item ${\cal I}(El_{List}(\mathsf{cons}(t,t''),t',(y,z,u)q))\simeq {\cal I}(q)[{\cal I}(t)/y,{\cal I}(t'')/z,{\cal I}(El_{List}(t,t',(y,z,u)q))/u]$.
 
\end{enumerate}

\subsubsection*{The interpretation of sets}
Here we define the interpretation of sets in $\mtts$ with the exception of those obtained as $\tau(p)$ for some term $p$. Every such a set is interpreted as a definable quotient of a definable class of $\tar$ (and actually of $\mathsf{HA}$). This means that every set $A$ is interpreted as a pair
$${\cal I}(A)=(\,{\cal J}(A)\,,\,\sim_{{\cal I}(A)}\,)$$
where ${\cal J}(A)$ is a definable class of $\tar$ and $\sim_{{\cal I}(A)}$ is a definable equivalence relation on the class ${\cal J}(A)$.

Since sets in \mtt\ include small propositions, here we also define a realizability relation between natural numbers and propositions. Indeed it is more convenient to define the realizability interpretation of propositions by adopting
an extension of usual Kleene's interpretation of intuitionistic connectives.

Note that we use the notation ${\cal I}(A)[s/y]$ to mean the definable class in which we substitute $y$ with $s$ in the membership and in the equivalence relation of ${\cal I }(A)$.



\begin{definition}
\label{inteset}
We define in $\tar$ a realizability relation $n\,\Vdash\, \phi$ between natural numbers and small propositions, by induction on the definition of small propositions $\phi$, simultaneously together with the definition of the following
formulas $n\varepsilon {\cal J}(A)$  and $n\sim_{{\cal I}(A)}m$ for sets $A$, by induction on the definition of sets (with the exception of those obtained using $\tau(p)$ for some term $p$), as follows:
\\
\begin{enumerate}
\item[$(\bot)$] $n\,\Vdash\, \bot$ is $\bot$;
\item[$(\wedge)$] $n\,\Vdash\, \phi\wedge \phi'$ is $( p_{1}(n)\,\Vdash\, \phi)\,\wedge\, ( p_{2}(n)\,\Vdash\, \phi')$;
\item[$(\vee)$] $n\,\Vdash\, \phi\vee \phi'$ is $( p_{1}(n)=0\,\wedge\,  p_{2}(n)\,\Vdash\, \phi)\,\vee\, ( p_{1}(n)\neq 0\,\wedge \, p_{2}(n)\,\Vdash\, \phi')$; 
\item[$(\rightarrow)$] $n\,\Vdash\, \phi\rightarrow \phi'$ is $\forall t\,((t\,\Vdash\, \phi)\,\rightarrow\,(\{n\}(t)\,\Vdash\, \phi'))$;
\item[$(\exists)$] $n\,\Vdash\, (\exists x\in A)\,\psi$ is $ p_{1}(n)\,\varepsilon\, {\cal J}(A)\,\wedge\, ( p_{2}(n)\,\Vdash\, \psi)[ p_{1}(n)/x]$;
\item[$(\forall)$] $n\,\Vdash\, (\forall x\in A)\,\psi$ is $\forall x\,(x\,\varepsilon\, {\cal J}(A)\,\rightarrow\, (\{n\}(x)\,\Vdash\, \psi))$;
\item[$(\mathsf{Id})$] $n\,\Vdash\, \mathsf{Id}(A,t,s)$ is ${\cal I}(t)\sim_{{\cal I}(A)}{\cal I}(s)$;
\item[$(N_{0})$] $n\,\varepsilon\, {\cal J}(N_{0})$ is $\bot$ and
\item[] $n\sim_{{\cal I}(N_{0})}m$ is $\bot$;

\item[$(N_{1})$] $n\,\varepsilon\, {\cal J}(N_{1})$ is $n=0$ and
\item[] $n\sim_{{\cal I}(N_{1})}m$ is $n=0\wedge n=m$;

\item[$(N)$] $n\,\varepsilon\, {\cal J}(N)$ is $n=n$ and
\item[] $n\sim_{{\cal I}(N)}m$ is $n=m$;

\item[$(\Pi)$] $n\,\varepsilon\, {\cal J}((\Pi x\in A)\,B)$ is \\
$\forall x\,(x\,\varepsilon\, {\cal J}(A)\rightarrow \{n\}(x)\in {\cal J}(B))\wedge\forall x \forall y\,(x\sim_{{\cal I}(A)}y\rightarrow \{n\}(x)\sim_{{\cal I}(B)}\{n\}(y)\}$\footnote{Note that the variable $x$ may be in ${\cal I}(B)$ here and in the following definition for $\Pi$ and $\Sigma$ sets, as it comes from the definition of the untyped syntax.} and
\item[] $n\sim_{{\cal I}((\Pi x\in A)\,B)}m$ is\\ $n\,\varepsilon\, {\cal J}((\Pi x\in A)B)\wedge m\,\varepsilon\, {\cal J}((\Pi x\in A)B)\wedge \forall x\,(x\,\varepsilon\, {\cal J}(A)\rightarrow \{n\}(x)\sim_{{\cal I}(B)}\{m\}(x))$;
\item[$(\Sigma)$] $n\,\varepsilon\, {\cal J}((\Sigma x\in A)\,B)$ is $ p_{1}(n)\,\varepsilon\,{\cal J}(A)\wedge  \forall x\,(x\sim_{{\cal I}(\mathsf{A})}p_{1}(n)\rightarrow p_{2}(n)\,\varepsilon\, {\cal J}(B))$ and
\item[] $n\sim_{{\cal I}((\Sigma x\in A)\,B)}m$ is the conjunction of $n\,\varepsilon\, {\cal J}((\Sigma x\in A)B)\wedge m\,\varepsilon\, {\cal J}((\Sigma x\in A)B)$ and \\$ p_{1}(n)\sim_{{\cal I}(A)} p_{1}(m)\wedge \forall x\, (x\sim_{{\cal I}(A)}p_{1}(n)\rightarrow p_{2}(n)\sim_{{\cal I}(B)} p_{2}(m))$;

\item[$(+)$] $n\,\varepsilon\, {\cal J}(A+A')$ is $( p_{1}(n)=0\wedge  p_{2}(n)\,\varepsilon\, {\cal J}(A))\vee ( p_{1}(n)=1 \wedge  p_{2}(n)\,\varepsilon\, {\cal J}(A'))$ and \\$n\sim_{{\cal I}(A+A')}m$ is the conjunction of $n\,\varepsilon\, {\cal J}(A+A')\wedge m\,\varepsilon\, {\cal J}(A+A')\wedge p_{1}(n)= p_{1}(m)$ and \\  $( p_{1}(n)=0\,\wedge\,  p_{2}(n)\sim_{{\cal I}(A)} p_{2}(m))\,\vee\,( p_{1}(n)=1 \,\wedge\,  p_{2}(n)\sim_{{\cal I}(A')} p_{2}(m))$;

\item[$(List)$] $n\,\varepsilon\, {\cal J}(List(A))$ is $\forall j\,(j<lh(n)\,\rightarrow\, (n)_{j}\,\varepsilon\, {\cal J}(A))$ and
\item[] $n\sim_{{\cal I}(List(A))}m$ is the conjunction of $n\,\varepsilon\, {\cal J}(List(A))\,\wedge\, m\,\varepsilon\, {\cal J}(List(A))$ and \\$lh(n)=lh(m)\,\wedge\, \forall j\,(j<lh(n)\,\rightarrow\, (n)_{j}\sim_{{\cal I}(A)}(m)_{j})$;

\item[$(\psi)$] $n\,\varepsilon\, {\cal J}(\psi)$ is $n\Vdash \psi$ and
\item[] $n\sim_{{\cal I}(\psi)}m$ is $n\,\varepsilon\, {\cal J}(\psi)\,\wedge\, m\,\varepsilon\, {\cal J}(\psi)$ (i.\,e.\,{\it proof-irrelevance}).
\end{enumerate}
\end{definition}

\begin{remark}
We can notice some preliminary properties of this realizability interpretation:
\begin{enumerate}
\item for every set $A$ we have that $\sim_{{\cal I}(A)}$ is really a definable equivalence relation on the definable class ${\cal J}(A)$, in fact\\
\begin{enumerate}
\item[] $n\,\varepsilon\, {\cal J}(A)\,\vdash_{\tar}\,n\sim_{{\cal I}(A)}n$\\
\item[] $n\sim_{{\cal I}(A)}m\,\vdash_{\tar}\,m\sim_{{\cal I}(A)}n$\\
\item[] $n\sim_{{\cal I}(A)}m\,\wedge\, m\sim_{{\cal I}(A)}l\,\vdash_{\tar}\,n\sim_{{\cal I}(A)}l$\\
\end{enumerate}
\item for every set $A$ we have that 
\[n\sim_{{\cal I}(A)}m\,\vdash_{\tar}\,n\,\varepsilon\, {\cal J}(A)\wedge m\,\varepsilon\, {\cal J}(A)\]
\item if \emph{numerical} sets are defined according to the following conditions
\begin{enumerate}
\item $N_{0}$, $N_{1}$ and $N$ are numerical sets;
\item if $A$ and $B$ are numerical sets, then $(\Sigma x\in A)\,B$, $A+B$ and $List(A)$ (if they are well defined) are numerical sets,
\end{enumerate}
then the equality of the interpretation of numerical sets is numerical, which means that 
\[n\sim_{{\cal I}(A)}m\,\vdash_{\tar}\,n=m\]
\item for all propositions $\psi$, the equivalence relation $\sim_{{\cal I}(\psi)}$ is trivial (i.\,e.\,all pairs of elements of ${\cal I}(\psi)$ are equivalent). This means that uniqueness of propositional proofs, called {\it proof-irrelevance}, is imposed.


\end{enumerate}
\end{remark}

\subsubsection*{The encoding of all \mtts-sets}
In the previous sections we have seen the interpretation of \mtts-sets which include small propositions. It remains to define the interpretation of proper collections, including
that of sets, small propositions and small propositional functions on a set.

The interpretation of the collection of small propositions $\mathsf{Set}$ in \tar\ is the most difficult
point and to define it we mimick the technique adopted by Beeson~\cite{beeson}  to interpret Martin-L{\"o}f's universe via a fix point
of some arithmetical operator with positive parameters. Hence, it is to define the interpretation of $\mathsf{Set}$, and in turn of the collection of small propositions $\mathsf{prop_{s}}$ and of small propositional functions $A\rightarrow \mathsf{prop_s}$ on a set $A$, that we need to employ the full power of \tar\ with fix points.

The idea is to define a \tar-formula which defines codes of sets with their interpretation {\it as a fix point}.
It appears necessary to define called $\mathbf{Set}(n)$ expressing that $n$ is a code of an \mtts-set together with its realizability interpretation in \tar.
Observe that in \mtts\ the type of all sets is not present and hence no \mtts-type will be interpreted as $\{n|\,\mathbf{Set}(n)\}$.
As in Beeson's semantics, to define the formula $\mathbf{Set}(n)$ of set codes with their arithmetical interpretation in $\tar$ we need to encode membership and equality of sets: $t \overline{\varepsilon}n$ and
$t \,\equiv_n\, s$. In turn in order to define them, we need to represent the notion of {\it a family of sets} used to interpret
an $\mtts$-dependent set.


 {\it A family of sets} coded by $m$ on a set coded by $n$ could be described by the formula
\[\mathbf{Set}(n)\,\wedge\, \forall t\,(t\,\overline{\varepsilon}\, n\,\rightarrow\, \mathbf{Set}(\{m\}(t)))\,\wedge\,\]
\[\;\;\;\;\;\;\forall t\forall s\,(t\equiv_{n} s\,\rightarrow \,(\forall j\,(j\,\overline{\varepsilon}\, \{m\}(t)\,\leftrightarrow\, j\,\overline{\varepsilon}\, \{m\}(s))\,\wedge\, \forall j\forall k\,(j\equiv_{\{m\}(t)}k\,\leftrightarrow\, j\equiv_{\{m\}(s)}k))).\]
But in this formula not all occurrences of 
$t\,\overline{\varepsilon}\, n$ and $t\,\equiv_{n}\, s$ are positive. However it is classically equivalent to the conjunction of the formula 
$\mathbf{Set}(n)\,\wedge\,(\neg t\,\overline{\varepsilon}\, n\, \vee\, \mathbf{Set}(\{m\}(t)))$
and the formula $\forall t \forall s\,(\neg t\,\equiv_{n}\,s\,\vee\, ( P_{1}\,\wedge\, P_{2}))$
where $P_{1}$ is
$$\forall j\,((\neg j\,\overline{\varepsilon}\, \{m\}(t)\,\vee\, j\,\overline{\varepsilon}\,\{m\}(s))\,\wedge\, (\neg j\,\overline{\varepsilon}\, \{m\}(s)\,\vee\, j\,\overline{\varepsilon}\,\{m\}(t)))$$
and $P_{2}$ is 
$$\forall j\forall k\,((\neg j\,\equiv_{ \{m\}(t)}k\,\vee\, j\,\equiv_{\{m\}(s)}k)\,\wedge\,(\neg j\,\equiv_{ \{m\}(s)}k\,\vee\, j\,\equiv_{\{m\}(t)}k))$$
simply substituting all the instances of the schema $a\rightarrow b$ with the classically equivalent $\neg a\vee b$.
Now the trick consists in defining some predicates $t\noteps n$ and $t\,\not\equiv_{n}\,s$ mimicking the negations of $t\,\overline{\varepsilon}\, n$ and $t\,\equiv_{n}s$ as fix point predicates, too,  in order to get a a {\bf positive} arithmetical
operator. 
Note that the use of a {\bf classical} arithmetic theory with fix points seems unavoidable to be able to interpret the collection of sets via a positive arithmetical operator.

From now on we write 
$$\mathbf{Fam}(m,n) \, \ \equiv\, \  \Set{n}\,\wedge\, \forall t\,(t\noteps n\,\vee\, \Set{\{m\}(t)})\,\wedge\,\forall t \forall s\,(t\,\not\equiv_{n}s\,\vee\, (P'_{1}\,\wedge\, P'_{2}))$$
where $P'_{1}$ and $P'_{2}$ are obtained from $P_{1}$ and $P_{2}$ by substituting
negated istances of membership and of equality predicates with their mentioned primitive negated versions
$$P_1'\, \equiv\, \forall j\,((j\noteps \{m\}(t)\,\vee\, j\,\overline{\varepsilon}\, \{m\}(s))\,\wedge\,(j\noteps \{m\}(s)\,\vee\, j\,\overline{\varepsilon}\, \{m\}(t)))$$
$$P_2'\, \equiv\, \forall j\forall k\,(( j\,\not\equiv_{ \{m\}(t)}k\,\vee\, j\,\equiv_{\{m\}(s)}k)\,\wedge\, ( j\,\not\equiv_{ \{m\}(s)}k\,\vee\, j\,\equiv_{\{m\}(t)}k)).$$

In order to define the positive clauses for the codes of sets we must introduce some notations. In this way we transform the clauses for realizability for sets automatically in the clauses needed to define the fix points $\mathbf{Set}(n)$, $t\overline{\varepsilon} n$, $t\noteps n$, $t \equiv_{n}s$ and $t\not\equiv_{n}s$.

First of all, we define a function $[\;]$ which assigns a value to a set according to the table in section $3.2$ as follows. 
\begin{enumerate}
\item if $\sigma$ is one of the symbols $A,\,A',\,B,\,\phi,\,\phi',\,\psi,\,t,\,s$, then $[\sigma]$ is $a,\,a',\,\{b\}(x),\,c,\,c',\,\{d\}(x),\,e,\,f$ respectively;
\item if $\sigma$ is $N_{0}$, $N_{1}$, $N$, $(\Pi x\in A)\,B$, $(\Sigma x\in A)\,B$, $A+A'$, $List(A)$ then $[\sigma]$ is $\pair{1}{0}$, $\pair{1}{1}$, $\pair{1}{2}$, $\pair{2}{\pair{a}{b}}$, $\pair{3}{\pair{a}{b}}$, $\pair{4}{\pair{a}{a'}}$, $\pair{5}{a}$ respectively;
\item if $\sigma$ is $\bot$, $\phi\wedge \phi'$, $\phi\vee \phi'$, $\phi\rightarrow \phi'$, $(\exists x\in A)\,\psi$, $(\forall x\in A)\,\psi$, $\mathsf{Id}(A,t,s)$ then $[\sigma]$ is $\pair{6}{0}$,  $\pair{7}{\pair{c}{c'}}$, $\pair{8}{\pair{c}{c'}}$, $\pair{9}{\pair{c}{c'}}$, $\pair{10}{\pair{a}{d}}$, $\pair{11}{\pair{a}{d}}$, $\pair{12}{\pair{a}{\pair{e}{f}}}$ respectively.
\end{enumerate}
We denote by $[\;]^{-1}$ the inverse function of $[\;]$.
Now, all clauses in the realizability interpretation of sets are defined using formulas which are obtained starting from arithmetical formulas or primitive formulas with $\varepsilon$ or $\sim$, by using connectives, first order quantifiers or explicit instances of substitution in $x$. For such formulas $\varphi$ we define $\varphi^{+}$ as follows:\\
\begin{enumerate}
\item if $\varphi$ is arithmetical, then $\varphi^{+}$ is defined as $\varphi$ itself. If $\varphi$ is a primitive formulas with $\varepsilon$ or $\sim$ we will transform $\varepsilon\,{{\cal J}(\sigma)}$ and $\sim_{{\cal I}(\sigma)}$ in $\overline{\varepsilon}\,[\sigma]$ and $\equiv_{[\sigma]}$ respectively, in order to obtain $\varphi^{+}$;
\item $(\varphi[\alpha/x])^{+}$ is $\varphi^{+}[\alpha/x]$;
\item $(\varphi\wedge \varphi')^{+}$ is $\varphi^{+}\wedge \varphi'^{+}$;
\item $(\varphi\vee \varphi')^{+}$ is $\varphi^{+}\vee \varphi'^{+}$;
\item $(\varphi\rightarrow \varphi')^{+}$ is $\overline{\varphi^{+}}\vee \varphi'^{+}$;
\item $(\forall u\, \varphi)^{+}$ is $\forall u\, \varphi^{+}$ for every variable $u$;
\item $(\exists u\, \varphi)^{+}$ is $\exists u\, \varphi^{+}$ for every variable $u$;
\end{enumerate}

where $\overline{\varphi}$ is defined by the following clauses:
\begin{enumerate}
\item if $\varphi$ is an arithmetical formula $\overline{\varphi}$ is $\neg \varphi$;
\item if $\varphi$ is a relation between two terms through $\overline{\varepsilon}$, $\noteps$, $\equiv$ or $\not\equiv$, then $\overline{\varphi}$ is obtained by transforming them in $\noteps$, $\overline{\varepsilon}$, $\not\equiv$ or $\equiv$ respectively;
\item $\overline{\varphi\wedge \varphi'}$ is $\overline{\varphi}\vee \overline{\varphi'}$;
\item $\overline{\varphi\vee \varphi'}$ is $\overline{\varphi}\wedge \overline{\varphi'}$;
\item $\overline{\forall u\, \varphi}$ is $\exists u\, \overline{\varphi}$ for every variable $u$;
\item $\overline{\exists u\, \varphi}$ is $\forall u\, \overline{\varphi}$ for every variable $u$;
\end{enumerate}

We can now define the positive clauses we needed. For $\tau$ equal to $\pair{1}{0}$, $\pair{1}{1}$, $\pair{1}{2}$, $\pair{2}{\pair{a}{b}}$, $\pair{3}{\pair{a}{b}}$, $\pair{4}{\pair{a}{a'}}$, $\pair{5}{a}$, $\pair{6}{0}$,  $\pair{7}{\pair{c}{c'}}$, $\pair{8}{\pair{c}{c'}}$, $\pair{9}{\pair{c}{c'}}$, $\pair{10}{\pair{a}{d}}$, $\pair{11}{\pair{a}{d}}$, $\pair{12}{\pair{a}{\pair{e}{f}}}$ we have the following clauses\footnote{By $n\varepsilon {\cal J}([\tau]^{-1})$ and  $n\sim_{{\cal I}([\tau]^{-1})}m$ we mean the right-hand side of the respective clause in the realizability interpretation of sets, taking into account that for small propositions membership coincides with the realizability relation.}:\\
\begin{enumerate}
\item $\Set{\tau}$ if $\mathsf{Cond}(\tau)$;
\item $n\,\overline{\varepsilon}\, \tau$ if $\mathsf{Cond}(\tau)\,\wedge\, (n\,\varepsilon\, {\cal J}([\tau]^{-1}))^{+}$;
\item $n\noteps \tau $ if $\mathsf{Cond}(\tau)\,\wedge\, \overline{(n\,\varepsilon\,{\cal J}([\tau]^{-1}))^{+}}$;
\item $n\equiv_{\tau}m$ if $\mathsf{Cond}(\tau)\,\wedge\,  (n\sim_{{\cal I}([\tau]^{-1})}m)^{+}$;
\item $n\not\equiv_{\tau}m$ if $\mathsf{Cond}(\tau)\,\wedge\, \overline{(n\sim_{{\cal I}([\tau]^{-1})}m)^{+}}$;\\
\end{enumerate}
where $\mathsf{Cond}(\tau)$ is\\
\begin{enumerate}
\item $\top$ if $\tau$ has first component $1$ or $6$;
\item $\mathbf{Fam}(b,a)$ if $\tau$ has first component $2$ or $3$;
\item $\mathbf{Set}(a)\,\wedge\, \mathbf{Set}(a')$ if $\tau$ has first component $4$;
\item $\mathbf{Set}(a)$ if $\tau$ has first component $5$;
\item $\mathbf{Set}(c)\,\wedge\, \mathbf{Set}(c')\,\wedge\, \pi_{1}(c)>5\,\wedge\, \pi_{1}(c')>5$ if $\tau$ has first component $7$, $8$ or $9$;
\item $\mathbf{Fam}(d,a)\,\wedge\, \forall x\,(x\noteps a\,\vee\, \pi_{1}(\{d\}(x))>5)$ if $\tau$ has first component $10$, $11$;
\item $\mathbf{Set}(a)\,\wedge\, e\,\overline{\varepsilon}\, a\,\wedge\, f\,\overline{\varepsilon}\,a$ if $\tau$ has first component $12$.
\end{enumerate}

By sake of example we present here the clauses for codes of $\Pi$-sets.

\begin{enumerate}
\item[] $\Set{\pair{2}{\pair{a}{b}}}$ if $\mathbf{Fam}(b,a)$;
\item[] $n\,\overline{\varepsilon}\, \pair{2}{\pair{a}{b}}$ if $$\mathbf{Fam}(b,a)\,\wedge\,\forall x\,(x\noteps a \,\vee\, \{n\}(x)\,\overline{\varepsilon}\, \{b\}(x))\,\wedge\, \forall x \forall y\, (x\,\not\equiv_{a}\,y \,\vee\, \{n\}(x)\,\equiv_{\{b\}(x)}\{n\}(y));$$
\item[] $n\noteps\pair{2}{\pair{a}{b}}$ if $$\mathbf{Fam}(b,a)\,\wedge\,(\exists x\,(x\,\overline{\varepsilon}\, a \,\wedge\, \{n\}(x)\noteps \{b\}(x))\,\vee\,\exists x \exists y\, (x\,\equiv_{a}\,y\,\wedge\, \{n\}(x)\,\not\equiv_{\{b\}(x)}\{n\}(y)));$$
\item[] $n\,\equiv_{\pair{2}{\pair{a}{b}}}m$ if $$\mathbf{Fam}(b,a) \,\wedge\, n\,\overline{\varepsilon}\, \pair{2}{\pair{a}{b}}\,\wedge\, m\,\overline{\varepsilon}\, \pair{2}{\pair{a}{b}}\,\wedge\, \forall x\,(x\noteps a\,\vee\, \{n\}(x)\,\equiv_{\{b\}(x)} \{m\}(x));$$
\item[] $n\,\not\equiv_{\pair{2}{\pair{a}{b}}}m$ if  $$\mathbf{Fam}(b,a) \,\wedge\, (n\noteps \pair{2}{\pair{a}{b}}\,\vee\, m\noteps \pair{2}{\pair{a}{b}}\,\vee\, \exists x\,(x\,\overline{\varepsilon}\, a\,\wedge\, \{n\}(x)\,\not\equiv_{\{b\}(x)} \{m\}(x))$$
\end{enumerate}

The formulas $\mathbf{Set}(n)$, $t\overline{\varepsilon} n$, $t\noteps n$, $t\equiv_{n}s$ and  $t\not\equiv_{n}s$ are components of a predicate $P_{\theta}(n)$ defined in $\tar$ as a fix point of an operator $\theta(n,X)$ defined by glueing together the clauses expressing the code of each \mtts-set-constructor with its interpretation in $\tar$.


\subsubsection*{The interpretation of collections}
Here we extend the realizability relation, membership and equality in definition \ref{inteset} in order to interpret collections, propositions and the decoding operators.

\begin{definition}
$n\,\Vdash\, \phi$ between natural numbers and \mtt-propositions and formulas $n\varepsilon {\cal J}(D)$  and $n\sim_{{\cal I}(D)}m$ for collections $D$
are defined
by including all clauses in definition \ref{inteset} plus the following:
\begin{enumerate}
\item $n\,\Vdash\, \tau(p)$ and $n\,\varepsilon\,{\cal J}(\tau(p))$ are both given by $n\,\overline{\varepsilon}\,{\cal I}(p)$
\item[] $n\sim_{{\cal I}(\tau(p))}m$ is $n\,\varepsilon\,{\cal J}(\tau(p))\wedge m\,\varepsilon\,{\cal J}(\tau(p))$;\\
\item The realizability relation $n\Vdash \eta$ for propositions is completely analogous to the realizability relation for small propositions and the interpretation of propositions is given by the class of realizers equipped with the trivial equivalence relation;\\
\item $\Sigma$-collections are interpreted exactly in the same way as $\Sigma$-sets;\\
\item $n\varepsilon {\cal J}(\mathsf{Set})$ is $\mathbf{Set}(n)\,\wedge\, \forall t\,(t\,\overline{\varepsilon}\, n\leftrightarrow \neg t\,\noteps\, n)\,\wedge\,\forall t\forall s\,(t\,\equiv_{n}s\leftrightarrow \neg t\,\not\equiv_{n}s)$. This is because  
$\noteps$ and $\not\equiv$, which are defined by fix point, don't behave necessarily as negations of $\overline{\varepsilon}$ and $\equiv$ and hence we need to add $\forall t\,(t\,\overline{\varepsilon}\, n\leftrightarrow \neg t\,\noteps \,n)$ and $\forall t\forall s\,(t\,\equiv_{n}\,s\leftrightarrow \neg t\,\not\equiv_{n}\,s)$;
\item[] The interpretation of  $n\sim_{{\cal I}(\mathsf{Set})}m$ is 
$$n\,\varepsilon\,{\cal J}(\mathsf{Set})\,\wedge\,m\,\varepsilon\, {\cal J}(\mathsf{Set})\,\wedge\, \forall t\,(t\,\overline{\varepsilon}\, n \leftrightarrow t\,\overline{\varepsilon}\, m)\,\wedge\, \forall t\forall s\,(t\,\equiv_{n} s\leftrightarrow t\,\equiv_{m} s);$$
\item $n\,\varepsilon\, {\cal J}(\mathsf{prop_s})$ is $n\,\varepsilon\, {\cal J}(\mathsf{Set})\,\wedge\, \pi_{1}(n)>5\,\wedge\, \forall t \forall s\,(t\,\overline{\varepsilon}\, n \,\wedge\, s\,\overline{\varepsilon}\, n\leftrightarrow t\,\equiv_{n}s)$ (recall that small propositions are encoded with $\pi_{1}(n)>5$ and enjoy the proof-irrelevance); 
\item[] The interpretation of  $n\sim_{{\cal I}(\mathsf{prop_s})}m$ is $n\,\varepsilon\,{\cal J}(\mathsf{prop_s})\,\wedge\, m\,\varepsilon\, {\cal J}(\mathsf{prop_s})\,\wedge\, \forall t\,(t\,\overline{\varepsilon}\, n \leftrightarrow t\,\overline{\varepsilon}\, m)$;\\
\item $n\,\varepsilon\, {\cal J}(A\rightarrow \mathsf{prop_s})$ is  
 $\forall t\forall s\,(t\sim_{{\cal I}(A)}s\,\rightarrow\, \{n\}(t)\sim_{{\cal I}(\mathsf{prop_s})}\{n\}(s))$
\item[] and $n\sim_{{\cal I}(A\rightarrow \mathsf{prop_s})}m$ is \\ 
$n\,\varepsilon\, {\cal J}(A\rightarrow \mathsf{prop_s})\,\wedge\, m\,\varepsilon\, {\cal J}(A\rightarrow \mathsf{prop_s})\,\wedge\,\forall t\,(t\,\varepsilon\, {\cal J}(A)\,\rightarrow\, \{n\}(t)\sim_{{\cal I}(\mathsf{prop_s})}\{m\}(t))$
\end{enumerate}
\end{definition}

\subsubsection*{The interpretation of judgements}
We now need to say how judgements are interpreted in our realizability model. First of all, if ${\cal A}=(A,\,\simeq_{{\cal A}})$ and ${\cal B}=(B,\,\simeq_{{\cal B}})$ are definable classes of $\tar$ equipped with a definable equivalence relation, then we denote with ${\cal A}\doteq {\cal B}$ the formula $\forall t \forall s\,(t\simeq_{\cal A} s \,\leftrightarrow\, t\simeq_{\cal B}s).$

The judgements of $\mtts$ are interpreted as follows:\\
\begin{enumerate}
\item if $type\in \{set, col, prop_{s}, prop\}$, the interpretation of $A\; type$ is ${\cal I}(A)\doteq {\cal I}(A)$;
\item if $type\in \{set, col, prop_{s}, prop\}$, the interpretation of $A=B\; type$ is ${\cal I}(A)\doteq {\cal I}(B)$;
\item the judgement $t\in A$ is interpreted as ${\cal I}(t)$;
\item the judgement $t=s\in A$ is interpreted as ${\cal I}(t)\sim_{{\cal I}(A)} {\cal I}(s)$.\\
\item if $type\in \{set, col, prop_{s}, prop\}$, the interpretation of $A\; type\; [x_{1}\in A_{1},...,x_{\mathsf{n}}\in A_{\mathsf{n}}]$ is 
$$\forall x_{1}\forall y_{1}...\forall x_{\mathsf{n}}\forall y_{\mathsf{n}}\,(x_{1}\sim_{{\cal I}(A_{1})}y_{1}\,\wedge...\wedge\, x_{\mathsf{n}}\sim_{ {\cal I}(A_{\mathsf{n}})}y_{\mathsf{n}}\, \rightarrow\, {\cal I}(A)\doteq{\cal I}(A)\,[y_{1}/x_{1},...,y_{\mathsf{n}}/x_{\mathsf{n}}])~\footnote{Note that this definition and the following  exploit the fact that \mtt-variables are interpreted as themselves thought of as \tar-variables.  }$$
\item if $type\in \{set, col, prop_{s}, prop\}$, the interpretation of $A=B\; type\; [x_{1}\in A_{1},...,x_{\mathsf{n}}\in A_{\mathsf{n}}]$ is 
\[\forall x_{1}...\forall x_{\mathsf{n}}\,(x_{1}\,\varepsilon\, {\cal J}(A_{1})\,\wedge... \wedge\, x_{\mathsf{n}}\,\varepsilon \,{\cal J}(A_{\mathsf{n}}) \,\rightarrow\, {\cal I}(A)\doteq{\cal I}(B))\]
\item the judgement $t\in A[x_{1}\in A_{1},...,x_{\mathsf{n}}\in A_{\mathsf{n}}]$ is interpreted as 
\[\forall x_{1}\forall y_{1}...\forall x_{\mathsf{n}}\forall y_{\mathsf{n}}\,(x_{1}\sim_{{\cal I}(A_{1})}y_{1}\,\wedge...\wedge\, x_{\mathsf{n}}\sim_{{\cal I}(A_{\mathsf{n}})}y_{\mathsf{n}}\,\rightarrow\, {\cal I}(t)\sim_{{\cal I}(A)}{\cal I}(t)\,[y_{1}/x_{1},..,y_{\mathsf{n}}/x_{\mathsf{n}}])\]
\item the judgement $t=s\in A[x_{1}\in A_{1},....,x_{\mathsf{n}}\in A_{\mathsf{n}}]$ is interpreted as 
\[\forall x_{1}...\forall x_{\mathsf{n}}\, (x_{1}\,\varepsilon\, {\cal J}(A_{1})\,\wedge...\wedge\, x_{\mathsf{n}}\,\varepsilon\,{\cal J}(A_{\mathsf{n}})\,\rightarrow\, {\cal I}(t)\sim_{{\cal I}(A)} {\cal I}(s))\]
\end{enumerate}

\subsection{The validity theorem}

A judgement $J$ in the language of \mtts\ is validated by the realizability model (${\cal R}\vDash J$) if $\tar\vdash {\cal I}(J)$,
where ${\cal I}(J)$ is the interpretation of $J$ according to the previous section. 
We say that a proposition $\phi$ is validated by the model (${\cal R}\vDash \phi$), if its interpretation can be proven to be inhabited, which means that  
\[\tar\vdash \exists r(r\varepsilon {\cal J}(\phi)) \mbox{  which is equivalent to } \tar\vdash \exists r(r\Vdash \phi).\]

In order to prove how substitution is interpreted in a easy way, it is convenient to  modify the presentation
of \mtts-rules, into an equivalent system (still denoted by \mtts),
where we supply the information that the members in a type equality judgement
are types, and members of term equality judgements are typed terms  as follows
with the warning of avoiding repetitions of same judgements: for $type\in \{set, col, prop_{s}, prop\}$
\\

\begin{tabular}{l}
\normalsize any rule \ \
$\normalsize  \displaystyle{ \frac{\displaystyle J_{1}...J_{n}}
       {\displaystyle A=B\; type\;[\Gamma]}}\  \ $ 
 is changed to \ \ 
$\normalsize \displaystyle{ \frac{\displaystyle J_{1}...J_{n},\ A\; type\;[\Gamma]\ , \ B\; type\; [\Gamma] }{\displaystyle A=B\; type\; [\Gamma]}}
$
\end{tabular}

\begin{tabular}{l}
\normalsize any rule $\normalsize \  
\displaystyle{ \frac{\displaystyle J_{1}...J_{n}}{\displaystyle b\in B\;[\Gamma] }}$ \ \
  is changed to 
$\normalsize \ \ \displaystyle{ \frac
       {\displaystyle J_{1}...J_{n},\  B\; type\;[\Gamma]\  }
       {\displaystyle b\in B\; [\Gamma]}}$
\end{tabular}

\begin{tabular}{l}
\normalsize any rule $\ \normalsize
\displaystyle{ \frac{\displaystyle  J_{1}...J_{n}}{\displaystyle  a=b\in A\;[\Gamma] }}$
 is changed to  \ \
$\normalsize\displaystyle{  \frac{\displaystyle J_{1}...J_{n},\  a\in A\; type\;[\Gamma]\ , \ b\in A\; type\; [\Gamma] }{\displaystyle a=b\in A\; type\; [\Gamma]}}
$
\end{tabular}
\\

\begin{tabular}{l}
\normalsize the substitution rule $\mbox{subT)}$ and $\mbox{sub)}$ in \cite{m09} are changed to\\
$\normalsize
\begin{array}{l}
    \mbox{subT$_m$)}\,\displaystyle{ \frac
       {\displaystyle\begin{array}{l}
C(x_1,\dots,x_n)\ type\ 
 [\, x_1\in A_1,\,  \dots,\,  x_n\in A_n(x_1,\dots,x_{n-1})\, ]   \\[2pt]
a_1\in A_1,\ \dots \ , a_n\in A_n(a_1,\dots,a_{n-1})\qquad
b_1\in A_1,\  \dots \  ,b_n\in A_n(b_1,\dots,b_{n-1})\\[2pt]
a_1=b_1\in A_1\ \dots \ a_n=b_n\in A_n(a_1,\dots,a_{n-1})\end{array}}
         {\displaystyle 
 C(a_1,\dots,a_{n})=C( b_1,\dots, b_n) \ type}}
   \end{array} 
$
\end{tabular}
\\
\\

\begin{tabular}{l}
$\normalsize
\begin{array}{l}
      \mbox{sub$_m$)} \ \
\displaystyle{ \frac
         { \displaystyle 
\begin{array}{l}
 c(x_1,\dots, x_n)\in C(x_1,\dots,x_n)\ \
 [\, x_1\in A_1,\,  \dots,\,  x_n\in A_n(x_1,\dots,x_{n-1})\, ]   \\[2pt]
 C(x_1,\dots,x_n)\ type \
 [\, x_1\in A_1,\,  \dots,\,  x_n\in A_n(x_1,\dots,x_{n-1})\, ]\\[2pt] 
a_1\in A_1,\ \dots \ , a_n\in A_n(a_1,\dots,a_{n-1})\qquad 
b_1\in A_1,\  \dots \  ,b_n\in A_n(b_1,\dots,b_{n-1})\\[2pt]
a_1=b_1\in A_1\ \dots \ a_n=b_n\in A_n(a_1,\dots,a_{n-1})
\end{array}}
         {\displaystyle c(a_1,\dots,a_n)=c(b_1,\dots, b_n)\in
 C(a_1,\dots,a_{n})  }}
\end{array}     
$
\end{tabular}

\begin{tabular}{l}
\normalsize the formation rules $\mbox{F-}\Sigma)$,  $\mbox{F-}\exists )$ and
 \mbox{F-$\forall$)} are changed to
\\
\\
$ \mbox{F-}\Sigma ) \ \
\displaystyle{ \frac{\displaystyle \begin{array}{l} B\ col\\
  C(x)
         \hspace{.1cm} \ col \ [x\in B]\end{array}}
         {\displaystyle \Sigma_{x\in B} C(x)\hspace{.1cm} col }}
\quad
 \mbox{F-}\exists ) \ \
\displaystyle{ \frac
         { \displaystyle\begin{array}{l} B\ col\\ C(x)
         \hspace{.1cm} prop \ \ [x\in B]\ \end{array} }
         {\displaystyle  \exists_{x\in B} C(x)\hspace{.1cm} prop }}\quad
\mbox{F-}\forall) \ \
\displaystyle{ \frac
{\displaystyle \begin{array}{l} B\ col\\ C(x)\hspace{.1cm} prop\ [x\in B]\end{array}  }
{\displaystyle \forall_{x\in B} C(x)\hspace{.1cm} prop }}
$
\end{tabular}
\\

\begin{tabular}{l}
\normalsize  the elimination rules \mbox{  E-$\Pi$)} and \mbox{  E-$\forall$)}
\normalsize are changed to\\
\\
$\normalsize\begin{array}{l}
 \mbox{  E-${\Pi}_m$)}\ \
 \displaystyle{\frac{ \displaystyle
    \begin{array}{l} C(x)\ set\, [x\in B] \ \ \ C(b)\ set\\
b\in B \hspace{.3cm} f\in \Pi_{x\in B} C(x)\end{array} }
 {\displaystyle \mathsf{Ap}(f,b)\in C(b) }}
\qquad
\mbox{E-${\forall}_m$) }\ \
\displaystyle{ \frac
{\displaystyle\begin{array}{l}C(x)\ prop\, [x\in B]\ \ \ C(b)\ prop \\
 b\in B \hspace{.3cm} f\in \forall_{x\in B} C(x)\end{array}}
{\displaystyle \mathsf{Ap}_{\forall}(f,b)\in C(b) }}
\end{array}
$
\end{tabular}
\\

Note that each \mtts-type is a collection and therefore in deriving
a typed term $b\in B$ under a context the addition of the information
that the type $B$ is a collection in the premise is certainly valid.


\begin{theorem}[Validity theorem]
For every judgement $J$ in the language of \mtts\ , if $J$ can be proven in \mtts\ ($\mtts\vdash J$), then $J$ is validated by the model (${\cal R}\vDash J$).
\end{theorem}
\emph{Proof}.\\
In order to prove the validity theorem it is necessary to prove by induction on the height of the proof tree in \mtts\ these three facts at the same time:
\begin{enumerate}
\item for every judgement $J$ in the language of \mtts, if $\mtts\vdash J$ then ${\cal R}\vDash J$;
\item (substitution)
 If $\mtts\vdash C\,type\,[x_{1}\in A_{1},...,x_{\mathsf{n}}\in A_{\mathsf{n}}]$  for $type\in \{set, col, prop_{s}, prop\}$ for all 
\[\mtts\vdash a_{\mathsf{1}}\in A_{1}[y_{1}\in B_{1},...,y_{\mathsf{m}}\in B_{\mathsf{m}}],...,\]
\[\qquad\qquad\qquad\qquad\qquad\qquad\mtts\vdash a_{\mathsf{n}}\in A_{\mathsf{n}}[a_{1}/x_{1},...,a_{\mathsf{n}-1}/x_{\mathsf{n}-1}][y_{1}\in B_{1},...,y_{\mathsf{m}}\in B_{\mathsf{m}}],\] if ${\cal R}\vDash a_{\mathsf{1}}\in A_{1}[y_{1}\in B_{1},...,y_{\mathsf{m}}\in B_{\mathsf{m}}] $,...,
\[\qquad\qquad\qquad{\cal R}\vDash a_{\mathsf{n}}\in A_{\mathsf{n}}[a_{1}/x_{1},...,a_{\mathsf{n}-1}/x_{\mathsf{n}-1}][y_{1}\in B_{1},...,y_{\mathsf{m}}\in B_{\mathsf{m}}],\] then 
\[\tar\,\vdash\, \forall y_{1}...\forall y_{\mathsf{m}}\,(y_{1}\,\varepsilon\, {\cal J}(B_{1})\,\wedge...\wedge\, y_{\mathsf{m}}\,\varepsilon\, {\cal J}(B_{\mathsf{m}})\,\rightarrow\]
\[\qquad\qquad\qquad\qquad\qquad\qquad{\cal I}(C)\,[{\cal I}(a_{1})/x_{1},...,{\cal I}(a_{\mathsf{n}})/x_{\mathsf{n}}]\doteq {\cal I}(C\,[a_{1}/x_{1},...,a_{\mathsf{n}}/x_{\mathsf{n}}])\]
and if $\mtts\vdash c\in C[x_{1}\in A_{1},...,x_{\mathsf{n}}\in A_{\mathsf{n}}]$  for all 
\[\mtts\vdash a_{\mathsf{1}}\in A_{1}[y_{1}\in B_{1},...,y_{\mathsf{m}}\in B_{\mathsf{m}}],...,\]
\[\qquad\qquad\qquad\qquad\qquad\qquad\mtts\vdash a_{\mathsf{n}}\in A_{\mathsf{n}}[a_{1}/x_{1},...,a_{\mathsf{n}-1}/x_{\mathsf{n}-1}][y_{1}\in B_{1},...,y_{\mathsf{m}}\in B_{\mathsf{m}}],\] if ${\cal R}\vDash a_{\mathsf{1}}\in A_{1}[y_{1}\in B_{1},...,y_{\mathsf{m}}\in B_{\mathsf{m}}] $,...,
\[\qquad\qquad\qquad{\cal R}\vDash a_{\mathsf{n}}\in A_{\mathsf{n}}[a_{1}/x_{1},...,a_{\mathsf{n}-1}/x_{\mathsf{n}-1}][y_{1}\in B_{1},...,y_{\mathsf{m}}\in B_{\mathsf{m}}],\] then\[\tar\,\vdash\, \forall y_{1}...\forall y_{\mathsf{m}}\,(y_{1}\,\varepsilon {\cal J}(B_{1})\,\wedge...\wedge\, y_{\mathsf{m}}\,\varepsilon\, {\cal J}(B_{\mathsf{m}})\,\rightarrow\]
\[\;\;\;\;\;\;\;\;\;\;\;\;\;\;\;\;\;\;\;\;\;\;\,\;\;\;\;\;\;\;\;\;\;{\cal I}(c)\,[{\cal I}(a_{1})/x_{1},...,{\cal I}(a_{\mathsf{n}})/x_{\mathsf{n}}]\sim_{{\cal I}(C\,[a_{1}/x_{1},...,a_{\mathsf{n}}/x_{\mathsf{n}}])}{\cal I}(c\,[a_{1}/x_{1},...,a_{\mathsf{n}}/x_{\mathsf{n}}]).\]
\item (coding) If $\mtts\vdash B\,set\,[x_{1}\in A_{1},...,_{\mathsf{n}}\in A_{\mathsf{n}}]$, then 
\[\tar\vdash \forall x_{1}...\forall x_{\mathsf{n}}\,(x_{1}\,\varepsilon\, {\cal J}(A_{1})\,\wedge...\wedge\, x_{\mathsf{n}}\,\varepsilon\, {\cal J}(A_{\mathsf{n}})\,\rightarrow  \,\mathbf{Set}({\cal I}(\hat{B}))\,\wedge\,\]
\[\qquad\qquad\forall t\,(t\,\varepsilon\, {\cal J}(B)\leftrightarrow t\,\overline{\varepsilon}\,{\cal I}(\hat{B}))\,\wedge\, \forall t\, (\neg t\,\varepsilon\, {\cal J}(B)\leftrightarrow t\noteps{\cal I}(\hat{B}))\,\wedge\,\]
\[\qquad\qquad\qquad\forall t \forall s\,(t\sim_{{\cal I}(B)}s\leftrightarrow t\,\equiv_{{\cal I}(\hat{B})}s)\,\wedge\, \forall t\forall s\,(\neg t\sim_{{\cal I}(B)}s\leftrightarrow t\,\not\equiv_{{\cal I}(\hat{B})}s)).\]

\end{enumerate}
We will prove the statements in the case in which $[x_{1}\in A_{1},...,x_{\mathsf{n}}\in A_{\mathsf{n}}]$ is $[x\in A]$ and $[y_{1}\in B_{1},...,y_{m}\in B_{m}]$ is $[\,]$, as the more general case is analogous. The choice of the empty context for the terms which will be used in substitutions doesn't give any loss of generality, as terms are interpreted as terms, variables as itselves and so everything remains true up to universal closure.

%
%
%

\subsubsection*{The empty set}
\subsubsection*{Empty set Formation}As $\tar\vdash \forall x\forall x'(x\sim_{{\cal I}(A)}x'\rightarrow \forall t\forall s(\bot\leftrightarrow \bot))$, then we obtain that ${\cal R}\vDash N_{0}\, set\,[x\in A]$.
\subsubsection*{Empty set elimination} Suppose we derived in $\mtts$ the judgment $\mathsf{emp}_{0}(a)\in A[a/x]$ by elimination after having derived $a\in N_{0}$ and $A\, set\,[x\in N_{0}]$.
By inductive hypothesis on validity $\tar\vdash {\cal I}(a)\varepsilon{\cal J}(N_{0})$, which means that $\tar\vdash \bot$, from which, by \emph{ex falso quodlibet}, one can prove in $\tar$ the interpretation of the judgment $\mathsf{emp}_{0}(a)\in A[a/x]$.
\subsubsection*{Substitution for empty set elimination}
The substitution for elimination is trivial as $\mathsf{emp}_{0}(t)$ is always interpreted as a constant.
\subsubsection*{The singleton set}
\subsubsection*{Singleton formation}
As $\tar\vdash \forall x\forall x'(x\sim_{{\cal I}(A)}x'\rightarrow \forall t\forall s(t=s\wedge t=0\leftrightarrow t=s\wedge t=0))$, we can conclude that  ${\cal R}\vDash N_{1}\, set\,[x\in A]$.
\subsubsection*{Singleton introduction}
Let us prove the validity of the judgment $\star\in N_{1}[x\in A]$. Its interpretation is 
\[\forall x\forall x'(x\sim_{{\cal I}(A)}x'\rightarrow 0\sim_{{\cal I}(N_{1})}0)\]
which is by definition
\[\forall x\forall x'(x\sim_{{\cal I}(A)}x'\rightarrow 0=0)\]
which trivially holds in $\tar$. The substitution holds trivially, because $\star$ does not contain any variable.
\subsubsection*{Substitution for singleton introduction}
Substitution is trivial as ${\cal I}(\star)$ and $\star$ don't contain variables.
\subsubsection*{Singleton elimination} Suppose we derived in $\mtts$ the judgment $El_{N_{1}}(t,c)\in C[t/y][x\in A]$ by elimination after having derived $t\in N_{1}[x\in A]$, $c\in C[\star/y][x\in A]$ and $C\, set\,[x\in A, y\in N_{1}]$. By inductive hypothesis on validity we have that in $\tar$
\[(*)\,\forall x(x\varepsilon {\cal J}(A)\rightarrow {\cal I}(t)=0)\]
Using $(*)$ and the inductive hypothesis on substitution applied to $C$ with respect to \\
$x\in A[x\in A]$ and $t\in N_{1}[x\in A]$ (namely, that $\forall x(x\varepsilon {\cal J}(A)\rightarrow {\cal I}(C[t/y])={\cal I}(C)[{\cal I}(t)/y])$) we can derive in $\tar$ that 
\[(**)\,\forall x(x\varepsilon {\cal J}(A)\rightarrow \forall t\forall s(t\sim_{{\cal I}(C)[0/y]}s\leftrightarrow t\sim_{{\cal I}(C[t/y])}s))\]
and using the inductive hypothesis on substitution applied to $C$ with respect to $x\in A[x\in A]$ and $\star\in N_{1}[x\in A]$, we have that 
\[(***)\,\forall x(x\varepsilon {\cal J}(A)\rightarrow \forall t\forall s(t\sim_{{\cal I}(C)[0/y]}s\leftrightarrow t\sim_{{\cal I}(C[\star/y])}s)).\]
Using the inductive hypothesis on validity for $c$ we obtain that 
\[\forall x\forall x'(x\sim_{{\cal I}(a)}x'\rightarrow {\cal I}(c)\sim_{{\cal I}(C[\star/y])} {\cal I}(c)[x'/x])\]
and using $(**)$ and $(*)$ we obtain that in $\tar$
\[\forall x\forall x'(x\sim_{{\cal I}(a)}x'\rightarrow {\cal I}(c)\sim_{{\cal I}(C[t/y])} {\cal I}(c)[x'/x])\]
which is exactly  
\[\forall x\forall x'(x\sim_{{\cal I}(a)}x'\rightarrow {\cal I}(El_{N_{1}}(t,c))\sim_{{\cal I}(C[\star/y])} {\cal I}(El_{N_{1}}(t,c))[x'/x])\]
So we have that ${\cal R}\vDash El_{N_{1}}(t,c)\in C[t/y][x\in A]$.
\subsubsection*{Substitution for singleton elimination} In addition to the hypothesis of the previous point suppose that $\mtts\vdash a\in A$ and \\$\tar\vdash {\cal I}(a)\varepsilon {\cal J}(A)$.
By inductive hypothesis on substitution for $c$ and $t$ with respect to $a$ we have in $\tar$ that 
\[{(*)\, \cal I}(c)[{\cal I}(a)/x]\sim_{{\cal I}(C[\star/y][a/x])}{\cal I}(c[a/x])\]
\[(**)\,{\cal I}(t)[{\cal I}(a)/x]\sim_{{\cal I}(N_{1})}{\cal I}(t[a/x])\]
In particular from $(**)$ we obtain that ${\cal I}(t[a/x])=0$ in $\tar$. 
Using in $(*)$ the inductive hypothesis on substitution for $C$ with respect to $\star\in N_{1}$ and $a\in A$, and recalling that ${\cal I}(t[a/x])=0$, we have that 
\[{\cal I}(c)[{\cal I}(a)/x]\sim_{{\cal I}(C)[{\cal I}(a)/x,0/y]}{\cal I}(c[a/x]) \]
which is 
\[{\cal I}(c)[{\cal I}(a)/x]\sim_{{\cal I}(C)[{\cal I}(a)/x,{\cal I}(t[a/x])/y]}{\cal I}(c[a/x]) .\]
Now observing that $\mtts\vdash t[a/x]\in N_{1}$, and hence by inductive hypothesis \\${\cal R}\vDash t[a/x]\in N_{1}$, as $\tar\vdash {\cal I}(t[a/x])=0$, we have (using substitution for $C$ with respect to $a\in A$ and $t[a/x]\in N_{1}$) that 
\[{\cal I}(c)[{\cal I}(a)/x]\sim_{{\cal I}(C[a/x,t[a/x]/y])}{\cal I}(c[a/x]) \]
which exactly is 
\[{\cal I}(c)[{\cal I}(a)/x]\sim_{{\cal I}(C[t/y][a/x])}{\cal I}(c[a/x]) \]
and this is 
\[{\cal I}(El_{N_{1}}(t,c))[{\cal I}(a)/x]\sim_{{\cal I}(C[t/y][a/x])} {\cal I}(El_{N_{1}}(t,c)[a/x]).\]
\subsubsection*{Singleton conversion}  Suppose we derived in $\mtts$ the judgment $El_{N_{1}}(\star, c)=c[x\in A]$ by conversion after having derived the judgments $c\in C[\star/y][x\in A]$ and $C\, set\,[x\in A, y\in N_{1}]$. Then by inductive hypothesis on validity we have that in $\tar$ 
\[\forall x(x\varepsilon {\cal J}(A)\rightarrow {\cal I}(c)\sim_{{\cal I}(C[\star/y]) }{\cal I}(c)),\]
by the reflexivity of $\sim_{{\cal I}(A)}$; this is exactly equivalent to 
\[\forall x(x\varepsilon {\cal J}(A)\rightarrow {\cal I}(El_{N_{1}}(\star,c))\sim_{{\cal I}(C[\star/y]) }{\cal I}(c)).\]
So ${\cal R}\vDash El_{N_{1}}(\star, c)=c\in C[\star/y][x\in A]$.
\subsubsection*{The set of natural numbers}

\subsubsection*{Natural numbers formation} Formation is trivial as ${\cal I}(N)$ does not depend on any variable.
\subsubsection*{Natural numbers introduction} Let us check that ${\cal R}\vDash 0\in N[x\in A]$. This is trivial as in $\tar$, we have that 
\[\forall x\forall x'(x\sim_{{\cal I}(A)}x'\rightarrow 0=0).\]
Suppose now we derived in $\mtts$ the judgment $\mathsf{succ}(n)\in N[x\in A]$ by introduction after having derived $n\in N[x\in A]$. By using the inductive hypothesis on validity we deduce similarly that ${\cal R}\vDash \mathsf{succ}(n)\in N[x\in A]$.
\subsubsection*{Substitution for natural numbers introduction} The case of substitution for $0$ is trivial as $0$ does not contain variables. Suppose in addition to the hypothesis in the previous point that $\mtts\vdash a\in A$ and this is valid in  ${\cal R}$. By inductive hypothesis on substitution we have that 
\[\tar\vdash {\cal I}(n)[{\cal I}(a)/x]={\cal I}(n[a/x])\]
from which we derive
\[\tar\vdash succ({\cal I}(n)[{\cal I}(a)/x])=succ({\cal I}(n[a/x]))\]
which is exactly 
\[\tar\vdash {\cal I}(\mathsf{succ}(n))[{\cal I}(a)/x]={\cal I}(\mathsf{succ}(n)[a/x]). \]
\subsubsection*{Natural numbers elimination} Suppose we derived $El_{N}(n,a,(y,z)b)\in B[n/u][x\in A]$ in $\mtts$ by elimination after having derived
\begin{enumerate}
\item $B\, set\,[x\in A, u\in N]$, 
\item $a\in B[0/u][x\in A]$, 
\item $n\in N[x\in A]$,
\item $b\in B[\mathsf{succ}(y)/u][x\in A, y\in N, z\in B[y/u]]$.
\end{enumerate} 
Using the inductive hypotheses we will first prove that in $\tar$
\[\forall x\forall x'(x\sim_{{\cal I}(A)} x'\rightarrow \]
\[\qquad\qquad\qquad\forall u(\{\mathbf{rec}\}({\cal I}(a),\Lambda y.\Lambda z.{\cal I}(b),u)\sim_{{\cal I}(B)}\{\mathbf{rec}\}({\cal I}(a)[x'/x],\Lambda y.\Lambda z. {\cal I}(b)[x'/x],u) )).\]
We suppose $x\sim_{{\cal I}(A)}x'$ and we prove this by induction on $u$.
First of all by inductive hypothesis on validity for $a$ we have that 
\[{\cal I}(a)\sim_{{\cal I}(B[0/u])} {\cal I}(a)[x'/x]\]
which by definition is 
\[\{\mathbf{rec}\}({\cal I}(a),\Lambda y.\Lambda z.{\cal I}(b),0)\sim_{{\cal I}(B[0/u])}\{\mathbf{rec}\}({\cal I}(a)[x'/x],\Lambda y.\Lambda z. {\cal I}(b)[x'/x],0).\]
Using the inductive hypothesis on substitution for $B$ with respect to $x\in A[x\in A]$ and $0\in N[x\in A]$ we obtain that 
\[\{\mathbf{rec}\}({\cal I}(a),\Lambda y.\Lambda z.{\cal I}(b),0)\sim_{{\cal I}(B)[0/u]}\{\mathbf{rec}\}({\cal I}(a)[x'/x],\Lambda y.\Lambda z. {\cal I}(b)[x'/x]).\]
Suppose now that 
\[\{\mathbf{rec}\}({\cal I}(a),\Lambda y.\Lambda z.{\cal I}(b),u)\sim_{{\cal I}(B)}\{\mathbf{rec}\}({\cal I}(a)[x'/x],\Lambda y.\Lambda z. {\cal I}(b)[x'/x],u)\]
By inductive hypothesis on validity for $b$ we have that 
\[{\cal I}(b)[u/y,\{\mathbf{rec}\}({\cal I}(a),\Lambda y.\Lambda z.{\cal I}(b),u)/z]\sim_{{\cal I}(B[\mathsf{succ}(y)/u])[u/y,\{\mathbf{rec}\}({\cal I}(a),\Lambda y.\Lambda z.{\cal I}(b),u)/z]}\]
\[\qquad\qquad\qquad{\cal I}(b)[x'/x,u/y,\{\mathbf{rec}\}({\cal I}(a)[x'/x],\Lambda y.\Lambda z. {\cal I}(b)[x'/x],u)/z].\]

Using the inductive hypothesis on substitution for $B$ with respect to \\
$x\in A[x\in A,y\in N,z\in B[y/u]]$ and $\mathsf{succ}(y)\in N[x\in A,y\in N]$ and the inductive hypothesis on substitution for $B$ with respect to $x\in A[x\in A, y\in N]$ and $y\in N[x\in A, y\in N]$ we obtain that 
\[{\cal I}(b)[u/y,\{\mathbf{rec}\}({\cal I}(a),\Lambda y.\Lambda z.{\cal I}(b),u)/z]\sim_{{\cal I}(B)[succ(u)/u]}\]
\[\qquad\qquad\qquad{\cal I}(b)[x'/x,u/y,\{\mathbf{rec}\}({\cal I}(a)[x'/x],\Lambda y.\Lambda z. {\cal I}(b)[x'/x],u)/z]\]
which is by definition
\[\{\mathbf{rec}\}({\cal I}(a),\Lambda y.\Lambda z.{\cal I}(b),u+1)\sim_{{\cal I}(B)[succ(u)/u]}\{\mathbf{rec}\}({\cal I}(a)[x'/x],\Lambda y.\Lambda z. {\cal I}(b)[x'/x],u+1).\]
So we can conclude that 
\[\forall x\forall x'(x\sim_{{\cal I}(A)} x'\rightarrow \]
\[\qquad\qquad\qquad\forall u(\{\mathbf{rec}\}({\cal I}(a),\Lambda y.\Lambda z.{\cal I}(b),u)\sim_{{\cal I}(B)}\{\mathbf{rec}\}({\cal I}(a)[x'/x],\Lambda y.\Lambda z. {\cal I}(b)[x'/x],u) )).\]
Now considering that by inductive hypothesis on validity we have that 
\[\tar\vdash \forall x\forall x'(x\sim_{{\cal I}(A)}x'\rightarrow {\cal I}(n)={\cal I}(n)[x'/x])\]
we obtain that 
\[\forall x\forall x'(x\sim_{{\cal I}(A)} x'\rightarrow \{\mathbf{rec}\}({\cal I}(a),\Lambda y.\Lambda z.{\cal I}(b),{\cal I}(n))\sim_{{\cal I}(B)[{\cal I}(n)/u]}\]
\[\qquad\qquad\qquad\{\mathbf{rec}\}({\cal I}(a)[x'/x],\Lambda y.\Lambda z. {\cal I}(b)[x'/x],{\cal I}(n)[x'/x]) )\]
which is by definition
\[\forall x\forall x'(x\sim_{{\cal I}(A)} x'\rightarrow {\cal I}(El_{N}(n,a,(y,z)b))\sim_{{\cal I}(B)[{\cal I}(n)/u]}{\cal I}(El_{N}(n,a,(y,z)b))[x'/x]).\]
Now using the inductive hypothesis on substitution for $B$ with respect to $x\in A[x\in A]$ and $n\in N[x\in A]$ we have that 
\[\forall x\forall x'(x\sim_{{\cal I}(A)} x'\rightarrow {\cal I}(El_{N}(n,a,(y,z)b))\sim_{{\cal I}(B[n/u])}{\cal I}(El_{N}(n,a,(y,z)b))[x'/x])\]
which means that ${\cal R}\vDash El_{N}(n,a,(y,z)b)\in B[n/u]$.
\subsubsection*{Substitution for natural numbers elimination} We add to the hypotheses of the previous point, the hypothesis that $\bar{a}\in A$ is provable in $\mtts$ and valid in ${\cal R}$. 
By inductive hypothesis on substitution we have in $\tar$ that ${\cal I}(n)[{\cal I}(\bar{a})/x]={\cal I}(n[\bar{a}/x])$. 
We must prove by induction on $u$ that 
\[\forall u(\{\mathbf{rec}\}({\cal I}(a),\Lambda y.\Lambda z. {\cal I}(b),u)[{\cal I}(\bar{a})/x]\sim_{{\cal I}(B[\bar{a}/x])}\{\mathbf{rec}\}({\cal I}(a[\bar{a}/x]),\Lambda y.\Lambda z. {\cal I}(b[\bar{a}/x]),u) ).\]
We prove this statement in a way similar to that of the previous point, using the inductive hypotheses on substitution. 
From this we derive that 
\[\{\mathbf{rec}\}({\cal I}(a),\Lambda y.\Lambda z. {\cal I}(b),{\cal I}(n))[{\cal I}(\bar{a})/x]\sim_{{\cal I}(B[\bar{a}/x])[{\cal I}(n[\bar{a}/x])
/u]}\]
\[\qquad\qquad\qquad \{\mathbf{rec}\}({\cal I}(a[\bar{a}/x]),\Lambda y.\Lambda z. {\cal I}(b[\bar{a}/x]),{\cal I}(n[\bar{a}/x])) .\]
Using the inductive hypothesis on substitution applied to $\bar{a}\in A[u\in N]$ and $u\in N[u\in N]$ we obtain that 
\[\{\mathbf{rec}\}({\cal I}(a),\Lambda y.\Lambda z. {\cal I}(b),{\cal I}(n))[{\cal I}(\bar{a})/x]\sim_{{\cal I}(B)[{\cal I}(\bar{a})/x,{\cal I}(n[\bar{a}/x])
/u]}\]
\[\qquad\qquad\qquad \{\mathbf{rec}\}({\cal I}(a[\bar{a}/x]),\Lambda y.\Lambda z. {\cal I}(b[\bar{a}/x]),{\cal I}(n[\bar{a}/x])) .\]
So we have in $\tar$, using the inductive hypothesis on substitution for $B$ with respect to $\bar{a}\in A$ and $n[\bar{a}/x]\in N$, that 
\[\{\mathbf{rec}\}({\cal I}(a),\Lambda y.\Lambda z. {\cal I}(b),{\cal I}(n))[{\cal I}(\bar{a})/x]\sim_{{\cal I}(B[\bar{a}/x, n[\bar{a}/x]/y])
/u]} \]
\[\qquad\qquad\qquad\{\mathbf{rec}\}({\cal I}(a[\bar{a}/x]),\Lambda y.\Lambda z. {\cal I}(b[\bar{a}/x]),{\cal I}(n[\bar{a}/x])) .\]
This is equivalent to say that 
\[{\cal I}(El_{N}(n,a,(y,z)b))[{\cal I}(\bar{a})/x]\sim_{{\cal I}(B[n/y][\bar{a}/x])
/u]} {\cal I}(El_{N}(n,a,(y,z)b)[\bar{a}/x]) .\]
\subsubsection*{Natural numbers conversion} For conversion suppose that we derived $El_{N}(0,a,(y,z)b)=a\in B[0/u][x\in A]$ in $\mtts$ by conversion after having derived 
\begin{enumerate}
\item $B\, set\,[x\in A, u\in N]$, 
\item $a\in B[0/u][x\in A]$, 
\item $b\in B[\mathsf{succ}(y)/u][x\in A, y\in N, z\in B[y/u]]$. 
\end{enumerate}
By inductive hypothesis on validity we have that in $\tar$
\[\forall x(x\varepsilon {\cal J}(A)\rightarrow {\cal I}(a)\sim_{{\cal I}(B[0/y])}{\cal I}(a))\]
which by definition is 
\[\forall x(x\varepsilon {\cal J}(A)\rightarrow {\cal I}(El_{N}(0,a,(y,z)b))\sim_{{\cal I}(B[0/y])}{\cal I}(a)).\]
Suppose we derived $El_{N}(\mathsf{succ}(n),a,(y,z)b)\in B[\mathsf{succ}(n)/u][x\in A]$ in $\mtts$ by conversion after having derived the previous judgments and $n\in N[x\in A]$.
By inductive hypothesis on substitution applied to $b$ with respect to $x\in A[x\in A]$, $n\in N[x\in A]$ and \\
$El_{N}(n,a,(y,z)b)\in B[n/u][x\in A]$ we have that in $\tar$
\[\forall x(x\varepsilon {\cal J}(A)\rightarrow {\cal I}(b)[{\cal I}(n)/y, {\cal I}(El_{N}(n,a,(y,z)b))/z]\sim_{{\cal I}(B[\mathsf{succ}(n)/y])} {\cal I}(b[n/y,El_{N}(n,a,(y,z)b)/z]))\]
which is exactly
\[\forall x(x\varepsilon {\cal J}(A)\rightarrow {\cal I}(El_{N}(\mathsf{succ}(n),a,(y,z)b))\sim_{{\cal I}(B[\mathsf{succ}(n)/y])} {\cal I}(b[n/y,El_{N}(n,a,(y,z)b)/z]))\]
\subsubsection*{Dependent products} 
\subsubsection*{Dependent product formation}
Suppose that ${\cal R}\vDash C\, set\,[x\in A,y\in B]$ and ${\cal R}\vDash B\, set\,[x\in A]$, then 
\[\tar\vdash\forall x\forall x'\forall y\forall y'(x\sim_{{\cal I}(A)}x'\wedge y\sim_{{\cal I}(B)}y'\rightarrow \forall t\forall s(t\sim_{{\cal I}(C)}s\leftrightarrow t\sim_{{\cal I}(C)[x'/x,y'/y]}s))\]
\[\tar\vdash\forall x\forall x'(x\sim_{{\cal I}(A)}x'\rightarrow \forall t\forall s(t\sim_{{\cal I}(B)}s\leftrightarrow t\sim_{{\cal I}(B)[x'/x]}s))\]
From $x\sim_{\cal I}(A) x'$ we can also deduce in $\tar$ that
\[\forall y\forall y'(y\sim_{{\cal I}(B)}y'\rightarrow \{t\}(y)\sim_{{\cal I}(C)}\{t\}(y'))\leftrightarrow\forall y\forall y'(y\sim_{{\cal I}(B)[x'/x]}y'\rightarrow \{t\}(y)\sim_{{\cal I}(C)[x'/x]}\{t\}(y'))\]
\[\forall y\forall y'(y\sim_{{\cal I}(B)}y'\rightarrow \{s\}(y)\sim_{{\cal I}(C)}\{s\}(y'))\leftrightarrow\forall y\forall y'(y\sim_{{\cal I}(B)[x'/x]}y'\rightarrow \{s\}(y)\sim_{{\cal I}(C)[x'/x]}\{s\}(y'))\]
\[\forall y(y\varepsilon {\cal J}(B)\rightarrow \{t\}(y)\sim_{{\cal I}(C)}\{s\}(y))\leftrightarrow\forall y(y\varepsilon {\cal J}(B)[x'/x]\rightarrow \{t\}(y)\sim_{{\cal I}(C)[x'/x]}\{s\}(y))\]
which means that $\tar\vdash \forall x\forall x'(x\sim_{{\cal I}(A)}x'\rightarrow \forall t\forall s(t\sim_{{\cal I}((\Pi y\in B)C)}s\leftrightarrow t\sim_{{\cal I}((\Pi y\in B)C)[x'/x]}s))$ that is ${\cal R}\vDash (\Pi y\in B)C\, set\,[x\in A]$. 
\subsubsection*{Substitution for dependent product formation}
In addition to the hypotheses that $\mtts\vdash B\, set\,[x\in A]$ and $\mtts\vdash C\, set\,[x\in A, y\in B]$, suppose that $\mtts\vdash a\in A$ and $\tar\vdash {\cal I}(a)\varepsilon{\cal J}(A)$.  Then by inductive hypothesis on substitution for $B$ with respect to $a$ and for $C$ with respect to $a\in A[y\in B[a/x]]$ and $y\in B[a/x][y\in B[a/x]]$ we have that 
\[\tar\vdash \forall t\forall s(t\sim_{{\cal I}(B)[{\cal I}(a)/x]}s\leftrightarrow t\sim_{{\cal I}(B[a/x])}s)\]
and
\[\tar\vdash \forall y(y\varepsilon {\cal J}(B[a/x])\rightarrow  \forall t\forall s(t\sim_{{\cal I}(C)[{\cal I}(a)/x]}s\leftrightarrow t\sim_{{\cal I}(C[a/x])}s).\]
We can deduce in $\tar$ that for every $t$ and $s$
\[t\sim_{{\cal I}(\Pi y\in B)C[{\cal I}(a)/x]}s\] 
is equivalent to 
\[\forall y\forall y'(y\sim_{{\cal I}(B)[{\cal I}(a)/x]}y'\rightarrow \{t\}(y)\sim_{{\cal I}(C)[{\cal I}(a)/x]} \{t\}(y'))\wedge\]
\[\forall y\forall y'(y\sim_{{\cal I}(B)[{\cal I}(a)/x]}y'\rightarrow \{s\}(y)\sim_{{\cal I}(C)[{\cal I}(a)/x]} \{s\}(y'))\wedge\]
\[\forall y(y\varepsilon {\cal J}(B)[{\cal I}(a)/x]\rightarrow \{t\}(y)\sim_{{\cal I}(C)[{\cal I}(a)/x]}\{s\}(y))\]
which is equivalent to 
\[\forall y\forall y'(y\sim_{{\cal I}(B[a/x])}y'\rightarrow \{t\}(y)\sim_{{\cal I}(C[a/x])} \{t\}(y'))\wedge\]
\[\forall y\forall y'(y\sim_{{\cal I}(B[a/x])}y'\rightarrow \{s\}(y)\sim_{{\cal I}(C[a/x])} \{s\}(y'))\wedge\]
\[\forall y(y\varepsilon {\cal J}(B[a/x])\rightarrow \{t\}(y)\sim_{{\cal I}(C[a/x])}\{s\}(y))\]
which is exactly $t\sim_{{\cal I}((\Pi y\in B)C)[a/x]}s$.
\subsubsection*{Dependent product introduction} Suppose we obtained $\mtts\vdash (\lambda y)c\in (\Pi y\in B)C$ by introduction after having proved \\$\mtts\vdash c\in C[x\in A, y\in B]$, $\mtts\vdash C\,set\,[x\in A, y\in B]$ and\\ $\mtts\vdash B\,set\,[x\in A]$.
By inductive hypothesis on validity we have that
\[\forall x\forall x'\forall y\forall y'(x\sim_{{\cal I}(A)}x'\wedge y\sim_{{\cal I}(B)}y'\rightarrow {\cal I}(c)\sim_{{\cal I}(C)}{\cal I}(c)[x'/x,y'/y])\]
which is equivalent to 
\[\forall x\forall x'(x\sim_{{\cal I}(A)}x'\rightarrow\forall y\forall y' (y\sim_{{\cal I}(B)}y'\rightarrow \{\Lambda y.{\cal I}(c)\}(y)\sim_{{\cal I}(C)}\{\Lambda y.{\cal I}(c)[x'/x]\}(y')))\]
which is equivalent to 
\[\forall x\forall x'(x\sim_{{\cal I}(A)}x'\rightarrow\forall y\forall y' (y\sim_{{\cal I}(B)}y'\rightarrow \{{\cal I}((\lambda y)c)\}(y)\sim_{{\cal I}(C)}\{{\cal I}((\lambda y)c)[x'/x]\}(y')))\]
From this it follows that \begin{enumerate}
\item $x\sim_{{\cal I}(A)}x$ implies that $\forall y\forall y' (y\sim_{{\cal I}(B)}y'\rightarrow \{{\cal I}((\lambda y)c)\}(y)\sim_{{\cal I}(C)}\{{\cal I}(\lambda y.c)\}(y'))$, which means that ${\cal I}((\lambda y)c)\varepsilon {\cal I}((\Pi y\in B)C)$;
\item $x'\sim_{{\cal I}(A)}x'$ implies that 
\[\forall y\forall y' (y\sim_{{\cal I}(B)}y'\rightarrow \{{\cal I}((\lambda y)c)[x'/x]\}(y)\sim_{{\cal I}(C)[x'/x]}\{{\cal I}((\lambda y)c[x'/x]\}(y')),\] which means that 
${\cal I}((\lambda y)c)[x'/x]\varepsilon {\cal I}((\Pi y\in B)C)$;
\item using the fact that $y\varepsilon {\cal J}(B)$ entails $y\sim_{{\cal I}(B)}y$ we have that \[\forall x\forall x'(x\sim_{{\cal I}(A)}x'\rightarrow\forall y(y\varepsilon{\cal J}(B)\rightarrow \{{\cal I}((\lambda y)c)\}(y)\sim_{{\cal I}(C)}\{{\cal I}((\lambda y)c)[x'/x]\}(y)).\]
This gives us that ${\cal R}\vDash (\lambda y)c\in (\Pi y\in B)C[x\in A]$.
\end{enumerate}
\subsubsection*{Substitution for dependent product introduction}For substitution in addition to the hypotheses of the previous point, suppose that 
\begin{enumerate}
\item $\mtts\vdash a\in A$ and 
\item $\tar\vdash {\cal I}(a)\varepsilon {\cal J}(A)$. 
\end{enumerate}
From this by inductive hypothesis on substitution with respect to $a\in A\,[y\in B[a/x]]$ and \\$y\in B[a/x]\,[y\in B[a/x]]$ we obtain that 
\[\forall y(y\varepsilon {\cal J}(B[a/x])\rightarrow {\cal I}(c)[{\cal I}(a)/x]\sim_{{\cal I}(C[a/x])} {\cal I}(c[a/x]))\]
which entails that 
\[\forall y(y\varepsilon {\cal J}(B[a/x])\rightarrow \{{\cal I}(\lambda y.c)[{\cal I}(a)/x]\}(y)\sim_{{\cal I}(C[a/x]))} \{{\cal I}((\lambda y)c[a/x])\}(y))\]
Moreover by the inductive hypothesis on validity for $c$ and the validity hypothesis for $a\varepsilon A$, we obtain that 
\[\forall y\forall y'(y\sim_{{\cal I}(B)[{\cal I}(a)/x]}y'\rightarrow \{{\cal I}(\lambda y.c)[{\cal I}(a)/x]\}(y)\sim_{{\cal I}(C)[{\cal I}(a)/x]} \{{\cal I}(\lambda y.c)[{\cal I}(a)/x]\}(y'))\]
and using the substitution hypothesis for $B$ and $C$ we obtain that 
\[\forall y\forall y'(y\sim_{{\cal I}(B[a/x])}y'\rightarrow \{{\cal I}(\lambda y.c)[{\cal I}(a)/x]\}(y)\sim_{{\cal I}(C[a/x])} \{{\cal I}(\lambda y.c)[{\cal I}(a)/x]\}(y')).\]
By using the inductive hypothesis on substitution for $c$, for $C$, the validity hypothesis on $c\in C[x\in A, y\in B]$ and the previous one we obtain also that 
\[\forall y\forall y'(y\sim_{{\cal I}(B[a/x])}y'\rightarrow \{{\cal I}(\lambda y.c[a/x])\}(y)\sim_{{\cal I}(C[a/x])} \{{\cal I}(\lambda y.c[a/x])\}(y')).\]
This entails that 
\[\tar\vdash {\cal I}((\lambda y)c)[{\cal I}(a)/x]\sim_{{\cal I}((\Pi y\in B)C[a/x])} {\cal I}((\lambda y)c[a/x]).\]

\subsubsection*{Dependent product elimination} Suppose we derived in $\mtts$, the judgment $\mathsf{Ap}(f,b)\in C[b/y][x\in A]$ by elimination after having derived in $\mtts$ the judgments $b\in B[x\in A]$ and $f\in (\Pi y\in B)C[x\in A]$. By inductive hypothesis on validity we have that in $\tar$
\[\forall x\forall x'(x\sim_{{\cal I}(A)}x'\rightarrow {\cal I}(b)\sim_{{\cal I}(B)} {\cal I}(b)[x'/x])\]
\[\forall x\forall x'(x\sim_{{\cal I}(A)}x'\rightarrow \]
\[\qquad\forall y\forall y'(y\sim_{{\cal I}(B)}y'\rightarrow \{{\cal I}(f)\}(y)\sim_{{\cal I}(C)} \{{\cal I}(f)\}(y'))\wedge \]
\[\qquad\qquad\forall y\forall y'(y\sim_{{\cal I}(B)}y'\rightarrow \{{\cal I}(f)[x'/x]\}(y)\sim_{{\cal I}(C)} \{{\cal I}(f)[x'/x]\}(y'))\wedge\]
\[\qquad\qquad\qquad\forall y(y\varepsilon {\cal J}(B)\rightarrow \{{\cal I}(f)\}(y)\sim_{{\cal I}(C)} \{{\cal I}(f)[x'/x]\}(y) )).\]
From these we can easily deduce that in $\tar$
\[\forall x\forall x'(x\sim_{{\cal I}(A)}x'\rightarrow \{{\cal I}(f)\}({\cal I}(b))\sim_{{\cal I}(B)}\{{\cal I}(f)[x'/x]\}({\cal I}(b)[x'/x])).\]
This is exactly  \[\forall x\forall x'(x\sim_{{\cal I}(A)}x'\rightarrow {\cal I}(\mathsf{Ap}(f,b))\sim_{{\cal I}(C)[{\cal I}(b)/y]}{\cal I}(\mathsf{Ap}(f,b))[x'/x] ).\]
Now, applying the inductive hypothesis on substitution for $C$ with respect to $b\in B[x\in A]$ and $x\in A[x\in A]$, we obtain that 
\[\forall x\forall x'(x\sim_{{\cal I}(A)}x'\rightarrow {\cal I}(\mathsf{Ap}(f,b))\sim_{{\cal I}(C[b/y])}{\cal I}(\mathsf{Ap}(f,b))[x'/x] ).\]
So in particular we have that 
\[{\cal R}\vDash \mathsf{Ap}(f,b)\in C[b/y][x\in A].\]
\subsubsection*{Substitution for dependent product elimination} In addition to the hypotheses of the previous point, we add the hypotheses that \\
$\mtt\vdash a\in A$ and $\tar\vdash {\cal I}(a)\varepsilon {\cal J}(A)$.
By inductive hypothesis on substitution we have in $\tar$ that 
\[{\cal I}(b)[{\cal I}(a)/x]\sim_{{\cal I}(B[a/x])} {\cal I}(b[a/x])\]
\[\qquad\forall y\forall y'(y\sim_{{\cal I}(B[a/x])}y'\rightarrow \{{\cal I}(f)[{\cal I}(a)/x]\}(y)\sim_{{\cal I}(C[a/x])} \{{\cal I}(f)[{\cal I}(a)/x]\}(y'))\wedge \]
\[\qquad\qquad\forall y\forall y'(y\sim_{{\cal I}(B[a/x])}y'\rightarrow \{{\cal I}(f[a/x])\}(y)\sim_{{\cal I}(C[a/x])} \{{\cal I}(f[a/x])\}(y'))\wedge\]
\[\qquad\qquad\qquad\forall y(y\varepsilon {\cal J}(B[a/x])\rightarrow \{{\cal I}(f)[{\cal I}(a)/x]\}(y)\sim_{{\cal I}(C[a/x])} \{{\cal I}(f[a/x])\}(y) )).\]
From this we can deduce in $\tar$ that 
\[\{{\cal I}(f)[{\cal I}(a)/x]\}({\cal I}(b)[{\cal I}(a)/x])\sim_{{\cal I}(C[a/x])} \{{\cal I}(f[a/x])\}({\cal I}(b[a/x]))\]
which is 
\[{\cal I}(\mathsf{Ap}(f,b))[{\cal I}(a)/x]\sim_{{\cal I}(C[a/x])} {\cal I}(\mathsf{Ap}(f,b)[a/x])\]

\subsubsection*{Dependent product conversion}
Suppose we derived in $\mtts$ the judgment $\mathsf{Ap}(\lambda y. c,b)=c[b/y][x\in A]$ by  conversion after having derived $\mtts\vdash c\in C[x\in A, y\in B]$ and $\mtts\vdash b\in B[x\in A]$. By inductive hypothesis on validity we can suppose that these two judgments are validated by ${\cal R}$ and we can use the inductive hypothesis on substitution applied to $c$ with respect to $x\in A[x\in A]$ and $b\in B[x\in A]$, obtaining that in $\tar$
\[\forall x(x\varepsilon {\cal J}(A)\rightarrow {\cal I}(c)[{\cal I}(b)/y]\sim_{{\cal I}(C[b/y])} {\cal I}(c[b/y]))\]
which is exactly 
\[\forall x(x\varepsilon {\cal J}(A)\rightarrow {\cal I}(\mathsf{Ap}(\lambda y.c,b))\sim_{{\cal I}(C[b/y])} {\cal I}(c[b/y])).\]
So we have that ${\cal R}\vDash \mathsf{Ap}(\lambda y. c,b)=c[b/y][x\in A]$.

\subsubsection*{Dependent sum sets}
\subsubsection*{Dependent sum formation} Suppose that ${\cal R}\vDash C\, set\,[x\in A,y\in B]$ and ${\cal R}\vDash B\, set\,[x\in A]$, then 
\[\tar\vdash\forall x\forall x'\forall y\forall y'(x\sim_{{\cal I}(A)}x'\wedge y\sim_{{\cal I}(B)}y'\rightarrow \forall t\forall s(t\sim_{{\cal I}(C)}s\leftrightarrow t\sim_{{\cal I}(C)[x'/x,y'/y]}s))\]
\[\tar\vdash\forall x\forall x'(x\sim_{{\cal I}(A)}x'\rightarrow \forall t\forall s(t\sim_{{\cal I}(B)}s\leftrightarrow t\sim_{{\cal I}(B)[x'/x]}s))\]
From $x\sim_{{\cal I}(A)}x'$, we can also deduce in $\tar$ that
\[ p_{1}(t)\sim_{{\cal I}(B)} p_{1}(s)\wedge \forall y(y\sim_{{\cal I}(B)} p_{1}(t)\rightarrow  p_{2}(t)\sim_{{\cal I}(C)} p_{2}(s))\leftrightarrow\]
\[\qquad\qquad\qquad  p_{1}(t)\sim_{{\cal I}(B)[x'/x]} p_{1}(s)\wedge \forall y(y\sim_{{\cal I}(B)[x'/x]} p_{1}(t)\rightarrow  p_{2}(t)\sim_{{\cal I}(C)[x'/x]} p_{2}(s)) \] 
which means that 
\[\tar\vdash \forall x\forall x'(x\sim_{{\cal I}(A)}x'\rightarrow \forall t\forall s(t\sim_{{\cal I}((\Sigma y\in B)C)}s\leftrightarrow t\sim_{{\cal I}((\Sigma y\in B)C)[x'/x]}s))\]
 that is ${\cal R}\vDash (\Sigma y\in B)C\, set\,[x\in A]$. 

\subsubsection*{Substitution for dependent sum formation} In addition the hypotheses that $\mtts\vdash B\, set\,[x\in A]$ and $\mtts\vdash C\, set\,[x\in A, y\in B]$, suppose that $\mtts\vdash a\in A$ and $\tar\vdash {\cal I}(a)\varepsilon{\cal J}(A)$. Then by inductive hypothesis on substitution for $B$ with respect to $a$ and for $C$ with respect to $a\in A[y\in B[a/x]]$ and $y\in B[a/x][y\in B[a/x]]$ we have that 
\[\tar\vdash \forall t\forall s(t\sim_{{\cal I}(B)[{\cal I}(a)/x]}s\leftrightarrow t\sim_{{\cal I}(B[a/x])}s)\]
and
\[\tar\vdash \forall y(y\varepsilon {\cal J}(B[a/x])\rightarrow  \forall t\forall s(t\sim_{{\cal I}(C)[{\cal I}(a)/x]}s\leftrightarrow t\sim_{{\cal I}(C[a/x])}s).\]
From these we can immediately deduce that 
\[\tar\vdash {\cal I}((\Sigma y\in B)C)[{\cal I}(a)/x]\doteq {\cal I}((\Sigma y\in B)C[a/x]),\] as $\sim_{{\cal I}(\Sigma y\in B)C}$ only depends on ${\cal I}(B)$ and ${\cal I}(C)$.

\subsubsection*{Dependent sum introduction}
Suppose we derived in $\mtts$, $\langle b,c\rangle \in (\Sigma y\in B)C[x\in A]$ by introduction after having derived in $\mtts$, $b\in B[x\in A]$, $c\in C[b/y][x\in A]$ and $C\,col[x\in A,y\in B]$. Suppose $x\sim_{{\cal I}(A)}x'$. By inductive hypothesis on validity applied to $b\in B[x\in A]$, we obtain in $\tar$ that ${\cal I}(b)\sim_{{\cal I}(B)}{\cal I}(b)[x'/x]$ (and so also ${\cal I}(b)\varepsilon {\cal J}(B)$ and ${\cal I}(b)[x'/x]\in {\cal J}(B)$). By inductive hypothesis on the validity of \\$C\,col[x\in A,y\in B]$, if we have that $y\sim_{{\cal I}(B)}{\cal I}(b)$, we obtain that ${\cal I}(C)\doteq {\cal I}(C)[{\cal I}(b)/y]$. By inductive hypothesis on substitution applied to $C\,col[x\in A,y\in B]$, $b\in B[x\in A]$ and $x\in A[x\in A]$, we have that 
${\cal I}(C)[{\cal I}(b)/y]\doteq {\cal I}(C[b/y])$. Now by inductive hypothesis on validity on $c\in C[b/y][x\in A]$ we have ${\cal I}(c)\sim_{{\cal I}(C[b/y])} {\cal I}(c)[x'/x]$. From this and the previous remarks we derive that ${\cal I}(c)\sim_{{\cal I}(C)} {\cal I}(c)[x'/x]$ (and so moreover ${\cal I}(c)\varepsilon {\cal J}(C)$ and ${\cal I}(c)[x'/x]\varepsilon {\cal J}(C)$).  

Moreover, by what we said before, $y\sim_{{\cal I}(B)}{\cal I}(b)$ is equivalent to $y\sim_{{\cal I}(B)}{\cal I}(b)[x'/x]$.

This means that
\[\tar\vdash \forall x\forall x'(x\sim_{{\cal I}(A)}x'\rightarrow\]
\[\qquad\qquad\qquad {\cal I}(b)\sim_{{\cal I}(B)}{\cal I}(b)[x'/x]\wedge \forall y(y\sim_{{\cal I}(B)}{\cal I}(b)\rightarrow {\cal I}(c)\sim_{{\cal I}(C)}{\cal I}(c)[x'/x])).\]
This exactly means that ${\cal R}\vDash \langle b,c\rangle\in (\Sigma y\in B)C[x\in A]$.

\subsubsection*{Substitution for dependent sum introduction}
For substitution, under the same hypothesis, suppose that 
\begin{enumerate}
\item $\mtts\vdash a\in A$ and 
\item $\tar\vdash {\cal I}(a)\varepsilon {\cal J}(A)$.
\end{enumerate}
By the rules of $\mtts$, we know also that $\mtts\vdash C[b/y]\,col\,[x\in A]$ and \\
$\mtts \vdash B\,col\,[x\in A]$.
By inductive hypothesis on substitution of $a$ in $b$, we have that \[{\cal I}(b)[{\cal I}(a)/x]\sim_{{\cal I}(B[a/x])} {\cal I}(b[a/x]).\] 

By inductive hypothesis on substitutions of $a$ in $C[b/y]$, of $b[x\in A]$ and $x\in A[x\in A]$ in $C$, of $a\in A[y\in B[a/x]]$ and $y\in B[a/x][y\in B[a/x]]$ in $C$ we obtain that 
\[(*)\,{\cal I}(C[b/y][a/x])\doteq {\cal I}(C[b/y])[{\cal I}(a)/x]\doteq {\cal I}(C)[{\cal I}(b)/y][{\cal I}(a)/x]\doteq\]
\[\qquad\qquad\qquad {\cal I}(C)[{\cal I}(a)/x][{\cal I}(b)[{\cal I}(a)/x]/y]\doteq {\cal I}(C[a/x])[{\cal I}(b)[{\cal I}(a)/x]/y].\]
Suppose that $y\sim_{{\cal I}(B[a/x])}{\cal I}(b)[{\cal I}(a)/x]$. By inductive hypothesis on substitution for $c$,  
\[{\cal I}(c)[{\cal I}(a)/x]\sim_{{\cal I}(C[b/y][a/x])}{\cal I}(c[a/x]).\] By the inductive hypothesis on substitution 
\[{\cal I}(B)[{\cal I}(a)/x]\doteq {\cal I}(B[a/x])\] (and so $y\sim_{{\cal I}(B)[{\cal I}(a)/x]}{\cal I}(b)[{\cal I}(a)/x]$), and using the inductive hypothesis on validity for $C\,set$ we obtain that ${\cal I}(C)[{\cal I}(a)/x, {\cal I}(b)[{\cal I}(a)/x]/y]\doteq {\cal I}(C)[{\cal I}(a)/x]$ and so using $(*)$ 
we derive that 
\[{\cal I}(c)[{\cal I}(a)/x]\sim_{{\cal I}(C)[a/x])}{\cal I}(c[a/x])\]
after having noticed that ${\cal I}(C[a/x])\doteq {\cal I}(C)[{\cal I}(a)/x]$ by using the inductive hypothesis on substitution for $C$ with respect to $a\in A[y\in B[a/x]]$ and $y\in B[a/x][y\in B[a/x]]$.
So we obtained what we needed.
\subsubsection*{Dependent sum elimination}
Suppose we deduced $El_{\Sigma}(d, (y,z)e)\in E[d/u][x\in A]$ in $\mtts$, after having derived 
\begin{enumerate}
\item $E\, set\,[x\in A,u\in (\Sigma y\in B)C]$, 
\item $d\in (\Sigma y\in B)C[x\in A]$, 
\item $e\in E[\langle y,z\rangle][x\in A,y\in B,z\in C]$ and so also 
\item $B\, set\,[x\in A]$, 
\item $C\, set\,[x\in A, y\in B]$ by the structure of the rules of $\mtts$.
\end{enumerate}
 By inductive hypothesis on validity we can suppose that all these judgments are valid in ${\cal R}$. In particular we have that ${\cal R}\vDash d\in(\Sigma y\in B)C[x\in A]$. This in particular implies that in $\tar$
\[(*)\,\forall x\forall x'(x\sim_{{\cal I}(A)}x'\rightarrow \]
\[\qquad\qquad\qquad( p_{1}({\cal I}(d))\sim_{{\cal I}(B)}  p_{1}({\cal I}(d)[x'/x])\wedge  p_{2}({\cal I}(d))\sim_{{\cal I}(C)[ p_{1}({\cal I}(d))/y]}  p_{2}({\cal I}(d)[x'/x]) ) ).\]
Moreover we have that ${\cal R}\vDash e\in E[\langle y,z\rangle][x\in A,y\in B,z\in C] $. This in particular implies that in $\tar$
\[(**)\,\forall x\forall x'\forall y\forall y'\forall z\forall z'(x\sim_{{\cal I}(A)}x'\wedge y\sim_{{\cal I}(B)}y'\wedge z\sim_{{\cal I}(C)}z'\rightarrow\]
\[\qquad\qquad\qquad\qquad\qquad\qquad {\cal I}(e)\sim_{{\cal I}(E[\langle y,z\rangle/u])} {\cal I}(e)[x'/x,y'/y,z'/z]).\]
Putting together $(*)$ and $(**)$ we obtain that 
\[\forall x\forall x'(x\sim_{{\cal I}(A)}x'\rightarrow \]
\[\qquad\qquad\qquad{\cal I}(e)[ p_{1}({\cal I}(d))/y,  p_{2}({\cal I}(d))/z]\sim_{{\cal I}(E[\langle y,z\rangle/u])[ p_{1}({\cal I}(d))/y,  p_{2}({\cal I}(d))/z]}\]
\[\qquad\qquad\qquad\qquad\qquad\qquad{\cal I}(e)[x'/x, p_{1}({\cal I}(d)[x'/x])/y, p_{2}({\cal I}(d)[x'/x])/z])\]
and this exactly means that 
\[(***)\,\forall x\forall x'(x\sim_{{\cal I}(A)}x'\rightarrow \]
\[{\cal I}(El_{\Sigma}(d,(y,z)e))\sim_{{\cal I}(E[\langle y,z\rangle/u])[ p_{1}({\cal I}(d))/y,  p_{2}({\cal I}(d))/z]} {\cal I}(El_{\Sigma}(d,(y,z)e))[x'/x]).\]
Now it is immediate to see that 
\[\mtts\vdash \langle y,z\rangle \in (\Sigma y\in B)C[x\in A,y\in B, z\in C]\]
\[\mtts\vdash x\in A[x\in A,y\in B, z\in C]\]
and these judgments are validated by ${\cal R}$, as $C\, set\,[x\in A, y\in B]$ is valid in ${\cal R}$.
In particular, by the inductive hypothesis on substitution applied to $E$ we have that 
\[\forall x\forall y\forall z(x\in {\cal J}(A)\wedge y\in {\cal J}(B)\wedge z\in {\cal J}(C)\rightarrow {\cal I}(E)[\langle y,z\rangle/u]\doteq{\cal I}(E[\langle y,z\rangle/u])).\]
Using $(*)$ we immediately obtain, from the previous, that 
\[\forall x(x\in {\cal J}(A)\rightarrow {\cal I}(E)[{\cal I}(d)/u]\doteq{\cal I}(E[\langle y,z\rangle/u])[ p_{1}({\cal I}(d))/y, p_{2}({\cal I}(d))/z]).\]
Moreover as we have that 
\[\mtts\vdash d\in (\Sigma y\in B)C[x\in A]\]
\[\mtts\vdash x\in A[x\in A]\]
and both these judgments are valid in ${\cal R}$ (the first by inductive hypothesis), 
we also have that 
\[\forall x(x\in {\cal J}(A)\rightarrow{\cal I}(E)[{\cal I}(d)/u]\doteq {\cal I}(E[d/u]))\]
which combined with what we proved before gives us that 
\[\forall x(x\in {\cal J}(A)\rightarrow {\cal I}(E[d/u])\doteq{\cal I}(E[\langle y,z\rangle/u])[ p_{1}({\cal I}(d))/y, p_{2}({\cal I}(d))/z]).\]
Recalling $(***)$ we can conclude that in $\tar$
\[\forall x\forall x'(x\sim_{{\cal I}(A)}x'\rightarrow {\cal I}(El_{\Sigma}(d,(y,z)e))\sim_{{\cal I}(E[d/u])} {\cal I}(El_{\Sigma}(d,(y,z)e))[x'/x])\]
which exactly means that 
\[{\cal R}\vDash El_{\Sigma}(d,(y,z)e)\in E[d/u][x\in A].\]

\subsubsection*{Substitution for dependent sum elimination}
Under the same hypothesis as in the previous point suppose we have $\mtts\vdash a\in A$ and \\$\tar\vdash {\cal I}(a)\in {\cal J}(A)$.
First of all by inductive hypothesis on substitution applied to $d$ we can derive that in $\mtts$
\[(i)\, p_{1}({\cal I}(d)[{\cal I}(a)/x])\sim_{{\cal I}(B[a/x])} p_{1}({\cal I}(d[a/x]))\wedge\]
\[\qquad\qquad\qquad p_{2}({\cal I}(d)[{\cal I}(a)/x])\sim_{{\cal I}(C[a/x])[p_{1}({\cal I}(d)[{\cal I}(a)/x])/y]} p_{2}({\cal I}(d[a/x])).\]
Using inductive hypothesis on substitution for $B$ and $C$ this is equivalent to 
\[p_{1}({\cal I}(d)[{\cal I}(a)/x])\sim_{{\cal I}(B)[{\cal I}(a)/x]} p_{1}({\cal I}(d[a/x]))\wedge \]
\[\qquad\qquad\qquad p_{2}({\cal I}(d)[{\cal I}(a)/x])\sim_{{\cal I}(C)[{\cal I}(a)/x, p_{1}({\cal I}(d)[{\cal I}(a)/x])/y]}  p_{2}({\cal I}(d[a/x])). \]
Using $(**)$ from the previous section we obtain that 
\[(ii)\,{\cal I}(e)[{\cal I}(a)/x,  p_{1}({\cal I}(d)[{\cal I}(a)/x])/y,  p_{2}({\cal I}(d)[{\cal I}(a)/x])/z]\sim_{{\cal I}(E[\langle y,z\rangle ])[{\cal I}(a)/x,  p_{1}({\cal I}(d)[{\cal I}(a)/x])/y,  p_{2}({\cal I}(d)[{\cal I}(a)/x])/z]}\]
\[\qquad\qquad\quad {\cal I}(e)[{\cal I}(a)/x,  p_{1}({\cal I}(d[a/x]))/y,  p_{2}({\cal I}(d[a/x]))/z].\]
Now we can easily see that the following judgments are derivable in $\mtts$ and valid in ${\cal R}$:
\[a\in A[y\in B[a/x], z\in C[a/x] ]\]
\[y\in B[a/x][y\in B[a/x],z\in C[a/x]]\]
\[z\in C[a/x][y\in B[a/x], z\in C[a/x]],\]
so by inductive hypothesis on substitution applied to $e$ we have that 
\[\forall y\forall z(y\varepsilon {\cal J}(B[a/x])\wedge z\varepsilon {\cal J}(C[a/x])\rightarrow {\cal I}(e)[{\cal I}(a)/x]\sim_{{\cal I}(E[\langle y,z\rangle/u][a/x])} {\cal I}(e[a/x])).\]
By using $(i)$ we obtain that in $\tar$
\[(iii)\,{\cal I}(e)[{\cal I}(a)/x][  p_{1}({\cal I}(d[a/x]))/y,  p_{2}({\cal I}(d[a/x]))/z]\sim_{{\cal I}(E[\langle y,z\rangle/u][a/x])[ p_{1}({\cal I}(d[a/x]))/y,  p_{2}({\cal I}(d[a/x]))/z]} \]
\[\qquad\qquad\qquad{\cal I}(e[a/x]))[  p_{1}({\cal I}(d[a/x]))/y,  p_{2}({\cal I}(d[a/x]))/z].\]

Combining $(ii)$ and $(iii)$ and using the inductive hypothesis on substitution applied to $E$ in a way similar to that of the previous subsection we obtain that 

\[{\cal I}(e)[{\cal I}(a)/x,  p_{1}({\cal I}(d)[{\cal I}(a)/x])/y,  p_{2}({\cal I}(d)[{\cal I}(a)/x])/z]\sim_{{\cal I}(E[d/u][a/x])}\]
\[\qquad\qquad\qquad {\cal I}(e[a/x]))[  p_{1}({\cal I}(d[a/x]))/y,  p_{2}({\cal I}(d[a/x]))/z]\]
which is exactly 
\[{\cal I}(El_{\Sigma}(d,(y,z)e))[{\cal I}(a)/x]\sim_{{\cal I}(E[d/u][a/x])} {\cal I}(El_{\Sigma}(d,(y,z)e)[a/x]).\]

\subsubsection*{Dependent sum conversion}
Suppose we derived in $\mtts$ the judgment $El_{\Sigma}(\langle b, c\rangle, (y,z)e)=e[b/y,c/z][x\in A]$ by conversion after having derived 
\begin{enumerate}
\item $b\in B[x\in A]$, 
\item $c\in C[b/y][x\in A]$, 
\item $E\, set\,[x\in A, u\in (\Sigma y\in B)C]$ and 
\item $e\in E[\langle y,z\rangle/u][x\in A, y\in B, z\in C]$. 
\end{enumerate}
By inductive hypothesis on validity we have that ${\cal R}\vDash b\in B[x\in A]$ and ${\cal R}\vDash c\in C[b/y][x\in A]$. Moreover the judgment $x\in A[x\in A]$ is provable in $\mtts$ and valid in ${\cal R}$. So we can apply the inductive hypothesis on substitution to $e$ obtaining that in $\tar$
\[\forall x(x\in {\cal J}(A)\rightarrow {\cal I}(e)[{\cal I}(b)/y,{\cal I}(c)/z]\sim_{{\cal I}(E[\langle b,c\rangle/u])}{\cal I}(e[b/y,c/z]))\]
and this is exactly 
\[\forall x(x\in {\cal J}(A)\rightarrow {\cal I}(El_{\Sigma}(\langle b,c\rangle),(y,z)e)\sim_{{\cal I}(E[\langle b,c\rangle/u]}{\cal I}(e[b/y,c/z])).\]

\subsubsection*{The binary sum set}
\subsubsection*{Binary sum formation} Suppose that ${\cal R}\vDash B\, set\,[x\in A]$ and ${\cal R}\vDash C\, set\,[x\in A]$, then 
\[\tar\vdash\forall x\forall x'(x\sim_{{\cal I}(A)}x'\rightarrow \forall t\forall s(t\sim_{{\cal I}(B)}s\leftrightarrow t\sim_{{\cal I}(B)[x'/x]}s));\]
\[\tar\vdash\forall x\forall x'(x\sim_{{\cal I}(A)}x'\rightarrow \forall t\forall s(t\sim_{{\cal I}(C)}s\leftrightarrow t\sim_{{\cal I}(C)[x'/x]}s)).\]
From $x\sim_{\cal I}(A) x'$ we can also deduce in $\tar$ that
\[ p_{1}(t)= p_{1}(s)\wedge (( p_{1}(t)=0\wedge  p_{2}(t)\sim_{{\cal I}(B)} p_{2}(s))\vee( p_{1}(t)=1\wedge  p_{2}(t)\sim_{{\cal I}(C)} p_{2}(s)))\leftrightarrow\]
\[\qquad p_{1}(t)= p_{1}(s)\wedge (( p_{1}(t)=0\wedge  p_{2}(t)\sim_{{\cal I}(B)[x'/x]} p_{2}(s))\vee( p_{1}(t)=1\wedge  p_{2}(t)\sim_{{\cal I}(C)[x'/x]} p_{2}(s)))\]
which means that $\tar\vdash \forall x\forall x'(x\sim_{{\cal I}(A)}x'\rightarrow \forall t\forall s(t\sim_{{\cal I}(B+C)}s\leftrightarrow t\sim_{{\cal I}(B+C)[x'/x]}s))$ that is ${\cal R}\vDash B+C\, set\,[x\in A]$. 

\subsubsection*{Substitution for binary sum formation} Suppose that in $\mtts$ we derived $B+C\,set\,[x\in A]$ by formation after having derived $B\, set\,[x\in A]$ and $C\, set\,[x\in A]$, and suppose that $\mtts\vdash a\in A$ and $\tar\vdash {\cal I}(a)\varepsilon {\cal J}(A)$. Then by inductive hypothesis on substitution for $B$ and $C$ with respect to $a$ we obtain that in $\tar$ 
\[\forall t\forall s(t\sim_{{\cal I}(B)[{\cal I}(a)/x]}s\leftrightarrow t\sim_{{\cal I}(B[a/x])}s));\]
\[\forall t\forall s(t\sim_{{\cal I}(C)[{\cal I}(a)/x]}s\leftrightarrow t\sim_{{\cal I}(C[a/x])}s));\]
and from these it follows that \[{\cal I}(B+C)[{\cal I}(a)/x]\doteq {\cal I}(B+C[a/x])\]
as $\sim_{{\cal I}(\Sigma y\in B)C}$ only depends on ${\cal I}(B)$ and ${\cal I}(C)$.

%
%
\subsubsection*{Binary sum introduction} Suppose $\mathsf{inl}(b)\in B+C[x\in A]$ is derived by introduction after having derived in $\mtts$ the judgment $b\in B[x\in A]$. Then we can suppose by inductive hypothesis on validity that
\[\tar\vdash\forall x \forall x'(x\sim_{{\cal I}(A)}x'\rightarrow {\cal I}(b)\sim_{{\cal I}(B)}{\cal I}(b)[x'/x])\]
from which it comes that 
\[\tar\vdash\forall x \forall x'(x\sim_{{\cal I}(A)}x'\rightarrow  p_{1}(\langle 0, {\cal I}(b)\rangle)= p_{1}(\langle 0, {\cal I}(b)[x'/x]\rangle)\wedge\]
\[\qquad\qquad  p_{1}(\langle 0, {\cal I}(b)\rangle )=0\wedge  p_{2}(\langle 0,{\cal I}(b)\rangle)\sim_{{\cal I}(B)} p_{2}(\langle 0,{\cal I}(b)[x'/x]\rangle))\]
which entails that ${\cal R}\vDash \mathsf{inl}(b)\in B+C[x\in A]$.
A similar reasoning holds for $\mathsf{inr}(c)$.
\subsubsection*{Substitution for binary sum introduction}
For substitution, the substitutions for $\mathsf{inl}(b)$ or $\mathsf{inr}(c)$ directly come from the inductive hypothesis on substitution for $b$ and $c$ respectively.
\subsubsection*{Binary sum elimination} Suppose we derived $El_{+}(d,(y)b,(y)c)\in C[d/z][x\in A]$ in $\mtts$ by elimination after having derived 
\begin{enumerate}
\item $d\in B+C[x\in A]$, 
\item $C\, set\,[x\in A,z\in B+C]$, 
\item $b\in C[\mathsf{inl}(y)/z][x\in A, y\in B]$ and 
\item $x\in C[\mathsf{inr}(y)/z][x\in A, y\in C]$. 
\end{enumerate}
By inductive hypothesis on validity, in $\tar$, if we assume $x\sim_{{\cal I}(A)}x'$ we have that 
\[( p_{1}({\cal I}(d))= p_{1}({\cal I}(d)[x'/x])=0\wedge  p_{2}({\cal I}(d))\sim_{{\cal I}(B)} p_{2}({\cal I}(d)[x'/x]))\vee\]
\[( p_{1}({\cal I}(d))= p_{1}({\cal I}(d)[x'/x])=1\wedge  p_{2}({\cal I}(d))\sim_{{\cal I}(C)} p_{2}({\cal I}(d)[x'/x]));\]
\[\forall y\forall y'(y\sim_{{\cal I}(B)} y'\rightarrow  {\cal I}(b)\sim_{{\cal I}(C[\mathsf{inl}(y)/z])} {\cal I}(b)[x'/x,y'/y]);\]
\[\forall y\forall y'(y\sim_{{\cal I}(B)} y'\rightarrow  {\cal I}(c)\sim_{{\cal I}(C[\mathsf{inr}(y)/z])} {\cal I}(c)[x'/x,y'/y]).\]
Using these three conditions we immediately obtain that 
\[( p_{1}({\cal I}(d))= p_{1}({\cal I}(d)[x'/x])=0\wedge\]
\[\qquad {\cal I}(b)[ p_{2}({\cal I}(d))/y]\sim_{{\cal I}(C[\mathsf{inl}(y)/z])[ p_{2}({\cal I}(d))/y]}{\cal I}(b)[x'/x, p_{2}({\cal I}(d)[x'/x])/y])\vee\]
\[\qquad\qquad( p_{1}({\cal I}(d))= p_{1}({\cal I}(d)[x'/x])=1\wedge \]
\[\qquad\qquad\qquad{\cal I}(c)[ p_{2}({\cal I}(d))/y]\sim_{{\cal I}(C[\mathsf{inr}(y)/z])[ p_{2}({\cal I}(d))/y]}{\cal I}(c)[x'/x, p_{2}({\cal I}(d)[x'/x])/y]).\]
Using the inductive hypothesis on substitution for $C$ with respect to 
\[x\in A[x\in A, y\in A],\,\mathsf{inl}(y)\in B+C[x\in A, y\in B]\] and with respect to 
\[x\in A[x\in A, y\in A],\,\mathsf{inl}(y)\in B+C[x\in A, y\in B]\]
 we obtain that 
\[( p_{1}({\cal I}(d))= p_{1}({\cal I}(d)[x'/x])=0\wedge \]
\[\qquad{\cal I}(b)[ p_{2}({\cal I}(d))/y]\sim_{{\cal I}(C)[{\cal I}(d)/u]}{\cal I}(b)[x'/x, p_{2}({\cal I}(d)[x'/x])/y])\vee\]
\[\qquad\qquad( p_{1}({\cal I}(d))= p_{1}({\cal I}(d)[x'/x])=1\wedge\]
\[\qquad\qquad\qquad {\cal I}(c)[ p_{2}({\cal I}(d))/y]\sim_{{\cal I}(C)[{\cal I}(d)/u]}{\cal I}(c)[x'/x, p_{2}({\cal I}(d)[x'/x])/y]).\]
Using the inductive hypothesis on substitution for $C$ with respect to $x\in A[x\in A]$ and \\$d\in B+C[x\in A]$ we obtain that 
\[( p_{1}({\cal I}(d))= p_{1}({\cal I}(d)[x'/x])=0\wedge\]
\[\qquad {\cal I}(b)[ p_{2}({\cal I}(d))/y]\sim_{{\cal I}(C[d/u])}{\cal I}(b)[x'/x, p_{2}({\cal I}(d)[x'/x])/y])\vee\]
\[\qquad\qquad( p_{1}({\cal I}(d))= p_{1}({\cal I}(d)[x'/x])=1\wedge\]
\[\qquad\qquad\qquad {\cal I}(c)[ p_{2}({\cal I}(d))/y]\sim_{{\cal I}(C[d/u])}{\cal I}(c)[x'/x, p_{2}({\cal I}(d)[x'/x])/y]).\]
By definition this is equivalent to 
\[{\cal I}(El_{+}(d,(y)b,(y)c))\sim_{{\cal I}(C[d/u])}{\cal I}(El_{+}(d,(y)b,(y)c))[x'/x].\]
This is what we needed.

\subsubsection*{Substitution for binary sum elimination} Suppose we derived $El_{+}(d,(y)b,(y)c)\in C[d/z][x\in A]$ in $\mtts$ by elimination after having derived 
\begin{enumerate}
\item $d\in B+C[x\in A]$, 
\item $C\, set\,[x\in A,z\in B+C]$, 
\item $b\in C[\mathsf{inl}(y)/z][x\in A, y\in B]$ and 
\item $x\in C[\mathsf{inr}(y)/z][x\in A, y\in C]$.
\end{enumerate}
Suppose moreover that $\mtts\vdash a\in A$ and $\tar\vdash {\cal I}(a)\varepsilon {\cal J}(A)$. Recall that by the structure of rules in $\mtts$ we have already derived $B\, set\,[x\in A]$ and $C\, set\,[x\in A]$. 
By inductive hypothesis on substitution we have that in $\tar$:
\[( p_{1}({\cal I}(d)[{\cal I}(a)/x])= p_{1}({\cal I}(d[a/x]))=0\wedge  p_{2}({\cal I}(d)[{\cal I}(a)/x])\sim_{{\cal I}(B)} p_{2}({\cal I}(d[a/x])))\vee\]
\[\qquad\qquad( p_{1}({\cal I}(d)[{\cal I}(a)/x])= p_{1}({\cal I}(d[a/x]))=1\wedge  p_{2}({\cal I}(d)[{\cal I}(a)/x])\sim_{{\cal I}(C)} p_{2}({\cal I}(d[a/x])));\]
\[{\cal I}(b)[{\cal I}(a)/x]\sim_{{\cal I}(C[\mathsf{inl}(y)/z][a/x])} {\cal I}(b[a/x]);\]
\[{\cal I}(c)[{\cal I}(a)/x]\sim_{{\cal I}(C[\mathsf{inr}(y)/z][a/x])} {\cal I}(c[a/x]).\]
Using the inductive hypothesis on validity for $b$ and $c$, the inductive hypothesis on substitution for $B$ and $C$ with respect $a$, and the previous relations we obtain that 
\[( p_{1}({\cal I}(d)[{\cal I}(a)/x])= p_{1}({\cal I}(d[a/x]))=0\wedge\]
\[ {\cal I}(b)[{\cal I}(a)/x][ p_{2}({\cal I}(d)[{\cal I}(a)/x])/y]\sim_{{\cal I}(C[\mathsf{inl}(y)/z][{\cal I}(a)/x, p_{2}({\cal I}(d[a/x])/y])}{\cal I}(b)[{\cal I}(a)/x][ p_{2}({\cal I}(d[a/x]))/y]\wedge\]
\[\qquad\qquad{\cal I}(b)[{\cal I}(a)/x][ p_{2}({\cal I}(d[a/x]))/y]\sim_{{\cal I}(C[\mathsf{inl}(y)/u][a/u])[{\cal I}(d[a/x])/y]} {\cal I}(b[a/x])[ p_{2}({\cal I}(d[a/x]))/y] )\vee\]
\[( p_{1}({\cal I}(d)[{\cal I}(a)/x])= p_{1}({\cal I}(d[a/x]))=1\wedge \]
\[{\cal I}(c)[{\cal I}(a)/x][ p_{2}({\cal I}(d)[{\cal I}(a)/x])/y]\sim_{{\cal I}(C[\mathsf{inr}(y)/z][{\cal I}(a)/x, p_{2}({\cal I}(d[a/x])/y])}{\cal I}(c)[{\cal I}(a)/x][ p_{2}({\cal I}(d[a/x]))/y]\wedge\]
\[\qquad\qquad{\cal I}(c)[{\cal I}(a)/x][ p_{2}({\cal I}(d[a/x]))/y] \sim_{{\cal I}(C[\mathsf{inr}(y)/u][a/u])[{\cal I}(d[a/x])/y]} {\cal I}(c[a/x])[ p_{2}({\cal I}(d[a/x]))/y] ).\]

Applying many times the inductive hypothesis on substitution for $C$ with respect to different terms, we obtain that
\[( p_{1}({\cal I}(d)[{\cal I}(a)/x])= p_{1}({\cal I}(d[a/x]))=0\wedge\]
\[\qquad {\cal I}(b)[{\cal I}(a)/x][ p_{2}({\cal I}(d)[{\cal I}(a)/x])/y]\sim_{{\cal I}(C[d/u][a/x]} {\cal I}(b[a/x])[ p_{2}({\cal I}(d[a/x]))/y] )\vee\]
\[\qquad\qquad( p_{1}({\cal I}(d)[{\cal I}(a)/x])= p_{1}({\cal I}(d[a/x]))=1\wedge\]
\[\qquad\qquad\qquad {\cal I}(c)[{\cal I}(a)/x][ p_{2}({\cal I}(d)[{\cal I}(a)/x])/y]\sim_{{\cal I}(C[d/u][a/u]})
{\cal I}(c[a/x])[ p_{2}({\cal I}(d[a/x]))/y] ).\]
Using the definition of the interpretation of $El_{+}(d,(y)b,(y)c)$ we can conclude.
\subsubsection*{Binary sum conversion} Suppose we derived $El_{+}(\mathsf{inl}(e),(y)b,(y)c)=b[e/y]\in C[\mathsf{inl}(e)/z][x\in A]$ in $\mtts$ by conversion after having derived $e\in B[x\in A]$, $C\, set\,[x\in A,z\in B+C]$, $b\in C[\mathsf{inl}(y)/z][x\in A, y\in B]$ and $x\in C[\mathsf{inr}(y)/z][x\in A, y\in C]$.
Suppose $x\in {\cal J}(A)$, then in $\tar$ we have that by inductive hypothesis on substitution for $b$ with respect to $x\in A[x\in A]$ and $e\in B[x\in A]$
\[{\cal I}(b)[{\cal I}(e)/y]\sim_{{\cal I}(B)[\mathsf{inl}(e)/u]} {\cal I}(b[e/y])\]
which is exactly 
\[{\cal I}(El_{+}(\mathsf{inl}(e),(y)b,(y)c))\sim_{{\cal I}(B)[\mathsf{inl}(e)/u]} {\cal I}(b[e/y])\]
which is what we needed.
The other case is symmetric.

\subsubsection*{List sets} This case is similar to that of natural numbers but little more complicated. You must use induction with respect to the length of lists.
\subsubsection*{Falsum propositions} Completely analogous to the case of empty set.
 \subsubsection*{Conjunction propositions} Analogous to dependent sums.
\subsubsection*{Disjunction propositions} Analogous to binary sums.
\subsubsection*{Implication propositions} Analogous to dependent products.
\subsubsection*{Universal quantification propositions} Analogous to dependent products.
\subsubsection*{Existential quantification propositions}Analogous to dependent sums.
\subsubsection*{Identity propositions}
\subsubsection*{Identity formation} Suppose that ${\cal R}\vDash B\, set\,[x\in A]$, ${\cal R}\vDash b\in B[x\in A]$ and ${\cal R}\vDash c\in B[x\in A]$, then 
\[\tar\vdash\forall x\forall x'(x\sim_{{\cal I}(A)}x'\rightarrow \forall t\forall s(t\sim_{{\cal I}(B)}s\leftrightarrow t\sim_{{\cal I}(B)[x'/x]}s));\]
\[\tar\vdash\forall x\forall x'(x\sim_{{\cal I}(A)}x'\rightarrow {\cal I}(b)\sim_{{\cal I}(B)}{\cal I}(b)[x'/x]);\]
\[\tar\vdash\forall x\forall x'(x\sim_{{\cal I}(A)}x'\rightarrow {\cal I}(c)\sim_{{\cal I}(B)}{\cal I}(c)[x'/x]).\]
Using the previous conditions, we can deduce in $\tar$ from $x\sim_{{\cal I}(A)} x'$ that
\[\left(t\sim_{{\cal I}(B)}{\cal I}(b)\wedge {\cal I}(b)\sim_{{\cal I}(B)}{\cal I}(c)\right)\leftrightarrow\left( t\sim_{{\cal I}(B)[x'/x]}{\cal I}(b)[x'/x]\wedge {\cal I}(b)[x'/x]\sim_{{\cal I}(B)[x'/x]}{\cal I}(c)[x'/x]\right) \]
which means that $\tar\vdash \forall x\forall x'(x\sim_{{\cal I}(A)}x'\rightarrow \forall t(t\varepsilon {\cal I}(\mathsf{Id}(B,b,c))\leftrightarrow t\varepsilon {\cal I}(\mathsf{Id}(B,b,c))[x'/x]))$ that is 
${\cal R}\vDash \mathsf{Id}(B,b,c)[x\in A]$.
\subsubsection*{Substitution for Formation} In addition to the hypotheses in the previous point add that $\mtts\vdash a\in A $ and \\
$\tar\vdash {\cal I}(a)\varepsilon {\cal J}(A)$. By inductive hypothesis on substitution we have that in $\tar$
\[\forall t\forall s(t\sim_{{\cal I}(B)[{\cal I}(a)/x]}s\leftrightarrow t\sim_{{\cal I}(B[a/x])}s);\]
\[{\cal I}(b)[{\cal I}(a)/x]\sim_{{\cal I}(B[a/x])}{\cal I}(b[a/x]);\]
\[{\cal I}(c)[{\cal I}(a)/x]\sim_{{\cal I}(B[a/x])}{\cal I}(c[a/x]).\]
From these we obtain that for every $t$, in $\tar$,
\[t\varepsilon {\cal J}(\mathsf{Id}(B,b,c))[{\cal I}(a)/x]\]
is equivalent to 
\[t\sim_{{\cal I}(B)[{\cal I}(a)/x]}{\cal I}(b)[{\cal I}(a)/x]\wedge {\cal I}(b)[{\cal I}(a)/x]\sim_{{\cal I}(B)[{\cal I}(a)/x]}{\cal I}(c)[{\cal I}(a)/x] \]
and this is equivalent to 
\[t\sim_{{\cal I}(B[a/x]}{\cal I}(b)[{\cal I}(a)/x]\wedge {\cal I}(b)[{\cal I}(a)/x]\sim_{{\cal I}(B[a/x])}{\cal I}(c)[{\cal I}(a)/x] \]
and this is equivalent to
\[t\sim_{{\cal I}(B[a/x]}{\cal I}(b[a/x])\wedge {\cal I}(b[a/x])\sim_{{\cal I}(B[a/x])}{\cal I}(c[a/x]) \]
which is 
\[t\varepsilon {\cal J}(\mathsf{Id}(B,b,c)[a/x]).\]
\subsubsection*{Identity introduction} Suppose we derive $\mathsf{id}(b)\in \mathsf{Id}(B,b,b)[x\in A]$ by introduction in $\mtts$ after having derived \\$b\in B[x\in A]$. 
By inductive hypothesis on validity we have that in $\tar$
\[\forall x\forall x'(x\sim_{{\cal I}(A)}x'\rightarrow {\cal I}(b)\sim_{{\cal I}(B)}{\cal I}(b)[x'/x]).\]
This implies that in $\tar$ we have 
\[\forall x\forall x'(x\sim_{{\cal I}(A)}x'\rightarrow {\cal I}(b)\sim_{{\cal I}(B)}{\cal I}(b)\wedge {\cal I}(b)[x'/x]\sim_{{\cal I}(B)}{\cal I}(b) )\]
which means that in $\tar$
\[\forall x\forall x'(x\sim_{{\cal I}(A)}x'\rightarrow {\cal I}(\mathsf{id}(b))\varepsilon {\cal I}(\mathsf{Id}(B,b,b))\wedge  {\cal I}(\mathsf{id}(b))[x'/x]\varepsilon {\cal I}(\mathsf{Id}(B,b,b)))\]
which means that ${\cal R}\vDash \mathsf{id}(b)\in \mathsf{Id}(B,b,b)[x\in A]$.

\subsubsection*{Substitution for identity introduction} In addition to the hypothesis in the previous point suppose that $\mtts\vdash a\in A$ and \\$\tar\vdash {\cal I}(a)\varepsilon {\cal J}(A)$.
By inductive hypothesis on substitution we have that 
\[\tar\vdash{\cal I}(b)[{\cal I}(a)/x]\sim_{ {\cal I}(B[a/x])}{\cal I}(b[a/x]).\]
 This implies in particular that 
\[{\cal I}(\mathsf{id}(b))[{\cal I}(a)/x]\varepsilon {\cal J}(\mathsf{Id}(B,b,b)[a/x])\wedge {\cal I}(\mathsf{id}(b)[a/x])\varepsilon {\cal J}(\mathsf{Id}(B,b,b)[a/x])\]
which is what we needed.
\subsubsection*{Identity elimination} Suppose we derived in $\mtts$ the judgment $El_{Id}(p,(y)r))\in \delta[b/y,c/z][x\in A]$ by elimination after having derived in $\mtts$ the judgments
\begin{enumerate} 
\item $b\in B[x\in A]$, 
\item $c\in B[x\in A]$, 
\item $p\in \mathsf{Id}(B,b,c)[x\in A]$, 
\item $\delta\,prop_{s}[x\in A,y\in B,z\in B]$,
\item $\delta[b/y,c/z]\,prop_{s}[x\in A]$ and 
\item $r\in \delta[y/z][x\in A, y\in B]$ (and so also $\delta[y/z]\, set\,[x\in A, y\in B]$ by the structure of the rules of $\mtts$). 
\end{enumerate}
By inductive hypothesis on validity for $p$, $a$ and $b$ we have that in $\tar$
\[(*)\,\forall x\forall x'(x\sim_{{\cal I}(A)}x'\rightarrow {\cal I}(p)\sim_{{\cal I}(B)} {\cal I}(b)\wedge {\cal I}(p)[x'/x]\sim_{{\cal I}(B)} {\cal I}(p)\wedge {\cal I}(b)\sim_{{\cal I}(B)} {\cal I}(c) \wedge\]
\[\qquad\qquad\qquad {\cal I}(b)\sim_{{\cal I}(B)}{\cal I}(b)[x'/x]\wedge {\cal I}(c)\sim_{{\cal I}(B)}{\cal I}(c)[x'/x]).\]
By inductive hypothesis on validity for $\delta[y/z]$ and what we just showed we have that in $\tar$
\[(**)\,\forall x(x\varepsilon {\cal J}(A)\rightarrow {\cal I}(\delta[y/z])[{\cal I}(p)/y]\leftrightarrow {\cal I}(\delta[y/z])[{\cal I}(b)/y]).\]
 By using the inductive hypothesis on substitution for $\delta$ with respect to $x\in A[x\in A, y\in B]$ and two copies of $y\in B[x\in A,y\in B]$ we obtain that in $\tar$
 \[(***)\,\forall x\forall y(x\varepsilon {\cal J}(A)\wedge y\varepsilon {\cal J}(B)\rightarrow {\cal I}(\delta)[y/z]\doteq {\cal I}(\delta[y/z])).\]
 Putting $(*)$, $(**)$ and $(***)$ together we obtain that in $\tar$
 \[\forall x(x\varepsilon {\cal J}(A)\rightarrow {\cal I}(\delta[y/z])[{\cal I}(p)/y]\leftrightarrow {\cal I}(\delta)[{\cal I}(b)/y,{\cal I}(b)/z]).\]
 Using the inductive hypothesis on validity for $\delta$ and $(*)$ we obtain that  in $\tar$
  \[\forall x(x\varepsilon {\cal J}(A)\rightarrow {\cal I}(\delta[y/z])[{\cal I}(p)/y]\leftrightarrow {\cal I}(\delta)[{\cal I}(b)/y,{\cal I}(c)/z]).\]
Using the inductive hypothesis on substitution for $\delta$ with respect to $x\in A[x\in A]$, \\
$b\in B[x\in A]$ and $c\in B[x\in A]$ we obtain that in $\tar$
  \[(****)\,\forall x(x\varepsilon {\cal J}(A)\rightarrow {\cal I}(\delta[y/z])[{\cal I}(p)/y]\leftrightarrow {\cal I}(\delta[b/y,c/z])).\]
Using the inductive hypothesis on validity for $p$ and $r$ we obtain that in $\tar$
\[(*****)\,\forall x\forall x'(x\sim_{{\cal I}(A)}x'\rightarrow {\cal I}(r)[{\cal I}(p)/y]\varepsilon {\cal J}(\delta[y/z])[{\cal I}(p)/y] \wedge\]
\[\qquad\qquad\qquad{\cal I}(r)[x'/x][{\cal I}(p)[x'/x]/y]  \varepsilon {\cal J}(\delta[y/z])[{\cal I}(p)/y]   ).\]
From $(****)$ and $(*****)$ we obtain that in $\tar$
\[\forall x\forall x'(x\sim_{{\cal I}(A)}x'\rightarrow {\cal I}(r)[{\cal I}(p)/y]\varepsilon {\cal J}(\delta[b/y,c/z]) \wedge{\cal I}(r)[x'/x][{\cal I}(p)[x'/x]/y]  \varepsilon {\cal J}(\delta[b/y,c/z])  )\]
which means that in $\tar$
\[\forall x\forall x'(x\sim_{{\cal I}(A)}x'\rightarrow {\cal I}(El_{Id}(p,(y)r))\varepsilon {\cal J}(\delta[b/y,c/z]) \wedge{\cal I}(El_{Id}(p,(y)r)[x'/x])  \varepsilon {\cal J}(\delta[b/y,c/z])  )\]
which exactly means that ${\cal R}\vDash El_{Id}(p,(y)r)\in \delta[b/y,c/z] [x\in A]$.


 \subsubsection*{Substitution for identity elimination}
 In addition to the hypotheses of the previous point assume that $\mtts\vdash a\in A $ and \\$\tar\vdash {\cal I}(a)\varepsilon {\cal J}(A)$. The proof is similar to that of the previous point. 
\subsubsection*{Collections of codes for small propositions and sets}
We start by considering validity and the rules of the collection $\mathsf{Set}$.
\subsubsection*{Universe of sets}
\subsubsection*{Universe of sets formation}
The formation is trivially verified.
\subsubsection*{Universe of sets introduction}
The validity of the judgments $\widehat{N_{0}}\in \mathsf{Set}$, $\widehat{N_{1}}\in \mathsf{Set}$ and $\widehat{N}\in \mathsf{Set}$ follows directly from the coding. It is an immediate exercise in classical logic to show that if ${\cal R}\vDash p\in \mathsf{Set}$ and ${\cal R}\vDash q\in \mathsf{Set}$, then also ${\cal R}\vDash p\hat{+}q\in \mathsf{Set}$ and if ${\cal R}\vDash p\in \mathsf{Set}$, then ${\cal R}\vDash \widehat{List}(p)\in \mathsf{Set}$. It is also immediate, by definition, to show that if ${\cal R}\vDash p\in \mathsf{prop}_{s}$, then ${\cal R}\vDash p\in \mathsf{Set}$. In the case of $(\widehat{\Sigma x\in A})p$ and $(\widehat{\Pi x\in A})p$ we must use the inductive hypothesis on coding and on validity to show (in a way analogous to that of the proof of coding which will follow) that if $\mtts \vdash A\,set$ and $\mtts \vdash p\in \mathbf{Set}[x\in A]$, then $\tar\vdash \mathbf{Fam}(\Lambda x. {\cal I}(p),{\cal I}(\widehat{A}))$ and then to show ${\cal R}\vDash  (\widehat{\Sigma x\in A})p\in \mathsf{Set}$ and ${\cal R}\vDash (\widehat{\Pi x\in A})p\in \mathsf{Set}$.
The proof of substitution consists of an easy verification. 

\subsubsection*{Universe of small propositions}
The case of $\mathsf{prop}_{s}$ is completely analogous and elimination and conversions follow from the proof-irrelevance. Proofs of the statements for propositions are analogous to those for small propositions. 
\subsubsection*{Other collections}Proofs of the statements for $\Sigma$-collections are analogous to those for $\Sigma$-sets. Proofs of the statements for $\rightarrow \mathsf{prop_{s}}$-collections are analogous to those for $\Pi$-sets. Substitutions can be proven by using the inductive hypotheses.

\subsubsection*{\textbf{General rules}}

\subsubsection*{Assumption of variable}
We must show that ${\cal R}\vDash x_{j}\in A_{j}[x_{1}\in A_{1},...,x_{\mathsf{n}}\in A_{\mathsf{n}}]$ for $1\leq j\leq \mathsf{n}$. This is obviously true as 
\[\tar\vdash\forall x_{1}\forall x'_{1}...\forall x_{\mathsf{n}}\forall x'_{\mathsf{n}}(x_{1}\sim_{{\cal I}(A_{1})}x'_{1}\wedge...\wedge x_{\mathsf{n}}\sim_{{\cal I}(A_{\mathsf{n}})}x'_{\mathsf{n}}\rightarrow x_{j}\sim_{{\cal I}(A_{j})}x'_{j}) .\]
For substitution if $\mtts\vdash a_{1}\in A_{1}$,...,$\mtts\vdash a_{\mathsf{n}}\in A_{\mathsf{n}}[a_{1}/x_{1},...,a_{n-1}/x_{n-1}]$ and these judgments are validated by ${\cal R}$, then in particular
\[{\cal I}(a_{j})\sim_{{\cal I}(A_{j}[a_{1}/x_{1},...,a_{j-1}/x_{j-1}])}{\cal I}(a_{j})\]
which is exactly 
\[{\cal I}(x_{j})[{\cal I}(a_{1})/x_{1},...,{\cal I}(a_{\mathsf{n}})/x_{\mathsf{n}}]\sim_{{\cal I}(A_{j}[a_{1}/x_{1},...,a_{j-1}/x_{j-1}])}{\cal I}(x_{j}[a_{1}/x_{1},...,a_{\mathsf{n}}/x_{\mathsf{n}}]).\]

\subsubsection*{Reflexivity, symmetry and transitivity of type equality}
Suppose that from $\mtts\vdash B\,type\,[x\in A]$ we derive by reflexivity that \\
$\mtts\vdash B=B\,type\,[x\in A]$. By inductive hypothesis on validity we have that \[\tar \vdash \forall x\forall x'(x\sim_{{\cal I}(A)}x'\rightarrow\forall t\forall s (t\sim_{{\cal I}(B)}s\leftrightarrow t\sim_{{\cal I}(B)[x'/x]}s) ).\] Now $x\varepsilon {\cal J}(A)$ implies in $\tar$ that $x\sim_{{\cal I}(A)}x$. So 
\[\tar \vdash \forall x(x\varepsilon {\cal J}(A)\rightarrow\forall t\forall s (t\sim_{{\cal I}(B)}s\leftrightarrow t\sim_{{\cal I}(B)}s) )\] which exactly means that 
\[{\cal R}\vDash B=B\,type\,[x\in A].\]

Suppose now that from $\mtts\vdash B=C\,type\,[x\in A]$  we derive by symmetry \[\mtts\vdash C=B\,type\,[x\in A].\] By inductive hypothesis on validity we have that \[\tar \vdash \forall x\forall x'(x\varepsilon {\cal J}(A)\rightarrow\forall t\forall s (t\sim_{{\cal I}(B)}s\leftrightarrow t\sim_{{\cal I}(C)}s) ).\]
This clearly entails that 
\[\tar \vdash \forall x\forall x'(x\varepsilon {\cal J}(A)\rightarrow\forall t\forall s (t\sim_{{\cal I}(C)}s\leftrightarrow t\sim_{{\cal I}(B)}s) )\]
which exactly means that ${\cal R}\vDash C=B\,type\,[x\in A]$.

Suppose now that from $\mtts\vdash B=C\,type\,[x\in A]$ and $\mtts\vdash C=D\,type\,[x\in A]$ we derive by transitivity $\mtts\vdash B=D\,type\,[x\in A]$. By inductive hypothesis on validity we have that $\tar \vdash \forall x\forall x'(x\varepsilon {\cal J}(A)\rightarrow\forall t\forall s (t\sim_{{\cal I}(B)}s\leftrightarrow t\sim_{{\cal I}(C)}s) )$ and 
\[\tar \vdash \forall x\forall x'(x\varepsilon {\cal J}(A)\rightarrow\forall t\forall s (t\sim_{{\cal I}(C)}s\leftrightarrow t\sim_{{\cal I}(D)}s) ).\]
And this clearly entails that 
\[\tar \vdash \forall x\forall x'(x\varepsilon {\cal J}(A)\rightarrow\forall t\forall s (t\sim_{{\cal I}(B)}s\leftrightarrow t\sim_{{\cal I}(D)}s) )\]
which exactly means that ${\cal R}\vDash B=D\,type\,[x\in A]$.

\subsubsection*{Substitution for types}

We restrict to the case of one substitution.

 Suppose that $\mtts\vdash D[b_{1}/y_{1},c_{1}/z_{1}]=D[b_{2}/y_{2},c_{2}/z_{2}]\,type\,[x\in A]$ is derived by $sub-T$ after having derived in $\mtts$ the judgments \begin{enumerate}
\item $D\,type\,[x\in A, y\in B, z\in C]$, 
\item $b_{1}\in B[x\in A]$, 
\item $b_{2}\in B[x\in A]$,
\item $c_{1}\in C[b_{1}/y][x\in A]$, 
\item $c_{2}\in C[b_{2}/y][x\in A]$, 
\item $b_{1}=b_{2}\in B[x\in A]$, 
\item $c_{1}=c_{2}\in C[b_{1}/y][x\in A]$, 
\item $B\,type[x\in A]$ and 
\item $C\,type[x\in A, y\in B]$.
\end{enumerate}
Using the inductive hypothesis on validity and on substitution we can obtain that in $\tar$
\[\forall x(x\varepsilon {\cal J}(A)\rightarrow {\cal I}(D)[{\cal I}(b_{1})/y,{\cal I}(c_{1})/z]\doteq {\cal I}(D[b_{1}/y,c_{1}/z])),\] 
\[\forall x(x\varepsilon {\cal J}(A)\rightarrow {\cal I}(D)[{\cal I}(b_{2})/y,{\cal I}(c_{2})/z]\doteq {\cal I}(D[b_{2}/y,c_{2}/z])),\]
\[\forall x(x\varepsilon {\cal J}(A)\rightarrow {\cal I}(C)[{\cal I}(b_{1})/y]\doteq {\cal I}(C[b_{1}/y])).\]
These together with the inductive hypothesis on validity for 
\begin{enumerate}
\item $D\,type\,[x\in A, y\in B, z\in C]$, 
\item $b_{1}=b_{2}\in B[x\in A]$ (which is $\tar \vdash \forall x(x\varepsilon {\cal J}(A)\rightarrow {\cal I}(b_{1})\sim_{{\cal I}(B)}{\cal I}(b_{2}))$),
\item $c_{1}=c_{2}\in C[b_{1}/y][x\in A]$ (which is $\tar \vdash \forall x(x\varepsilon {\cal J}(A)\rightarrow {\cal I}(c_{1})\sim_{{\cal I}(C[b_{1}/y])}{\cal I}(c_{2}))$), 
\end{enumerate}
gives us that in $\tar$
\[\forall x(x\varepsilon {\cal J}(A)\rightarrow {\cal I}(D[b_{1}/y,c_{1}/z])\doteq {\cal I}(D[b_{2}/y,c_{2}/z])).\]

\subsubsection*{Reflexivity, symmetry and transitivity of definitional equality}
The rules of reflexivity, symmetry and transitivity for terms preserve the validity with premisses provable in $\mtts$, thanks to \ref{prelim}. 
 \subsubsection*{Substitution for terms}
 Suppose that $\mtts\vdash d[b_{1}/y_{1},c_{1}/z_{1}]=d[b_{2}/y_{2},c_{2}/z_{2}]\in D[b_{1}/y_{1},c_{1}/z_{1}]\,[x\in A]$ is derived by $sub$ after having derived in $\mtts$ the judgments $d\in D[x\in A, y\in B, z\in C]$, \\$D\,type\,[x\in A, y\in B, z\in C]$, $b_{1}\in B[x\in A]$, $b_{2}\in B[x\in A]$, $c_{1}\in C[b_{1}/y][x\in A]$, \\$c_{2}\in C[b_{2}/y][x\in A]$, $b_{1}=b_{2}\in B[x\in A]$, $c_{1}=c_{2}\in C[b_{1}/y][x\in A]$, $B\,type\,[x\in A]$ and $C\,type\,[x\in A, y\in B]$ .  As in the case of $sub-T$ we derive that in $\tar$
\[\forall x(x\varepsilon {\cal J}(A)\rightarrow {\cal I}(C)[{\cal I}(b_{1})/y]\doteq {\cal I}(C[b_{1}/y])),\]
\[\forall x(x\varepsilon {\cal J}(A)\rightarrow {\cal I}(D[b_{1}/y,c_{1}/z])\doteq {\cal I}(D[b_{2}/y,c_{2}/z])).\]
Now by using the inductive hypothesis on substitution for $D$ and $d$ we obtain that 
\[\forall x(x\varepsilon {\cal J}(A)\rightarrow {\cal I}(D)[{\cal I}(b_{1})/y,{\cal I}(c_{1})/z]\doteq {\cal I}(D[b_{1}/y,c_{1}/z])),\] 
\[\forall x(x\varepsilon {\cal J}(A)\rightarrow {\cal I}(d)[{\cal I}(b_{1})/y,{\cal I}(c_{1})/z]\sim {\cal I}(d[b_{1}/y,c_{1}/z])),\] 
\[\forall x(x\varepsilon {\cal J}(A)\rightarrow {\cal I}(d)[{\cal I}(b_{2})/y,{\cal I}(c_{2})/z]\sim {\cal I}(d[b_{2}/y,c_{2}/z])),\]
and using these and the previous together with the inductive hypothesis on validity for $d$, $b_{1}=b_{2}$ and $c_{1}=c_{2}$, we obtain that 
\[\forall x(x\varepsilon {\cal J}(A)\rightarrow {\cal I}(d[b_{1}/y,c_{1}/z])\sim_{{\cal I}(D[b_{1}/y,c_{1}/z])} {\cal I}(d[b_{2}/y,c_{2}/z])).  \]

\subsubsection*{Rules of conversions}

Suppose $\mtts\vdash b\in C[x\in A]$ is derived by the rule $conv$ after having derived \\$\mtts\vdash B=C[x\in A]$ and $\mtts\vdash b\in B[x\in A]$. By inductive hypothesis we have that 
\[\forall x\forall x'(x\sim_{{\cal I}(A)}x'\rightarrow {\cal I}(b)\sim_{{\cal I}(B)}{\cal I}(b)[x'/x]);\]
\[\forall x (x\varepsilon {\cal J}(A)\rightarrow {\cal I}(B)\doteq {\cal I}(C)).\]
From these it immediately follows that 
\[\forall x\forall x'(x\sim_{{\cal I}(A)}x'\rightarrow {\cal I}(b)\sim_{{\cal I}(C)}{\cal I}(b)[x'/x])\]
which means that ${\cal R}\vDash b\in C[x\in A]$.

Let us now prove substitution for $conv$. Suppose $\mtts\vdash a\in A$. Then by inductive hypothesis on substitution (using the fact that we know that there are shorter proofs of $B\,type\,[x\in A]$ and $C\,type\,[x\in A]$) we have that if $x\varepsilon {\cal J}(A)$ we have that  in $\tar$ 
\[{\cal I}(b)[{\cal I}(a)/x]\sim_{{\cal I}(B[a/x])}{\cal I}(b[a/x]);\]
\[{\cal I}(B)[{\cal I}(a)/x]\sim {\cal I}(B[a/x]);\]
\[{\cal I}(C)[{\cal I}(a)/x]\sim {\cal I}(C[a/x]);\]
and using the inductive hypothesis on validity we obtain that 
\[{\cal I}(b)[{\cal I}(a)/x]\sim_{{\cal I}(C[a/x])}{\cal I}(b[a/x])\]
which is what we needed.

The rule of conversion immediately follows from the definition of the interpretation of judgments.

\subsubsection*{Coding condition}
First of all notice that, for coding, it is sufficient to show that if $\mtts\vdash A\,set$, then $\tar\vdash \mathbf{Set}(\widehat{A})$. In fact if basic formulas $x\varepsilon {\cal J}(-)$, $\neg(x \varepsilon {\cal J}(-))$, $x\sim_{{\cal I}(-)}y$ and $\neg(x\sim_{{\cal I}(-)}y)$ are equivalent respectively to $x\overline{\varepsilon} {\cal I}(\widehat{-})$, $x\noteps  {\cal I}(\widehat{-})$, $x\equiv_{{\cal I}(\widehat{-})}y$ and $x\not\equiv_{{\cal I}(\widehat{-})}y$, then $(\phi)^{+}$ is equivalent to $\phi$ and $\overline{\phi}$ is equivalent to $\neg \phi$. So we must suppose that $\mtts\vdash A\,set$ is derived by formation from other provable judgments and then we must prove, using the inductive hypothesis, that $\tar\vdash \mathbf{Set}(\widehat{A})$.

The cases $\mathsf{N}_{0},\mathsf{N}_{1}, \mathsf{N}, A+A', List(A), \bot, A\wedge A', A\vee A', A\rightarrow A', \mathsf{Id}(A,a,b)$ are immediate.

\subsubsection*{Coding condition for dependent sums and products}
Suppose that we derived $(\Pi x\in A)B\,set$ or $(\Sigma x\in A)B\,set$ in $\mtts$ by formation after having derived $A\,set$ and $B\, set\,[x\in A]$.
By inductive hypothesis on coding for $A$ we have that 
\[(*)\,\tar\vdash \mathsf{Set}({\cal I}(\widehat{A})).\]
By inductive hypothesis on coding for $B$ we have also that 
\[\tar\vdash\forall x(x\varepsilon{\cal J}(A)\rightarrow \mathbf{Set}({\cal I}(\widehat{B})))\]
and so using classical logic we have that 
\[\tar\vdash\forall x(\neg x\varepsilon{\cal J}(A)\vee \mathbf{Set}({\cal I}(\widehat{B})))\]
and using the inductive hypothesis on coding for $A$ we have that 
\[(**)\,\tar\vdash\forall x(x\noteps {\cal I}(\widehat{A})\vee \mathbf{Set}({\cal I}(\widehat{B}))).\]
Now suppose that $x\sim_{{\cal I}(A)}x'$, then by inductive hypothesis on validity we can deduce in $\tar$ that 
\[\forall t(t\varepsilon {\cal J}(B)\leftrightarrow t\varepsilon {\cal J}(B)[x'/x])\wedge \forall t\forall s(t\sim_{{\cal I}(B)} s\leftrightarrow t\sim_{{\cal I}(B)[x'/x]} s)\]
which is equivalent, by classical logic, to
\[\forall t((\neg t\varepsilon {\cal J}(B)\vee t\varepsilon {\cal J}(B)[x'/x])\wedge (\neg t\varepsilon {\cal J}(B)[x'/x]\vee t\varepsilon {\cal J}(B)))\wedge \]
\[\qquad\qquad\qquad\forall t\forall s((\neg t\sim_{{\cal I}(B)} s\vee t\sim_{{\cal I}(B)[x'/x]} s)\wedge (\neg t\sim_{{\cal I}(B)[x'/x]} s\vee t\sim_{{\cal I}(B)} s) )\]
which is equivalent by inductive hypothesis on coding to 
\[\forall t((t\noteps  {\cal I}(\widehat{B})\vee t\overline{\varepsilon} {\cal I}(\widehat{B})[x'/x])\wedge (t\noteps  {\cal I}(\widehat{B})[x'/x]\vee t\overline{\varepsilon} {\cal I}(\widehat{B})))\wedge \]
\[\qquad\qquad\qquad\forall t\forall s((t\not\equiv_{{\cal I}(\widehat{B})} s\vee t\equiv_{{\cal I}(\widehat{B})[x'/x]} s)\wedge (t\not\equiv_{{\cal I}(\widehat{B})[x'/x]} s\vee t\equiv_{{\cal I}(\widehat{B})} s) ).\]
So 
\[\tar\vdash \forall x\forall x'(x\sim_{{\cal I}(A)}x'\rightarrow \forall t((t\noteps  {\cal I}(\widehat{B})\vee t\overline{\varepsilon} {\cal I}(\widehat{B})[x'/x])\wedge (t\noteps  {\cal I}(\widehat{B})[x'/x]\vee t\overline{\varepsilon} {\cal I}(\widehat{B})))\wedge \]
\[\qquad\qquad\qquad\forall t\forall s((t\not\equiv_{{\cal I}(\widehat{B})} s\vee t\equiv_{{\cal I}(\widehat{B})[x'/x]} s)\wedge (t\not\equiv_{{\cal I}(\widehat{B})[x'/x]} s\vee t\equiv_{{\cal I}(\widehat{B})} s) ).\]
By using classical logic and the inductive hypothesis on coding for $A$, we obtain that 
\[(***)\,\tar\vdash \forall x\forall x'(x\not\equiv_{{\cal I}(\widehat{A})}x'\vee \forall t((t\noteps  {\cal I}(\widehat{B})\vee t\overline{\varepsilon} {\cal I}(\widehat{B})[x'/x])\wedge (t\noteps  {\cal I}(\widehat{B})[x'/x]\vee t\overline{\varepsilon} {\cal I}(\widehat{B})))\wedge \]
\[\qquad\qquad\qquad\forall t\forall s((t\not\equiv_{{\cal I}(\widehat{B})} s\vee t\equiv_{{\cal I}(\widehat{B})[x'/x]} s)\wedge (t\not\equiv_{{\cal I}(\widehat{B})[x'/x]} s\vee t\equiv_{{\cal I}(\widehat{B})} s) ).\]
Combining $(*)$, $(**)$ and $(***)$ we obtain that 
\[\tar\vdash \mathbf{Fam}(\Lambda x.{\cal I}(\widehat{B}),\widehat{A})\]
from which we deduce that 
\[\tar\vdash\mathsf{Set}(\widehat{(\Pi x\in A)B})\wedge\mathsf{Set}(\widehat{(\Sigma x\in A)B})\]

The cases of $\forall$ and $\exists$ are similar.


%
%


\subsubsection*{Consequences of the validity theorem}
We discuss here about the validity in our realizability model for $\mtt$ of some principles, namely Extensionality Equality of Functions, Axiom of Choice and formal Church Thesis.

\begin{enumerate}
\item \textbf{Extensionality Equality of Functions} can be formulated as a proposition in \mtt\ as follows:
\[\begin{array}{rl}({\bf extFun}) & (\forall f\in (\Pi x\in A)\,B)\, (\forall g\in (\Pi x\in A)\,B)\qquad \\
& \qquad \qquad ((\forall x\in A)\,\mathsf{Id}(B,\mathsf{Ap}(f,x),\mathsf{Ap}(g,x))\rightarrow \mathsf{Id}((\Pi x\in A)\,B, f,g))\end{array}\]
Since the judgements $f=g\in (\Pi x\in A)\,B$ and $\mathsf{Ap}(f,x)=\mathsf{Ap}(g,x)\in B\,[x\in A]$
have the same interpretation, {\bf extFun} can be realized by the term $\Lambda f.\Lambda g.\Lambda r.0$, i.\,e.\, our model realises ${\bf extFun}$.\\

\item The \textbf{Axiom of Choice} {\bf AC}$_{A,B}$ is represented in \mtt\ by the following proposition:
\[({\bf AC}_{A,B})\;(\forall x\in A)\,(\exists y\in B)\,\rho(x,y)\rightarrow (\exists f\in (\Pi x\in A)\,B)\,(\forall x\in A)\,\rho(x,\mathsf{Ap}(f,x))\]
Unfortunately a realizer $r$ for $(\forall x\in A)\,(\exists y\in B)\,\rho(x,y)$ cannot be turned into a recursive function from ${\cal J}(A)$ to ${\cal J}(B)$ respecting equivalence relations $\sim_{{\cal I}(A)}$ and $\sim_{{\cal I}(B)}$, as the interpretation of propositions is proof-irrelevant and we can have different elements $a$ and $a'$ of ${\cal J}(A)$ which are equivalent in ${\cal I}(A)$ for which $\pi_{1}(\{r\}(a))$ and $\pi_{1}(\{r\}(a'))$ are not equivalent in ${\cal I}(B)$.
This problem can be avoided if $A$ is a numerical set and in particular in the case of the set $N$. In this case the natural number $\Lambda r.\pair{\Lambda n.\pi_{1}(\{r\}(n))}{\Lambda n.\pi_{2}(\{r\}(n))}$ is a realizer for the axiom of choice {\bf AC}$_{N,B}$. So ${\cal R}\Vdash {\bf AC}_{N,B}$ for every $B$.

Moreover also the axiom of unique choice ${\bf AC}_{!}$ given by
\[({\bf AC}_{!})\;(\forall x\in A)\,(\exists ! y\in B)\,\rho(x,y)\rightarrow (\exists f\in (\Pi x\in A)\,B)\,(\forall x\in A)\,\rho(x,\mathsf{Ap}(f,x))\]
is validated by the model ${\cal R}$.\footnote{$(\exists ! x\in A)P(x)$ is defined as $(\exists x\in A)P(x)\wedge (\forall x\in A)(\forall x'\in A)(P(x)\wedge P(x')\rightarrow \mathsf{Id}(A,x,x'))$.} In fact if $\rho(x,y)$ is a proposition in context $[x\in A, y\in B]$, then 
in particular $\tar\,\vdash\,\forall x \forall x' \forall y\forall t\,(x\sim_{{\cal I}(A)}x'\,\wedge\, y\,\varepsilon\, {\cal J}(B)\,\wedge\, t\,\Vdash\, \rho(x,y)\rightarrow t\,\Vdash\, \rho(x',y))$.
This implies that we can easily choose a realizer for the axiom of unique choice.\\

\item 
%
If $\varphi$ is a formula of first-order arithmetic $\mathsf{HA}$, then we can define a proposition $\overline{\varphi}$ in $\mtt$, according to the following conditions:
\[
\begin{tabular}{lll}
 $\overline{\bot}$ is $\bot$ & $\overline{\varphi\wedge \varphi'}$ is $\overline{\varphi}\wedge \overline{\varphi'}$ &$\overline{\exists x\,\varphi}$ is $(\exists x\in N)\,\overline{\varphi}$\\
  $\overline{t=s}$ is $\mathsf{Id}(N,\overline{t},\overline{s})$ & $\overline{\varphi\vee \varphi'}$ is $\overline{\varphi}\vee \overline{\varphi'}$ & $\overline{\forall x\,\varphi}$ is $(\forall x\in N)\,\overline{\varphi}$ \\
  & $\overline{\varphi\rightarrow \varphi'}$ is $\overline{\varphi}\rightarrow \overline{\varphi'} $ & \\
\end{tabular}
\]
%
where $\overline{t}$ and $\overline{s}$ are the translations of terms of $\mathsf{HA}$ in $\mtt$ (in particular primitive recursive functions of $\mathsf{HA}$ are translated via $El_{N}$, $succ$ and $0$ are translated in the obvious corresponding ones and variables are interpreted as themselves\footnote{Here we suppose that variables of $\mathsf{HA}$ coincides with variables of the untyped syntax of $\mtts$.}). 
The language of $\mathsf{HA}$ can also be naturally interpreted in $\tar$ by using the fact that each primitive recursive function can be encoded by a numeral. If $t$ is a term of $\mathsf{HA}$ we will still write $t$ for its translation in $\tar$.
The following lemma is an immediate consequence of the definition of our realizability interpretation where  $\Vdash_{k}$ denotes Kleene realizability in $ \mathsf{HA} $ (see \cite{DT88}):\\
\begin{lemma}\label{realiz} If $t$ is a term of $\mathsf{HA}$ and $\varphi$ is a formula of $\mathsf{HA}$, then 
\begin{enumerate}
\item $\tar\,\vdash\,{\cal I}(\overline{t})=t$
\item $\tar\,\vdash\,n\,\Vdash_{k} \,\varphi \leftrightarrow n\,\Vdash\, \overline{\varphi}.$\\
\end{enumerate}
\end{lemma}

The \textbf{formal Church Thesis} {\bf CT} can be expressed in $\mtt$
  as the following proposition
\[({\bf CT})\,(\forall x\in N)\,(\exists y\in N)\,\rho(x,y)\rightarrow (\exists e\in N)\,(\forall x\in N)\,(\exists u\in N)\,(\overline{T(e,x,u)}\,\wedge\, \rho(x, \overline{U(u)})\] 
where $T$ and $U$ are the Kleene predicate and the primitive recursive function representing Kleene application in $\mathsf{HA}$. 
Note that the validity of {\bf CT} can be obtained by glueing ${\bf{AC}}_{N,N}$ together with the following restricted form of Church Thesis for type-theoretic functions:
\[({\bf CT}_{\lambda})\,(\forall f\in (\Pi x\in N)N)\,(\exists e\in N)\,(\forall x\in N)\,(\exists u\in N)\,(\overline{T(e,x,u)}\wedge\mathsf{Id}(N,\mathsf{Ap}(f,x),\overline{U(u)}))\]

We know by general results on Kleene realizability that there exists a numeral $\mathbf{r}$ for which $\mathsf{HA}\,\vdash\, \exists u\, T(f,x,u)\rightarrow (\{\mathbf{r}\}(f,x)\,\Vdash\, \exists u\,T(f,x,u)).$
Using this remark, the fact that $\{f\}(x)\downarrow$ is equivalent to $\exists u\, T(f,x,u)$ in $\tar$, the proof irrelevance and lemma \ref{realiz} we can show that ${\bf CT}_{\lambda}$ can be realized by 
$$\Lambda f.\langle f,\Lambda x.\langle\{\mathbf{p}_{1}\}(\{\mathbf{r}\}(f,x),\langle\{\mathbf{p}_{2}\}(\{\mathbf{r}\}(f,x),0\rangle\rangle\rangle.$$ 
In fact every function from $N$ to $N$ is interpreted in the model as a code for a total recursive function and we can send this code to itself in order to realize Church Thesis. \emph{Proof irrelevance allows to ignore the problem that different codes can give rise to extensionally equal functions, which is crucial to prove validity of} {\bf CT}.

We can conclude this section by stating the following consistency results:\\
\begin{theorem}
$\mtt$ is consistent with ${\bf CT}$.
\end{theorem}
\begin{corollary}
$\emtt$ is consistent with {\bf CT}.
\end{corollary}
\begin{proof} According to the interpretation of \emtt\ in \mtt\ in \cite{m09}, the interpretation of {\bf CT} turns now to be equivalent to {\bf CT} itself. Therefore a model showing consistency of \mtt\ with {\bf CT} can be extended to a model of \emtt\ with {\bf CT}.
\end{proof}
\end{enumerate}

\section{Conclusions}
As explained in the introduction,
the semantics built here is the best Kleene realizability model
we can construct for 
the extensional level \emtt\ of the Minimalist Foundation, since \emtt\ validates Extensionality Equality of Functions and it is constructively incompatible with the Axiom of Choice on generic sets (see \cite{m09}), which is instead valid in Beeson's model. In our semantics
instances of the axiom of choice are still valid only on numerical
sets, which include the interpretation of basic intensional types as the set of natural numbers.

On the contrary, for the intensional level \mtt\ of the Minimalist Foundation  we 
hope to build a more intensional realizability semantics \`a la Kleene
where we validate not only  {\bf CT}  but also the Axiom of Choice {\bf AC} on generic types.
Recalling from \cite{m09} that our \mtt\ can be naturally interpreted in
Martin-L{\"o}f's type theory with one
universe, such an intensional Kleene realizability for \mtt\ could
be obtained by modelling intensional Martin-L{\"o}f's type theory with one
universe (with explicit substitutions in place of the usual substitution term equality rules) together with  {\bf CT}. However, as far as we know,  the consistency of intensional Martin-L{\"o}f's type theory  with {\bf CT} is still an open problem.
\\

{\bf Acknowledgements.}
We acknowledge many useful fruitful discussions with Takako Nemoto on realizability
models for Martin-L{\"o}f's type theory during her visits to our department.
We also thank Laura Crosilla,  Giovanni Sambin and Thomas Streicher for
other interesting fruitful discussions on topics of this paper.
We are grateful to Ferruccio Guidi for its constant help with typesetting.

\bibliography{bibliopsp}

\section{Appendix: The typed calculus \mtt\ }
\label{mttsyn}
 We present here the inference rules to build types
 in \mtt. 
The inference rules involve judgements  written in the style of Martin-L{\"o}f's type theory
\cite{ML84,PMTT} that may be of the form:\\

{\small  $$A \ type \ [\Gamma] \hspace{.5cm} A=B\ type\ [\Gamma] 
\hspace{.5cm} a \in A\ 
 [\Gamma] \hspace{.5cm} a=b \in A\ [\Gamma] $$}
where types include collections, sets, propositions and small propositions,
namely 
{\small $$type \in \{ col, set,prop,prop_s\, \}$$}

For easiness, the piece of context common to all judgements
involved in a rule
is omitted and
 typed variables appearing in a context
are meant to be added to the implicit
context as the last one.

\noindent
 Note that to write
 the elimination constructors of our types
 we adopt the higher-order syntax  in
\cite{PMTT}~\footnote{For example, note that the elimination constructor of
disjunction ${\bf \it El}_{\vee}(w,a_{B},a_{C})$  binds the open terms
$a_{B}(x)\in A\ [x\in B]$ and 
  $      a_{C}(y)\in A \ [y\in C]$.
Indeed, given that they are needed in the disjunction conversion rules,
it follows that these open terms
 must be  encoded into 
the elimination constructor.
To encode them we use the higher-order syntax as  in \cite{PMTT} (see also \cite{gui}).
According to this syntax
the open term $a_B(x)\in A\ [x\in B]$ yields to  $(x\in B)\,  a_B(x)$
of higher type $(x\in B)\, A$. Then, by $\eta$-conversion among higher types,
 it follows that  $(x\in B)\,  a_B(x)$ is equal to $a_B$. Hence, we often simply write the short expression
$a_B$ to recall the open term  where it comes from.}.

\noindent
We also have a form of judgement to build contexts:
{\small $$\Gamma\ cont$$}

\noindent
whose rules are the following

{\small
$$ \begin{array}{ll}
\emptyset\ cont\qquad  &\qquad \mbox{F-c}\displaystyle{
 \frac{\displaystyle\  A\ type\  [\Gamma]\ }
{\displaystyle \ \Gamma, x\in A \ cont\ }}\  (x \in A \not\in \Gamma)
\end{array}
$$}

\noindent 
Then, the first rule to build elements of type is 
the assumption of variables:

{\small
$$
\mbox{ var) }
\displaystyle{\frac{\displaystyle\  \Gamma, x\in A, \Delta \hspace{.3cm}\ cont\ }{\displaystyle\  x\in A\ [ \Gamma,x\in A , \Delta]\ }}
$$
}

\noindent
Among types there are the following embeddings: sets are collections and propositions are collections
\\
 
\noindent
{\small
$\begin{array}{l}
\mbox{\bf set-into-col) }\ \ \displaystyle{ \frac
       {\displaystyle\  A
        \hspace{.1cm} set\ }
{ \displaystyle\  A
        \hspace{.1cm} col\ }}
\end{array}
\qquad\qquad
\begin{array}{l}
    \mbox{\bf prop-into-col) }\ \ 
\displaystyle{ \frac
       {\displaystyle\  A\ prop\  }
      { \displaystyle\  A\ col  \ }}
\end{array}
$}
\\





\noindent
Moreover, collections are closed under strong indexed sums:\\

\noindent
{\small
$\begin{array}{l}
      \mbox{ \bf Strong Indexed Sum } \\[10pt]
      \mbox{F-}\Sigma ) \ \
\displaystyle{ \frac{\displaystyle   C(x)
         \hspace{.1cm} \ col \ [x\in B]}
         {\displaystyle \Sigma_{x\in B} C(x)\hspace{.1cm} col }}
         \qquad
      \mbox{I-}\Sigma )\ \
\displaystyle{ \frac
         {\displaystyle b\in  B \hspace{.3cm} c\in C(b)\qquad C(x)\ col\ [x\in B]}
         {\displaystyle \langle b,c\rangle\in  \Sigma_{x\in B} C(x)}}
\\[15pt]
\mbox{E-}\Sigma )\ \
\displaystyle{ \frac
         {\displaystyle \begin{array}{l}
M(z)\ col \ [ z\in \Sigma_{x\in B} C(x)]\\
d\in  \Sigma_{x\in B} C(x) \hspace{.3cm}
 m(x,y)\in M(\langle x, y \rangle)\ [x\in B,
         y\in C(x)]
\end{array}}
      {\displaystyle {\it  El}_{\Sigma}(d,m)\in  M(d)}}
      \\[15pt]
\mbox{C-}\Sigma ) \ \
\displaystyle{ \frac
         {\displaystyle \begin{array}{l}
M(z)\ col \ [ z\in \Sigma_{x\in B} C(x)]\\
b\in B\ \ \  c\in C(b)\hspace{.3cm} m(x,y)\in M(\langle x,y \rangle)\ [x\in B,
         y\in C(x)]\end{array}}
      {\displaystyle {\bf \it El}_{\Sigma}(\, \langle b, c\rangle
,m\, )=m(b,c)\in M(\langle b,c\rangle)} }
      \end{array}$
\\
\\

}

\noindent
Sets are generated  as follows:
\\
\\

\noindent
{\small
$\begin{array}{l}
      \mbox{ \bf Empty set} \\
      \mbox{ F-Em)}\ \ \mathsf{N_0} \hspace{.1cm} set \qquad
\mbox{ E-Em)}\ \
\displaystyle{ \frac
         {\displaystyle a\in  \mathsf{N_0} \hspace{.3cm} A(x)\hspace{.1cm} col \
[x\in \mathsf{N_0}] }
         {\displaystyle \orig{emp_{o}}(a)\in A(a)}}
      \end{array}$
\\
\\

\noindent
$\begin{array}{l}
\mbox {\bf Singleton}\\
 \mbox{\small S)}\ \orig{\mathsf{N_1}} \hspace{.1cm} set
 \qquad
 \mbox{\small I-S)}\ 
 \orig{\star} \in\orig{\mathsf{N_1}}
 \qquad 
 \mbox{\small E-S)}\
\displaystyle{  \frac
    {t\in \orig{\mathsf{N_1}}\quad  M(z)\ col\ 
[z\in \mathsf{\mathsf{N_1}}] \quad c\in M(\star)}
    {{ \it El}_{ \orig{\mathsf{N_1}} }(t,c)\in  M(t)}}
\qquad 
 \mbox{\small C-S)}\
\displaystyle{  \frac
    { M(z)\ col\ 
[z\in \mathsf{\mathsf{N_1}}] \quad c\in M(\star)}
    {{ \it El}_{ \orig{\mathsf{N_1}} }(\star,c)=c\in  M(\star)}}
 \end{array}$
\\
\\

\noindent
$\begin{array}{l}
      \mbox{ \bf Strong Indexed Sum set } \\[10pt]
      \mbox{F-}\Sigma_s ) \ \
\displaystyle{ \frac{\displaystyle   C(x)
         \hspace{.1cm} set\ \ [x\in B]\qquad B\ set}
         {\displaystyle \Sigma_{x\in B} C(x)\hspace{.1cm} set }}
      \end{array}$
\\
\\

\noindent
      $\begin{array}{l}
\mbox{\bf List set} \\[10pt]
      \mbox{F-list)}\
\displaystyle{ \frac
         { \displaystyle C \hspace{.1cm} set}
         {\displaystyle List(C) \hspace{.1cm} set }}
      \qquad
       \mbox{${\rm I}_{1}$-list)}\ \
\displaystyle{ \frac
         {\displaystyle \quad List(C) \hspace{.1cm} set\  }
         {\epsilon \in List(C)}}
      \qquad
      \mbox{${\rm I}_{2}$-list)}\ \
\displaystyle{ \frac
         {\displaystyle s\in List(C) \hspace{.3cm} c\in  C}
         {\displaystyle \orig{cons}(s,c)\in List(C)}}
      \end{array}$
\\
\\

\noindent
$\begin{array}{l}
\mbox{E-list)}\ \
\displaystyle{ \frac
         {\displaystyle \begin{array}{l}
L(z)\ col\  [z\in List(C)]
\hspace{.3cm}s\in List(C) \hspace{.3cm}\qquad  a\in L(\epsilon)\hspace{.3cm}\\
    l(x,y,z)\in
         L(\orig{cons}(x,y))  \ [x\in List(C),y\in C, z\in L(x)]
\end{array}}
      {\displaystyle {\bf \it El}_{List}(s,a, l)\in  L(s)}}
      \end{array}$
\\
\\

\noindent
$\begin{array}{l}
\mbox{${\rm C}_{1}$-list)}\ \
\displaystyle{  \frac
         {\displaystyle\begin{array}{l}
L(z)\ col\  [z\in List(C)] \hspace{.3cm}\qquad
 a\in L(\epsilon)\hspace{.3cm}\\
     l(x,y,z)\in
         L(\orig{cons}(x,y))  \ [x\in List(C),y\in C, z\in L(x)]
\end{array}}
      {\displaystyle {\bf \it El}_{List}( \epsilon, a,l)=a\in  L(\epsilon)}}
      \\[15pt]
\mbox{ ${\rm C}_{2}$-list)}\ \
\displaystyle{ \frac
         {\displaystyle\begin{array}{l}
L(z)\ col\  [z\in List(C)]
\hspace{.3cm} s\in List(C) \hspace{.3cm}c\in C \hspace{.3cm} a\in
L(\epsilon)\hspace{.3cm} \\
l(x,y,z)\in
         L(\orig{cons}(x,y))  \ [x\in List(C),y\in C, z\in L(x)]
\end{array}}
      {\displaystyle {\bf \it El}_{List}(\orig{cons}(s,c),a,l)=l(s,
c,{\bf \it El}_{List}(s,a,
      l))\in  L(\orig{cons}(s,c))}}
\end{array}$
\\
\\

\noindent
      $\begin{array}{l}
      \mbox{\bf Disjoint Sum set } \\[10pt]
      \mbox{F-+)} \ \
\displaystyle{  \frac
         { \displaystyle B \hspace{.1cm} set \hspace{.3cm}C
         \hspace{.1cm} set}
         {\displaystyle B+ C \hspace{.1cm} set }}
      \qquad
      \mbox{${\rm I}_{1}$-}+ ) \ \
\displaystyle{  \frac
         {\displaystyle b\in  B\qquad B\ set \qquad C\  set }
         {\displaystyle\orig{inl}(b)\in B+ C}}
      \qquad
\mbox{${\rm I}_{2}$-}+ ) \ \
\displaystyle{  \frac
         {\displaystyle c\in  C \qquad B\ set\qquad C\ set}
         {\displaystyle  \orig{inr}(c)\in B+ C}}
      \\[15pt]
\mbox{E-} + ) \ \
\displaystyle{  \frac
         {\displaystyle\begin{array}{l}
A(z)\ col\  [z\in B+C]\\
      w\in B+ C \hspace{.3cm}
      a_{B}(x)\in A(\orig{inl}(x))\ [x\in B]
      \hspace{.3cm}\
      a_{C}(y)\in A(\orig{inr}(y))\ [y\in C]
\end{array}}
         {\displaystyle { \it El}_{+}(w,a_{B},a_{C})\in A(w)}}
      \\[15pt]
      \mbox{${\rm C}_{1}$-}+ ) \ \
\displaystyle{ \frac
       {\displaystyle \begin{array}{l}
A(z)\ col\  [z\in B+C]\\
b\in B \hspace{.3cm}  a_{B}(x)\in A(\orig{inl}(x))\ [x\in B]
      \hspace{.3cm}\
      a_{C}(y)\in A(\orig{inr}(y))\ [y\in C]\end{array}}
      {\displaystyle { \it El}_{+}(\orig{inl}(b),a_{B},a_{C})=a_{B}(b)
\in A(\orig{inl}(c))}}
\\[15pt]
      \mbox{${\rm C}_{2}$-}+ ) \ \
\displaystyle{ \frac
       {\displaystyle\begin{array}{l}
A(z)\ col\  [z\in B+C]\\
      c\in C\hspace{.3cm} a_{B}(x)\in A(\orig{inl}(x))\ [x\in B]
      \hspace{.3cm}\
      a_{C}(y)\in A(\orig{inr}(y))\ [y\in C]\end{array}}
      {\displaystyle {\bf \it El}_{+}(\orig{inr}(c),
a_{B},a_{C})=a_{C}(c)\in A(\orig{inr}(c))}}
      \end{array}$
\\
\\

\noindent
      $\begin{array}{l}
\mbox{\bf Dependent Product set}
\\[10pt]
\mbox{F-$\Pi$)}\ \
\displaystyle{\frac{ \displaystyle  C(x)
        \hspace{.1cm} set\ [x\in B] \qquad B\ set }
{\displaystyle \Pi_{x\in B} C(x)\hspace{.1cm} set }}
 \qquad
 \mbox{ I-$\Pi$)}\ \
 \displaystyle{\frac{ \displaystyle c(x)\in  C(x)\ [x\in B] \qquad C(x)
        \hspace{.1cm} set\ [x\in B] \qquad B\ set }
 { \displaystyle \lambda x^{B}.c(x)\in \Pi_{x\in B} C(x)}}
       \\[15pt]
 \mbox{  E-$\Pi$)}\ \
 \displaystyle{\frac{ \displaystyle
 b\in B \hspace{.3cm} f\in \Pi_{x\in B} C(x) }
 {\displaystyle \mathsf{Ap}(f,b)\in C(b) }}
\\[15pt]
 \mbox{$\mathbf{\beta}$C-$\Pi$)}\ \
 \displaystyle{\frac{ \displaystyle b\in B \hspace{.3cm} 
c(x)\in  C(x)\ [x\in B]\qquad C(x)
        \hspace{.1cm} set\ [x\in B] \qquad B\ set}
 {\displaystyle \mathsf{Ap}(\lambda x^{B}.c(x),b)=c(b)\in   C(b) }}
      \end{array}
$}
\\
\\



\noindent
Propositions are generated as follows:
\\

\noindent {\small
      $\begin{array}{l}
       {\bf Falsum} \\
      \mbox{ F-Fs)}\ \ \bot \hspace{.1cm} prop \qquad
\mbox{E-Fs)}\ \
\displaystyle{ \frac
         {\displaystyle a\in  \bot \hspace{.3cm} \phi\hspace{.1cm} prop }
         { \displaystyle \orig{r_{o}}(a)\in \phi}}
      \end{array}$
\\
\\

\noindent
      $\begin{array}{l}
\mbox{\bf Disjunction}
      \\[10pt]
      \mbox{F-}\vee )\ \
\displaystyle{ \frac
         { \displaystyle \psi \hspace{.1cm} prop \hspace{.3cm}\alpha
         \hspace{.1cm} prop}
         {\displaystyle \psi\vee \alpha \hspace{.1cm} prop }}
      \quad
      \mbox{${\rm I}_{1}$-}\vee) \ \
\displaystyle{ \frac
         {\displaystyle b\in  \psi\qquad \psi\ prop\qquad \alpha \ prop}
         {\displaystyle \mathsf{inl}_{\vee}(b)\in \psi\vee \alpha}}
      \quad
\mbox{${\rm I}_{2}$-}\vee ) \ \
\displaystyle{ \frac
         {\displaystyle c\in  \alpha \qquad \psi\ prop\qquad \alpha\ prop}
         {\displaystyle \orig{inr}_{\vee}(c)\in \psi\vee \alpha}}
\end{array}$\\[15pt]
$\begin{array}{l}
\mbox{E-} \vee ) \ \
\displaystyle{ \frac
         {\displaystyle \begin{array}{l}
\phi\  prop\ \\
      w\in \psi\vee \alpha \hspace{.3cm}
      a_{\psi}(x)\in \phi\ [x\in \psi]
      \hspace{.3cm}\
      a_{\alpha}(y)\in \phi \ [y\in \alpha]\end{array}}
         {\displaystyle {\it El}_{\vee}(w,a_{\psi},a_{\alpha})\in \phi}}
 \\[15pt]
      \mbox{${\rm C}_{1}$-}\vee ) \ \
\displaystyle{ \frac
       {\displaystyle \begin{array}{l}
\phi\  prop\qquad \psi\ prop\qquad \alpha\ prop\\
b\in \psi \hspace{.3cm}  a_{\psi}(x)\in \phi\ [x\in \psi]
      \hspace{.3cm}\
      a_{\alpha}(y)\in \phi \ [y\in \alpha]
\end{array}}
      { \displaystyle {\bf \it El}_{\vee}
(\orig{inl}_{\vee}(b),a_{\psi},a_{\alpha})=a_{\psi}(b)\in \phi}}
\\[15pt]
      \mbox{${\rm C}_{2}$-}\vee ) \ \
\displaystyle{ \frac
       {\displaystyle \begin{array}{l}
\phi\  prop\   \qquad \psi\ prop\qquad \alpha\ prop  \\
c\in \alpha\hspace{.3cm} a_{\psi}(x)\in \phi\ [x\in \psi]
      \hspace{.3cm}\
      a_{\alpha}(y)\in \phi \ [y\in \alpha]\end{array}}
      {\displaystyle {\bf \it El}_{\vee}
(\orig{inr}_{\vee}(c),a_{\psi},a_{\alpha})=a_{\alpha}(c)\in \phi}}
\end{array}
y$
\\
\\

\noindent
$\begin{array}{l}
\mbox{\bf Conjunction}
\\[10pt]
      \mbox{F-}\wedge) \ \
\displaystyle{ \frac
         { \displaystyle  \psi   \hspace{.1cm} prop\qquad  \alpha
         \hspace{.1cm} prop }
         {\displaystyle  \psi\wedge \alpha\hspace{.1cm} prop }}
         \qquad
      \mbox{I-}\wedge )\ \
\displaystyle{ \frac
         {\displaystyle b\in  \psi \hspace{.3cm} c\in \alpha\qquad \psi\ prop\qquad \alpha\ prop}
         {\displaystyle \langle b,_{\wedge} c \rangle\in \psi\wedge \alpha  }}
\\[15pt]
\mbox{$\mathbf{\textrm{ E}_{1}}$-$\wedge$) }\ \
\displaystyle{ \frac
       {\displaystyle d\in  \psi \wedge \alpha}
{\displaystyle \pi_{1}^{\psi}(d)\in  \psi}}
\qquad \qquad
\mbox{$\mathbf{\textrm{ E}_{2}}$-$\wedge$) }\ \
\displaystyle{ \frac
       {\displaystyle d\in  \psi \wedge \alpha}
{\displaystyle \pi_{2}^{\alpha}(d)\in  \alpha}}
\\[15pt]
\mbox{$\beta_{1}$ C-$ \wedge$) }\ \
\displaystyle{ \frac
       {\displaystyle b\in  \psi \hspace{.3cm} c\in \alpha\qquad \psi\ prop\qquad \alpha\ prop}
{\pi_{1}^{\psi}(\langle b,_{\wedge} c \rangle)=b\in \psi}}
\qquad\qquad
\mbox{$\beta_{2}$ C-$\wedge$) }\displaystyle{ \frac
       {\displaystyle b\in  \psi \hspace{.3cm} c\in \alpha\qquad \psi\ prop\qquad \alpha\ prop}
{\displaystyle \pi_{2}^{\alpha}(\langle b,_{\wedge} c \rangle)=c\in \alpha}}
\end{array}
$
\\
\\

\noindent
$
\begin{array}{l}
\mbox{\bf Implication}
\\[10pt]
\mbox{F-$\rightarrow$) }\ \
\displaystyle{ \frac
       {\displaystyle \psi
        \hspace{.1cm} prop \qquad \alpha
        \hspace{.1cm} prop  }
{\psi \rightarrow \alpha \hspace{.1cm} prop }}
\\[15pt]
\mbox{I-$\rightarrow$) }\ \
\displaystyle{ \frac
{\displaystyle
      c(x)\in  \alpha\ [x\in \psi] \qquad \psi \hspace{.1cm} prop\qquad \alpha\ prop\  }
{\displaystyle \lambda_\rightarrow x^{\psi}.c(x)\in \psi \rightarrow \alpha}}
        \qquad
\mbox{ E-$\rightarrow$) }\ \
\displaystyle{ \frac
{\displaystyle b\in \psi \hspace{.3cm} f\in \psi\rightarrow \alpha }
{\displaystyle \mathsf{Ap}_\rightarrow(f,b)\in \alpha }}
       \\[15pt]
\mbox{$\mathbf{\beta}$C-$\rightarrow$) }\ \
\displaystyle{ \frac
{\displaystyle b\in \psi \hspace{.3cm} c(x)\in  \alpha
 \
[x\in \psi]\qquad \psi \hspace{.1cm} prop\qquad \alpha\ prop}
{ \displaystyle \mathsf{Ap}_\rightarrow
(\lambda_\rightarrow x^{\psi}.c(x),b)=c(b)\in   \alpha }}
      \end{array}$
\\
\\

\noindent
      $\begin{array}{l}
      \mbox{\bf Existential quantification}
      \\[10pt]
      \mbox{F-}\exists ) \ \
\displaystyle{ \frac
         { \displaystyle   \alpha(x)
         \hspace{.1cm} prop \ \ [x\in \psi]}
         {\displaystyle  \exists_{x\in B} \alpha(x)\hspace{.1cm} prop }}
         \qquad
      \mbox{I-}\exists )\ \
\displaystyle{ \frac
         {\displaystyle b\in  B \hspace{.3cm} c\in \alpha(b)\qquad \alpha(x)\ prop\ [x\in B]}
         {\displaystyle \langle b,_{\exists} c \rangle\in  \exists_{x\in B} \alpha(x)}}
      \\[15pt]
\mbox{E-}\exists )\ \
\displaystyle{  \frac
         {\displaystyle \begin{array}{l}
\phi\ prop\ \\
d\in  \exists_{x\in B} \alpha(x) \hspace{.3cm} m(x,y)\in \phi\ [x\in B,
         y\in \alpha(x)]
\end{array}}
      {\displaystyle { \it El}_{\exists}(d,m)\in  \phi}}
      \\[15pt]
\mbox{C-}\exists ) \ \
\displaystyle{ \frac
         {\displaystyle \begin{array}{l}
\phi\ prop\qquad \alpha(x)\ prop\ [x\in B]\\
b\in B\ \ \  c\in \alpha(b)\hspace{.3cm} m(x,y)\in \phi\ [x\in B,
         y\in \alpha(x)]
\end{array}}
      {\displaystyle { \it El}_{\exists}(\langle b,_{\exists}c \rangle,m)=m(b,c)\in M} }
      \end{array}$
\\
\\

\noindent
      $\begin{array}{ll}
\mbox{\bf Universal quantification}
\\[10pt]
\mbox{F-$\forall$) }\ \
\displaystyle{ \frac
{\displaystyle \alpha(x)\hspace{.1cm} prop\ [x\in B]  }
{\displaystyle \forall_{x\in B} \alpha(x)\hspace{.1cm} prop }}
&
\mbox{I-$\forall$) }\ \
\displaystyle{ \frac
{\displaystyle c(x)\in  \alpha(x)\ [x\in B]\qquad  \alpha(x)\hspace{.1cm} prop\ [x\in B]  }
{\lambda_\forall x^{B}.c(x)\in \forall_{x\in B} \alpha(x)}}
        \\[15pt]
\mbox{E-$\forall$) }\ \
\displaystyle{ \frac
{\displaystyle b\in B \hspace{.3cm} f\in \forall_{x\in B} \alpha(x) }
{\displaystyle \mathsf{Ap}_{\forall}(f,b)\in \alpha(b) }}
       &
\mbox{$\mathbf{\beta}$C-$\forall$) }\ \
\displaystyle{ \frac
{\displaystyle b\in B \hspace{.3cm} c(x)\in  \alpha(x)\ [x\in B]\qquad
 \alpha(x)\ prop\  [x\in B]
}
{\displaystyle \mathsf{Ap}_{\forall}(\lambda_\forall
      x^{B}.c(x),b)=c(b)\in   \alpha(b) }}
      \end{array}$
\\
\\

\noindent
      $\begin{array}{l}
\mbox{\bf Propositional Equality}
      \\[15pt]
      \mbox{F-Id)}\ \
\displaystyle{ \frac
{\displaystyle  A \hspace{.1cm} col \hspace{.3cm}  a\in A \hspace{.3cm} b\in A}
         {\displaystyle \orig{Id}(A, a, b)\hspace{.1cm} prop }}
      \qquad
      \mbox{I-Id)}\ \
\displaystyle{ \frac
{\displaystyle  a \in A}
         {\displaystyle \orig{id_{A}}(a)\in \orig{Id}(A, a, a)}}
\\[15pt]
      \mbox{E-Id)}\ \
      \displaystyle{\frac
       {\displaystyle \begin{array}{l}\alpha(x,y)\ prop \ [x:A,y\in A]\\
a\in A\quad b\in A\quad p\in \orig{Id}(A, a, b)
\quad c(x)\in \alpha(x,x)
\ [x\in A]\end{array}}
      {\displaystyle {\it El}_{\mathsf{Id}}(p,(x)c(x))\in \alpha(a,b) }}
      \\[15pt]
       \mbox{C-Id)} \ \
\displaystyle{\frac
       {\displaystyle \begin{array}{l}\alpha(x,y)\ prop \ [x:A,y\in A]\\
a\in A
\quad c(x)\in \alpha(x,x)
\ [x\in A]\end{array}}
      {\displaystyle {\bf \it El}_{\mathsf{Id}}(\mathsf{id}_A(a),(x)c(x))=c(a)
\in \alpha(a,a) }}
      \end{array}$
\\
\\
}

\noindent
Then, we also have  the  collection of small propositions:
\\

\noindent
{\small
$ 
\begin{array}{l}
 \mbox{\bf Collection of small propositions}\\[5pt]
\mbox{F-Pr) } \ \displaystyle{\mathsf{prop_s}\  col }\qquad 
\end{array}
$}
\\
\\

\noindent
The collection of small propositions containes codes
of  small propositions
\\

\noindent
$
\mbox{T-Pr) }\displaystyle{ \frac
       {\displaystyle p\in
        \ \mathsf{prop_s} }
{ T(p) \ prop_s }}
$
\\
\\

\noindent
which are generated as follows:
\\

{\small
$\begin{array}{l} 
\mbox{Pr$_1$) }\widehat{\bot} \in \mathsf{prop_s}\qquad
\mbox{Pr$_2$) }\displaystyle{ \frac
         { \displaystyle  p\in  \mathsf{prop_s} \hspace{.3cm} q \in
         \ \mathsf{prop_s}}
         {\displaystyle p \widehat{\vee} q\in \mathsf{prop_s} }}\\[15pt]
\mbox{Pr$_3$) }\displaystyle{ \frac
       {\displaystyle p\in
        \ \mathsf{prop_s} \qquad q\in
        \ \mathsf{prop_s} }
{p\widehat{\rightarrow} q \in \mathsf{prop_s} }}\qquad
\mbox{Pr$_4$) }\displaystyle{ \frac
         { \displaystyle  p   \in \mathsf{prop_s}\qquad  q\in
           \mathsf{prop_s} }
         {\displaystyle  p\widehat{\wedge} q\in  \mathsf{prop_s} }}
\\[15pt]
\mbox{Pr$_5$) }\displaystyle{ \frac
         { \displaystyle   p(x)
         \ \mathsf{prop_s}\ \ [x\in B]\qquad B\  set}
         {\displaystyle  \widehat{\exists_{x\in B}} p(x)\in
\mathsf{prop_s}}}\qquad
\mbox{Pr$_6$) }\displaystyle{ \frac
{\displaystyle p(x) \in \mathsf{prop_s}\ [x\in B]\qquad B\ set  }
{\displaystyle \widehat{\forall_{x\in B}} p(x) \in \mathsf{prop_s} }}
\\[15pt]
\mbox{Pr$_7$) }\displaystyle{ \frac
{\displaystyle  A \hspace{.1cm} set \hspace{.3cm}  a\in A \hspace{.3cm} b\in A}
         {\displaystyle \widehat{\orig{Id}}(A, a, b) \in \mathsf{prop_s} }}
\end{array}
$}
\\
\\

\noindent
with the following definitional equalities:
\\

{\small
$\begin{array}{l} 
\mbox{eq-Pr$_1$) }T(\widehat{\bot})= \bot \qquad
\mbox{eq-Pr$_2$) }\displaystyle{ \frac
         { \displaystyle  p\in  \mathsf{prop_s} \hspace{.3cm} q \in
         \ \mathsf{prop_s}}
         {\displaystyle T( p \widehat{\vee} q)= T(p) \vee T(q)}}\\[15pt]
\mbox{eq-Pr$_3$) }\displaystyle{ \frac
       {\displaystyle p\in
        \ \mathsf{prop_s} \qquad q\in
        \ \mathsf{prop_s} }
{T(p\widehat{\rightarrow} q)= T(p)\rightarrow T(q)  }}\qquad
\mbox{eq-Pr$_4$) }\displaystyle{ \frac
         { \displaystyle  p   \in \mathsf{prop_s}\qquad  q\in
           \mathsf{prop_s} }
         {\displaystyle  T(p\widehat{\wedge} q)=T(p) \wedge T(q)   }}
\\[15pt]
\mbox{eq-Pr$_5$) }\displaystyle{ \frac
         { \displaystyle   p(x)
         \ \mathsf{prop_s}\ \ [x\in B]\qquad B\  set}
         {\displaystyle T( \widehat{\exists_{x\in B}} p(x))=
\exists_{x\in B}\ T(p(x)) }}\qquad
\mbox{eq-Pr$_6$) }\displaystyle{ \frac
{\displaystyle p(x) \in \mathsf{prop_s}\ [x\in B]\qquad B\ set  }
{\displaystyle T(\widehat{\forall_{x\in B}} p(x))= \forall_{x\in B} T(p(x)) }}
\\[15pt]
\mbox{eq-Pr$_7$) }\displaystyle{ \frac
{\displaystyle  A \hspace{.1cm} set \hspace{.3cm}  a\in A \hspace{.3cm} b\in A}
         {\displaystyle T(\, \widehat{\orig{Id}}(A, a, b)\,)=
\orig{Id}(A,a,b)    }}
\end{array}
$}
\\
\\

 \noindent
Then, we also have function collections from a set  toward
the  collection of small propositions:
\\

 {\small
       $\begin{array}{ll}
 \mbox{\bf Function collection to $\mathsf{prop_s}$ }
 \\[10pt]
 \mbox{ F-Fun)}\ 
 \displaystyle{\frac{ \displaystyle B\ set\  }
 {\displaystyle B\rightarrow \mathsf{prop_s} \hspace{.1cm} col }}
 &
 \mbox{ I-Fun)}\ \
 \displaystyle{\frac{ \displaystyle c(x)\in \mathsf{prop_s}\ [x\in B]
     \qquad B\ set }
 { \displaystyle \lambda x^{B}.c(x)\in  B\rightarrow \mathsf{prop_s}}}
       \\[15pt]
 \mbox{  E-Fun)}\ \
 \displaystyle{\frac{ \displaystyle
 b\in B \hspace{.3cm} f\in  B\rightarrow \mathsf{prop_s} }
 {\displaystyle \mathsf{Ap}(f, b)\in \mathsf{prop_s}  }}
 &
 \mbox{$\mathbf{\beta}$C-Fun)}\ \
 \displaystyle{\frac{ \displaystyle b\in B \hspace{.3cm} c(x)\in\mathsf{prop_s}
 \  [x\in B]  \qquad B\ set }
 {\displaystyle  \mathsf{Ap}(\lambda x^{B}.c(x),b)=c(b)\in  \mathsf{prop_s} }}
 \end{array}
 $}
 \\
 \\

\noindent
And we add rules saying that a small proposition is a proposition and
 that a small proposition is  a set:
\\

\noindent
{\small

$\begin{array}{l}
    \mbox{\bf prop$_s$-into-prop)}\ \displaystyle{ \frac
       {\displaystyle \phi\ prop_s }
      { \displaystyle \phi\ prop}}
\end{array}\qquad
\qquad
\begin{array}{l}
    \mbox{\bf prop$_s$-into-set)}\
\displaystyle{ \frac
       {\displaystyle \phi\ prop_s }
      { \displaystyle \phi\ set}}
\end{array}\qquad
$}
\\
\\


\noindent
Equality rules include those saying that  type equality is an equivalence relation and  substitution of equal terms in a type:
\\

\noindent
{\small
$\begin{array}{l}
    \mbox{ref)}\,
\displaystyle{ \frac
       {\displaystyle A\ type }
      { \displaystyle A=A\ type}}
\end{array}
\qquad
\begin{array}{l}
    \mbox{sym)}\,
\displaystyle{ \frac
       {\displaystyle A=B\ type }
      { \displaystyle B=A\ type}}
\end{array}
\qquad
\begin{array}{l}
    \mbox{tra)}\,
\displaystyle{ \frac
       {\displaystyle A=B\ type \quad \ B=C\ type }
      { \displaystyle A=C\ type}}
\end{array}
\\[15pt]
\begin{array}{l}
    \mbox{subT)}\,\displaystyle{ \frac
       {\displaystyle\begin{array}{l}
C(x_1,\dots,x_n)\ type\ 
 [\, x_1\in A_1,\,  \dots,\,  x_n\in A_n(x_1,\dots,x_{n-1})\, ]   \\[5pt]
a_1=b_1\in A_1\ \dots \ a_n=b_n\in A_n(a_1,\dots,a_{n-1})\end{array}}
         {\displaystyle 
 C(a_1,\dots,a_{n})=C( b_1,\dots, b_n) \ type}}
   \end{array} 
$}
\\
\\

\noindent
where {\small $type \in \{ col, set,prop, prop_s\, \}$} with 
the same choice both in the premise and in the conclusion.
\\

\noindent
For terms into sets we add the following equality rules:\\
\\

\noindent
{\small
$\begin{array}{l}
    \mbox{ref)}\,
\displaystyle{ \frac
       {\displaystyle \ a\in A\ }
      { \displaystyle\ a=a\in A\ }}
\end{array}
\qquad
\begin{array}{l}
    \mbox{sym)}\,
\displaystyle{ \frac
       {\displaystyle\  a=b\in A\ }
      { \displaystyle\  b=a\in A\ }}
\end{array}
\qquad
\begin{array}{l}
    \mbox{tra)}\,
\displaystyle{ \frac
       {\displaystyle \ a=b\in A \qquad b=c\in A\ }
      { \displaystyle\  a=c\in A\ }}
\end{array}
$
\\
\\

\noindent
$\begin{array}{l}
      \mbox{sub)} \ \
\displaystyle{ \frac
         { \displaystyle 
\begin{array}{l}
 c(x_1,\dots, x_n)\in C(x_1,\dots,x_n)\ \
 [\, x_1\in A_1,\,  \dots,\,  x_n\in A_n(x_1,\dots,x_{n-1})\, ]   \\[5pt]
a_1=b_1\in A_1\ \dots \ a_n=b_n\in A_n(a_1,\dots,a_{n-1})
\end{array}}
         {\displaystyle c(a_1,\dots,a_n)=c(b_1,\dots, b_n)\in
 C(a_1,\dots,a_{n})  }}
      \\[15pt]
\mbox{conv)} \ \
\displaystyle{\frac{\displaystyle a\in A\ \qquad A=B\ type }
{\displaystyle a\in B}}
\qquad   
\mbox{conv-eq)} \ \
\displaystyle{\frac{\displaystyle a=b\in A\ \qquad A=B\ type }
{\displaystyle a=b\in B\ }}
\end{array}
$}
\\
\\

\noindent
Now the equality rules about collections are the following:
\\
\\

\noindent
{\small
$\begin{array}{l}
      \mbox{ \bf Strong Indexed Sum-eq} \\[10pt]
      \mbox{eq-}\Sigma ) \ \
\displaystyle{ \frac{\displaystyle   C(x)=D(x)
         \hspace{.1cm} col \ \ [x\in B]\qquad B=E\ col}
         {\displaystyle \Sigma_{x\in B} C(x)= \Sigma_{x\in E} D(x)
\hspace{.1cm} col }}\end{array}
\qquad
\begin{array}{l}
\mbox{\bf Dependent Product-eq}
\\[10pt]
\mbox{eq-$\Pi$}\ \
\displaystyle{\frac{ \displaystyle  C(x)=D(x)
        \hspace{.1cm} col\ [x\in B] \qquad B=E\ col }
{\displaystyle \Pi_{x\in B} C(x)=\Pi_{x\in E} D(x)\hspace{.1cm} col }}
      \end{array}
$}
\\
\\

\noindent
Then, the equality about sets are the following:
\\
\\

\noindent
{\small
$\begin{array}{l}
\mbox{\bf Lists-eq} \\[10pt]
      \mbox{eq-list)}\
\displaystyle{ \frac
         { \displaystyle C=D \hspace{.1cm} set}
         {\displaystyle List(C)=List(D) \hspace{.1cm} set }}
\end{array}
\qquad
\begin{array}{l}
      \mbox{ \bf Strong Indexed Sum-eq} \\[10pt]
      \mbox{eq-}\Sigma ) \ \
\displaystyle{ \frac{\displaystyle   C(x)=D(x)
         \hspace{.1cm} set \ \ [x\in B]\qquad B=E\ set}
         {\displaystyle \Sigma_{x\in B} C(x)= \Sigma_{x\in E} D(x)
\hspace{.1cm} set }}\end{array}
$
\\
\\    

\noindent
      $\begin{array}{l}
      \mbox{\bf Disjoint Sum-eq} \\[10pt]
      \mbox{eq-+} )\  \
\displaystyle{  \frac
         { \displaystyle B=E \hspace{.1cm} set \hspace{.3cm}C=D
         \hspace{.1cm} set}
         {\displaystyle B+ C =E+D\hspace{.1cm} set }}
\end{array}\qquad\begin{array}{l}
\mbox{\bf Dependent Product-eq}
\\[10pt]
\mbox{eq-$\Pi$}\ \
\displaystyle{\frac{ \displaystyle  C(x)=D(x)
        \hspace{.1cm} set\ [x\in B] \qquad B=E\ set }
{\displaystyle \Pi_{x\in B} C(x)=\Pi_{x\in E} D(x)\hspace{.1cm} set }}
      \end{array}
$}
\\
\\

\noindent
Then, \mtt\ includes the following equalities rules about propositions:
\\
\\

\noindent
{\small
      $\begin{array}{l}
\mbox{\bf Disjunction-eq}
      \\[10pt]
      \mbox{eq-}\vee )\ \
\displaystyle{ \frac
         { \displaystyle \psi=\alpha \hspace{.1cm} prop \hspace{.3cm}\phi=\beta
         \hspace{.1cm} prop}
         {\displaystyle \psi\vee \phi= \alpha\vee \beta \hspace{.1cm} prop }}
\end{array}
\qquad \begin{array}{l}
\mbox{\bf Implication-eq}
\\[10pt]
\mbox{eq-$\rightarrow$}\ \
\displaystyle{ \frac
       {\displaystyle \psi=\alpha
        \hspace{.1cm} prop \qquad \phi=\beta
        \hspace{.1cm} prop  }
{\psi \rightarrow \phi= \alpha\rightarrow \beta \hspace{.1cm} prop }}
\end{array}
$
\\
\\

\noindent
$
 \begin{array}{l}
\mbox{\bf Conjunction-eq}
\\[10pt]
      \mbox{eq-}\wedge) \ \
\displaystyle{ \frac
         { \displaystyle  \psi=\alpha   \hspace{.1cm} prop\qquad  \phi=\beta
         \hspace{.1cm} prop }
         {\displaystyle  \psi\wedge \phi=\alpha\wedge \beta \hspace{.1cm} prop }}
      \end{array}
\qquad
\begin{array}{l}
\mbox{\bf Propositional equality-eq}
      \\[15pt]
      \mbox{eq-Id)}\ \
\displaystyle{ \frac
{\displaystyle  A=E \hspace{.1cm} col \hspace{.3cm}  a=e\in A \hspace{.3cm} 
b=c\in A}
         {\displaystyle \orig{Id}(A, a, b)= \orig{Id}(E, e, c)\hspace{.1cm} prop }}
      \end{array}$
\\
\\

\noindent
      $\begin{array}{l}
      \mbox{\bf Existential quantification-eq}
      \\[10pt]
      \mbox{eq-}\exists ) \ \
\displaystyle{ \frac
         { \displaystyle   \alpha(x)=\beta(x)
         \hspace{.1cm} prop \ \ [x\in B]\qquad B=E\ prop}
         {\displaystyle  \exists_{x\in B} \alpha(x)=\exists_{x\in E} \beta(x)
\hspace{.1cm} prop }}
\end{array}     \qquad
\begin{array}{l}
\mbox{\bf Universal quantification-eq}
\\[10pt]
\mbox{eq-$\forall$}\ \
\displaystyle{ \frac
{\displaystyle \alpha(x)=\beta(x)\hspace{.1cm} prop\ [x\in B]\qquad B=E \hspace{.1cm} prop  }
{\displaystyle \forall_{x\in B} \alpha(x)=\forall_{x\in E} \beta(x)\hspace{.1cm} prop }}
\end{array}
$}
\\
\\

\noindent
The equality of propositions is that of collections, that of small propositions coincides with that
of 
$\mathsf{prop_s}$ and is that of  propositions
and that of sets:
\\

\noindent
{\small
$\begin{array}{l}
    \mbox{\bf prop-into-col eq}\ 
\displaystyle{ \frac
       {\displaystyle \phi=\psi\ prop }
      { \displaystyle \phi=\psi\ col}}\end{array}
\qquad
\begin{array}{l}
    \mbox{\bf prop$_{s}$ eq1} \
\displaystyle{ \frac
       {\displaystyle \phi=\psi\ prop_s }
      { \displaystyle \phi=\psi\in \mathsf{prop_s}}}\end{array}
\qquad 
\begin{array}{l}
\mbox{\bf prop$_{s}$ eq2}\
\displaystyle{ \frac
       {\displaystyle \phi=\psi\in \mathsf{prop_s} }
      { \displaystyle \phi=\psi\ prop_s }}\end{array}
\\[15pt]
\begin{array}{l}
 \mbox{\bf prop$_{s}$-into-prop eq}\
\displaystyle{ \frac
       {\displaystyle \phi=\psi\ prop_s }
      { \displaystyle \phi=\psi\ prop}}\end{array}
\qquad \begin{array}{l}
    \mbox{\bf prop$_s$-into-set eq}\
\displaystyle{ \frac
       {\displaystyle \phi=\psi\ prop_s }
      { \displaystyle \phi=\psi\ set}}
\end{array}
$}
\\
\\

\end{document}